\nonstopmode
\documentclass[10pt, reqno]{amsart}
\usepackage{graphicx}
\usepackage{latexsym}
\usepackage{fancyhdr}
\usepackage{amsmath, amssymb, amsthm}
\usepackage[all]{xy}
\usepackage{pdflscape}
\usepackage{longtable}
\usepackage{rotating}
\usepackage{verbatim}
\usepackage
{hyperref}
\hypersetup{
    linkcolor = blue,
    citecolor = blue,
    urlcolor  = blue,
    colorlinks = true,
}

\usepackage{etoc}
\etocsettocstyle
  {\medskip}          
  {\bigskip}

\etocsetstyle{section}                                          
  {}
  {\leavevmode\leftskip 0pt\relax}            
  {\etocnumber \etocname          
   \dotfill\etocpage\par\smallskip}
  {}                                          

\etocsetstyle{subsection}
  {}
  {\leavevmode\leftskip 1.5em\relax}
  {\etocnumber\ \etocname\dotfill\etocpage\par\smallskip}
  {}

\etocsettocstyle
  {\section*{\contentsname}\small}           
  {\bigskip}                                 

\usepackage{cleveref}
\usepackage{subfigure}
\usepackage{mathrsfs}
\usepackage{mdwlist}
\usepackage{dsfont}

\usepackage{mathtools}
\mathtoolsset{showonlyrefs=true}

\usepackage{float}
\usepackage{color}
\usepackage{stmaryrd}
\usepackage{parskip}
\usepackage{orcidlink}

\usepackage{tkz-base}
\usepackage{pgfplots}
\pgfplotsset{compat=newest}

\usepackage{tkz-euclide}
\usepackage{etoolbox}

\definecolor{teal}{rgb}{0.0, 0.5, 0.5}

\newcounter{mnotecount}[section]

\newcommand{\rmnote}[1]{}

\allowdisplaybreaks

\DeclareFontFamily{U}{mathb}{\hyphenchar\font45}
\DeclareFontShape{U}{mathb}{m}{n}{
      <5> <6> <7> <8> <9> <10> gen * mathb
      <10.95> mathb10 <12> <14.4> <17.28> <20.74> <24.88> mathb12
      }{}
\DeclareSymbolFont{mathb}{U}{mathb}{m}{n}
\DeclareFontSubstitution{U}{mathb}{m}{n}

\theoremstyle{plain}
\newtheorem*{theorem*}{Theorem}
\newtheorem{theorem}{Theorem}[section]
\newtheorem*{lemma*}{Lemma}
\newtheorem{lemma}[theorem]{Lemma}
\newtheorem*{assumption*}{Assumption}

\newtheorem*{proposition*}{Proposition}
\newtheorem{proposition}[theorem]{Proposition}
\newtheorem*{corollary*}{Corollary}
\newtheorem{corollary}[theorem]{Corollary}
\newtheorem*{claim*}{Claim}

\newtheorem*{conjecture*}{Conjecture}

\newtheorem*{question*}{Question}

\newtheorem*{result*}{Result}

\theoremstyle{definition}
\newtheorem*{definition*}{Definition}
\newtheorem{definition}[theorem]{Definition}
\newtheorem*{example*}{Example}
\newtheorem{example}[theorem]{Example}

\newtheorem*{algorithm*}{Algorithm}
\newtheorem*{remark*}{Remark}
\newtheorem*{remarks*}{Remarks}
\newtheorem{remark}[theorem]{Remark}

\newtheorem*{convention*}{Convention}

\Crefname{l}{Lemma}{Lemmas}    
\Crefname{p}{Proposition}{Propositions}
\Crefname{t}{Theorem}{Theorems}
\Crefname{c}{Corollary}{Corollaries}
\Crefname{r}{Remark}{Remarks}
\Crefname{d}{Definition}{Definitions}
\Crefname{e}{Example}{Examples}
\Crefname{q}{Question}{Questions}

\numberwithin{equation}{section}

\def\al{\alpha}
\def\be{\beta}
\def\ga{\gamma}
\def\de{\delta}
\def\ep{\epsilon}

\def\et{\eta}
\def\th{\theta}

\def\ka{\kappa}
\def\la{\lambda}

\def\rh{\rho}
\def\vr{\varrho}
\def\si{\sigma}

\def\vh{\varphi}

\def\ps{\psi}
\def\om{\omega}

\def\Si{\Sigma}
\def\Up{\Upsilon}

\def\C{\mathbb{C}}

\def\N{\mathbb{N}}

\def\R{\mathbb{R}}

\def\cC{\mathcal{C}}

\def\cH{\mathcal{H}}

\def\cK{\mathcal{K}}
\def\cL{\mathcal{L}}

\def\sK{\mathscr{K}}

\def\sS{\mathscr{S}}

\def\p{\partial}

\def\<{\langle}
\def\>{\rangle}
\renewcommand{\o}{\circ}

\def\dist{\on{dist}}
\def\grad{\on{grad}}

\def\ol{\overline}

\def\nn{\mathbf{N}}
\def\dd{\mathbf{D}}

\def\Cnn{\mathbf{C}_\nn}

\def\dist{\on{dist}}
\def\diam{\on{diam}}

\def\ol{\overline}

\let\on=\operatorname

\DeclareMathOperator*{\esssup}{ess\,sup}

\newcommand{\sr}[1]%
{\ifmmode{}^\dagger\else${}^\dagger$\fi\ifvmode
\vbox to 0pt{\vss
 \hbox to 0pt{\hskip\hsize\hskip1em
 \vbox{\hsize3cm\raggedright\pretolerance10000
 \noindent #1\hfill}\hss}\vss}\else
 \vadjust{\vbox to0pt{\vss%
 \hbox to 0pt{\hskip\hsize\hskip1em%
 \vbox{\hsize3cm\raggedright\pretolerance10000%
 \noindent \small #1\hfill}\hss}\vss}}\fi%
}

\def\NN{\mathbb N}

\providecommand{\mapsfrom}{\kern.2em%
\setbox0=\hbox{$\leftarrow$\kern-.10em\rule[0.26mm]{0.1mm}{1.3mm}}\box0%
\kern.3em}

\title[Effective quasianalytic Remez inequalities]{Effective quasianalytic Remez inequalities\\ on tame sets}

\author[Armin Rainer]{Armin Rainer  \orcidlink{0000-0003-3825-3313}}

\address{Armin Rainer: Faculty of Mathematics and Geoinformation,
    Institute for Statistics and Mathematical Methods in Economics, E105-04,
TU Wien, Wiedner Hauptstraße 8, 1040 Vienna, Austria}
\email{armin.rainer@tuwien.ac.at}

\begin{document}

\begin{abstract}
    We establish a Remez inequality for functions in quasianalytic Denjoy--Carleman classes $\cC_M$ 
    on a large family of fat compact sets $K \subseteq \R^n$ with tame geometry, including all fat compact 
    sets definable in the o-minimal expansion of $\R$ by restricted $\cC_M$ functions.
    The inequality generalizes the classical Remez inequality for polynomials and 
    its quasianalytic versions on convex bodies, replacing the polynomial degree by the Bang degree,
    an integer associated with the weight $M$ and the size of the function. 
    The constants depend explicitly on the Bang degree and the geometry of $K$.
    We derive a range of quantitative consequences: 
    estimates for the volume of sublevel sets, comparison of $L^p$-norms, 
    effective inequalities of {\L}ojasiewicz, Harnack, and Markov type with explicit constants, 
    and decay estimates for oscillatory integrals. 
\end{abstract}

\thanks{This research was funded in whole or 
    in part by the Austrian Science Fund (FWF)  DOI 10.55776/PAT1381823.
For open access purposes, the authors have applied a CC BY public copyright license to any author-accepted manuscript version arising from this submission.}
\thanks{\orcidlinkc{0000-0003-3825-3313}}
\keywords{Quantitative Remez inequalities, quasianalytic functions, Denjoy--Carleman classes, tame sets, o-minimal structures,
{\L}ojasiewicz inequalities, Harnack inequalities, Markov inequalities, oscillatory integrals}
\subjclass[2020]{ 
  Primary  41A17,      
    26E10;  	
  Secondary  03C64,    
    32B20,  	
    26D10,   
    28A75,   
42B20}    

\date{\today}

\maketitle

\setcounter{tocdepth}{1}
\tableofcontents

\section{Introduction}

The Remez inequality is a fundamental tool in approximation theory.
The classical result of Remez \cite{Remes:1936wl} states that for a Lebesgue measurable subset $E$ of the unit interval $[0,1]$ with
Lebesgue measure $|E|>0$
we have
\[
  \|p\|_{[0,1]} \le T_d\Big(\frac{2-|E|}{|E|}\Big)\, \|p\|_E,
\]
for all polynomials $p \in \R[x]$ of degree at most $d$, 
where $T_d(x)= \cos (d \arccos x)$ is the $d$-th Chebyshev polynomial.

In higher dimensions, there is the generalization due to Yu.\ Brudnyi and Ganzburg \cite{Brudnyi:1973un}
for convex bodies $K \subseteq \R^n$: for Lebesgue measurable subsets $E\subseteq K$ with $|E|>0$
we have the sharp inequality
\begin{equation} \label{eq:R1}
    \|p\|_{K} \le T_d\Big(\frac{|K|^{1/n}+(|K|-|E|)^{1/n}}{|K|^{1/n}-(|K|-|E|)^{1/n}}\Big)\, \|p\|_E,
\end{equation}
for all polynomials $p \in \R[x_1,\ldots,x_n]$ of degree at most $d$. 

Beyond convex sets, \eqref{eq:R1} has been generalized by Pierzcha{\l}a \cite{Pierzchala:2015aa} 
to a large family of compact sets $K \subseteq \R^n$ admitting cusps, 
including all compact fat subanalytic sets: there exist $\nu >0$ and $m \in \N_{\ge 1}$, depending only on $K$, 
such that  
\begin{equation} \label{eq:R2}
    \|p\|_{K} \le T_{m d}\Big(\frac{|K|^{\nu}+(|K|-|E|)^{\nu}}{|K|^{\nu}-(|K|-|E|)^{\nu}}\Big)\, \|p\|_E,
\end{equation}
for all Lebesgue measurable $E \subseteq K$ with $|E|>0$ and all $p \in \R[x_1,\ldots,x_n]$ of degree at most $d$.
(Recall that a compact set $K \subseteq \R^n$ is called \emph{fat} if $K = \ol{\on{int}(K)}$.)

The Remez inequality has been extended to various classes of functions 
(e.g.\ analytic \cite{Brudnyi_1999}, plurisubharmonic \cite{Brudnyi:1999tb}, and holomorphic \cite{Brudnyi:2008aa}), 
generally, in the univariate or convex multivariate case. Still in the convex setting, 
\cite{Yomdin:2011tq} showed that, in the Remez inequality for polynomials, the Lebesgue measure can be replaced by a certain geometric invariant 
which may be nonzero for discrete or even finite sets $E$; see also \cite{BRUDNYI:2015vu}.

In this paper, we focus on quasianalytic Denjoy--Carleman classes $\cC_M$, that is, 
classes of smooth functions whose sequence of derivatives is bounded by a given positive sequence $M= (M_j)_{j \ge 0}$ 
with suitable properties. 

Nazarov, Sodin, and Volberg \cite{NazarovSodinVolberg04} proved a univariate Remez inequality for functions $f$
of class $\cC_M$, where the so-called \emph{Bang degree} replaces the polynomial degree. 
The Bang degree is a nonnegative integer which can be computed from the sequence $M$ 
and which depends on the size of $f$ as a parameter. This number was introduced by Bang \cite{Bang53}
and named Bang degree by \cite{NazarovSodinVolberg04}. 
A multivariate Remez inequality for functions $f$ of class $\cC_M$ on convex bodies $K \subseteq \R^n$
was obtained in \cite{Rainer:2022aa}, where also effective bounds on the zero sets were proved.

The main goal of this paper is to establish a Remez inequality for functions $f$ of class $\cC_M$ on 
a large family of compact sets $K \subseteq \R^n$ with ``tame geometry''. 
In addition, we develop several analytic and geometric consequences of this Remez inequality.
The family of admissible sets $K$ (the same as the one used by \cite{Pierzchala:2015aa})
includes all fat compact sets that are definable in $\R_{\cC_M}$, the o-minimal
expansion of the real ordered field by all restricted functions of class $\cC_M$ (assuming that $M$ satisfies some 
standard regularity conditions). In particular, all fat compact subanalytic sets are admissible.
The constants that appear in the Remez inequality depend only on the Bang degree associated with $M$ and 
on the geometry of $K$ and can be computed explicitly.
Similarly, all the consequences,   
including sublevel set estimates, reverse H\"older inequalities, 
inequalities of {\L}ojasiewicz, Harnack, and Markov type, and decay estimates for oscillatory integrals,
have a quantitative character. 
The Bang degree plays the role of an effective complexity parameter, analogous to the polynomial degree 
in semialgebraic geometry.

In the following, we present our main results in more detail.

\subsection{Quasianalytic Remez inequalities on tame sets}

Let $\mu^{|\infty} = (\mu_j)_{j \ge 0}$ be a positive infinite sequence satisfying $0 < \mu_1 \le \mu_2 \le \cdots$. 
Let us also assume here in the introduction (for simplicity of exposition) that $\mu_j \ge j$ for $j \ge 1$.
Given $M_0 >0$, we then set 
\[
    M_j := M_0 \prod_{i=1}^j \mu_i, \quad j \ge 1,
\]
and get a sequence $M=(M_j)_{j\ge 0}$, which conversely determines the pair $(\mu^{|\infty},M_0)$ 
in a unique way.

For open $U \subseteq \R^n$, the \emph{Denjoy--Carleman class} $\cC_M(U)$ consists of all $f \in \cC^\infty(U)$ 
such that for each compact $K \subseteq U$ there exist constants $C,\rh>0$ such that 
\begin{equation} \label{eq:DC}
     |\p^\al f(x)| \le C \rh^{|\al|} M_{|\al|}, \quad \text{ for all } x \in K, \, \al \in \N^n.
\end{equation}
For instance, $\cC_M(U)$ coincides with space of real analytic functions on $U$, if $\mu_j = j$.
By the Denjoy--Carleman theorem, the class $\cC_M$ is \emph{quasianalytic}, i.e., 
the Taylor series homomorphism $\cC_M(U) \to \R \llbracket x_1,\ldots,x_d \rrbracket$ at any point of $U$ is injective,
provided that $U$ is connected,
if and only if 
\begin{equation} \label{eq:Bang}
    \sum_{j \ge 1} \frac{1}{\mu_j} = \infty.
\end{equation}
See \cite{Rainer:2021aa} and \cite{Thilliez08} for background on Denjoy--Carleman classes. 
Throughout the introduction, we make the quasianalyticity assumption \eqref{eq:Bang}; 
this can be relaxed somewhat as shall be seen later.

Let $K \subseteq \R^n$ be nonempty and compact.
It will be convenient to work with the following variant of condition \eqref{eq:DC}:
a function $f : K \to \R$ 
is called \emph{$(\mu^{|\infty},M_0)$-smooth} if $f \in \cC^\infty(K)$ and 
\begin{equation}
    \|f\|_{j,K} := \sup_{x \in K} \sum_{|\al|=j} \frac{j!}{\al!} |\p^\al f(x)| \le M_j, \quad j \ge 0.
\end{equation}
(This is no restriction relative to \eqref{eq:DC} because the constants $C,\rh$ can be absorbed by $M$, 
by redefining $(\mu^{|\infty},M_0)$.)

Given $\mu^{|\infty}$ and $b>0$, we define the \emph{Bang degree} 
\begin{equation}
    \dd_{\mu^{|\infty}}(b) := \sup \Big\{ n \in \N : \sum_{j = j_0(b) +1}^n \frac{1}{\mu_j} < e \Big\},
\end{equation}
where $j_0(b) := \lceil \log(b^{-1}) \rceil_\N$ is the smallest nonnegative integer $j$ with $j \ge \log(b^{-1})$.
Thanks to \eqref{eq:Bang}, $\dd_{\mu^{|\infty}}(b)$ is a nonnegative integer.

We will define a collection $\sK(\R^n)$ of fat compact subsets of $\R^n$ admissible for the 
Remez inequality (see \Cref{ssec:admissiblesets}).
Informally, the sets in $\sK(\R^n)$ admit a uniformly controlled polynomial parametrization.
For the purposes of this introduction, it suffices to note that
$\sK(\R^n)$ includes all compact UPC-sets, 
in particular, convex bodies, fat compact sets with H\"older boundary, and all fat compact sets that are 
subanalytic or definable in $\R_{\cC_{M}}$, where $M = (M_j)_{j\ge 0}$ is any quasianalytic weight 
sequence fulfilling some standard regularity conditions (see \Cref{e:RCM}). 

With each $K \in \sK(\R^n)$ we will associate constants
\begin{equation}
    a_K > 0 \quad \text{ and } \quad 0 < \nu_K \le 1,
\end{equation}
which depend only on the geometry of $K$ (see \Cref{ssec:RemezKL}). 
For instance, if $K \subseteq \R^n$ is a convex body 
then $a_K = \diam(K)$ and $\nu_K = 1/n$.

The following quasianalytic Remez inequality is our main result. 

\begin{theorem} \label[t]{t:Remez_intro}
    Let $\mu^{|\infty} = (\mu_j)_{j \ge 1}$ and $M_0>0$ be as above.
    Let $L\subseteq \R^n$ be a convex body 
    and $K \in \sK(\R^n)$ such that $K \subseteq L$.
    Let $b_0 \in (0,M_0)$.
    Then for each $(\mu^{|\infty},M_0)$-smooth function $f : L \to \R$ with $\|f\|_K\ge b_0$ and 
    each Lebesgue measurable set $E \subseteq K$ with $|E| > 0$ we have
    \[
        \|f\|_K \le \Cnn\, \Big(\frac{|K|^{\nu_K}}{|K|^{\nu_K}-(|K|- |E|)^{\nu_K}}\Big)^{\nn} \, \|f\|_E,
    \]
    where
    \[
        \nn=2\, \dd_{a_K \,\mu^{|\infty}}(\tfrac{b_0}{M_0})
    \]
    and $\Cnn$ is a constant depending only on $\nn$ and $\mu^{|\infty}$.
\end{theorem}

\Cref{t:Remez_intro} can be interpreted as a 
``quantitative propagation-of-smallness principle'' for $(\mu^{|\infty},M_0)$-smooth functions
on general tame compact sets beyond the convex setting.

The constant $\Cnn$ is explicitly given in \Cref{d:Cnn}. For standard examples of $\mu^{|\infty}$ it 
is exponential in $\NN$.

The main defining property of the sets $K$ in $\sK(\R^n)$ is that they can be covered by (possibly infinite) 
families of images of compact convex or star-shaped subsets of the unit cube, with positive measure uniformly bounded below,
under polynomial maps $\R^n \to \R^n$ of uniformly bounded degree and uniformly controlled Jacobian.
While $\nu_K$ can be computed (with some effort) from these data, the constant $a_K$ is explicitly given as a 
function of the diameter $\diam(K)$ and the maximal degree $d$ of the polynomial maps; see \eqref{eq:aK}. 
It follows that $a_K$ grows linearly in $\diam(K)$ and asymptotically quadratically in $d$.

The proof of \Cref{t:Remez_intro} is based on a fundamental technical result of \cite{Pierzchala:2015aa}
on sets in $\sK(\R^n)$, combined with the univariate Remez inequality for $(\mu^{|\infty},M_0)$-smooth functions 
as well as its multivariate version on convex bodies.
In fact, a more general version of \Cref{t:Remez_intro} is stated in \Cref{t:RemezKL} which will be proved in \Cref{sec:poly,sec:Remez_proof}. 
In particular, \Cref{t:RemezKL} provides a Remez inequality for vector-valued 
maps $f = (f_1,\ldots,f_m) : L \to \R^m$.

\begin{remark}
    Contrary to the Remez inequality for polynomials (\eqref{eq:R1} and \eqref{eq:R2}), 
    the vector-valued Remez inequality of \Cref{t:RemezKL} does not simply follow from its single-valued version 
    applied component-wise. The reason is that we may have $\|f\|_K \ge b_0>0$ while $\|f_i\|_K$ is much smaller than $b_0$ for some $i$.
    A component-wise treatment would cause $\nn$ to be much larger than actually needed.
\end{remark}

\begin{remark} \label[r]{r:depb}
    The requirement that $\|f\|_K$ is bounded below by a positive number $b_0$, on which the degree $\nn$ depends, is vital, 
    as we will see in \Cref{ex:fk}. Setting $b_0 = \|f\|_K$ we obtain a Remez inequality, 
    where the Bang degree
    \begin{equation}
        \nn_{f,K} = 2\, \dd_{a_K \,\mu^{|\infty}}(\tfrac{\|f\|_K}{M_0})
    \end{equation}
    explicitly depends on $f|_K$,
    similarly to Remez's inequality for polynomials.  

    The Bang degree is however not a generalization of the polynomial degree and the 
    analogy between the two is limited. For instance,
    the Bang degree $\nn_{f,K}$ of a $(\mu^{|\infty},M_0)$-smooth function $f : L \to \R$ 
    tends to infinity
    as we shrink $K$ to a zero of $f$, since then $\|f\|_K \to 0$. (The accompaning decreasing effect on $a_K$ fails 
    to prevent this; see \Cref{p:an,p:Denjoy1,p:Denjoys}.)
\end{remark}

\begin{example} \label[e]{ex:fk}
Consider the sequence of polynomials $p_k(x) := \frac{x^k}{k!}$, for $k\ge 1$, on $\R$.
    Each $f_k := p_k|_{[0,1]}$ is $(\mu^{|\infty},1)$-smooth for $\mu_j \equiv 1$, or $\mu_j = j$ (if one insists on 
    the condition $\mu_j \ge j$), or other appropriate choices of $\mu^{|\infty}$.
    On the other hand,
    \begin{equation}
        \|f_k\|_{[0,1]} = \frac{1}{k!} \quad \text{ and } \quad \|f_k\|_{[0,\frac{1}{2}]} = \frac{1}{2^k k!}.
    \end{equation}
    Thus a Remez inequality for $f_k$ with a constant that is independent of the size of the functions $f_k$, i.e., 
    independent of $k$, cannot hold.
\end{example}

The set of functions for which \Cref{t:Remez_intro} applies (with 
specified $\mu^{\infty}$, $M_0$, and $b_0$)
can be infinite dimensional.

\begin{remark}
    Setting $K= B(x,r)$ and $b_0 = \|f\|_{B(x,r)}$ in \Cref{t:Remez_intro} we get a bound for the 
    \emph{doubling exponent} of $f$ on the ball $B(x,r)$:
    \begin{align*}
        \log \frac{\|f\|_{B(x,r)}}{\|f\|_{B(x,\tfrac{r}2)}} \le  (n \log 2 + \log n)\, \nn_f(x,r) + \log \mathbf{C}_{\nn_f(x,r)},
    \end{align*}
    where 
    \[
        \nn_{f}(x,r) = 2\, \dd_{2r \,\mu^{|\infty}}(\tfrac{\|f\|_{B(x,r)}}{M_0}).
    \]
\end{remark}

\subsection{Finite determinacy}

So far we only considered $\cC^\infty$-functions with derivatives of all orders controlled by 
the sequence $M=(M_j)$ and we assumed that the quasianalyticity assumption \eqref{eq:Bang} is 
satisfied.

Actually, the result in \Cref{t:Remez_intro} and many of its consequences 
depend only on the derivatives up to some finite order
and remain true even for functions that are only differentiable up to said finite order.
Also the quasianalyticity requirement can often be relaxed.

We will make this precise below, in particular, we will specify the finite order and use 
customized terminology; see \Cref{sec:weights}.

Next we discuss various analytic and geometric consequences of \Cref{t:Remez_intro} and 
retain its setting for simplicity.

\subsection{Volume of sublevel sets and applications}

In the setup of \Cref{t:Remez_intro}, we obtain the following corollaries; see \Cref{sec:sublevel}, 
where more general versions are proved.

\begin{corollary}[Volume of sublevel sets] \label[c]{c:sublevel_intro}
    The sublevel set $K_t := \{x \in K: |f(x)| \le t\}$ of $f|_K$ satisfies
    \begin{equation}
        |K_t| \le \frac{|K|}{\nu_K} \, \Big(\frac{\Cnn \,t}{\|f\|_K}\Big)^{1/\nn}, \quad t>0.
    \end{equation}
\end{corollary}

This simple consequence is very useful. For instance, it implies the following 
inequality of Nikolskii type which reverses H\"older's inequality.

\begin{corollary}[Comparison of $L^p$-norms] \label[c]{c:rH_intro}
    For all $0 < q < p \le  \infty$,
    \begin{equation}
        \|f\|_{L^p(K)} \le \Big(\frac{\Cnn}{\nu_K^\nn}\Big)^{1-\frac{q}{p}} \Big(\frac{q \nn+1}{|K|}\Big)^{\frac{1}{q}-\frac{1}{p}} \|f\|_{L^q(K)}.
    \end{equation}
\end{corollary}

We conclude integrability of $|f|^{-\al}$ in the range $0<\al<1/\nn$.

\begin{corollary}[Integrability of $|f|^{-\al}$]
    The integral $\int_K |f(x)|^{-\al}\, dx$ converges, provided that $0 < \al < 1/\nn$.
\end{corollary}

We get a bound for the mean oscillation of $\log |f|$ over Lebesgue measurable sets.

\begin{corollary}[Mean oscillation of $\log |f|$]
For each Lebesgue measurable set $E \subseteq K$ with $|E|>0$, the mean oscillation of $\log |f|$ over $E$ 
satisfies 
\begin{equation}
\on{mo}_E(\log |f|) \le
2\log \Cnn + 2\nn \,\Big( 1 + \log\Big(\frac{|K|}{\nu_K\, |E|}\Big)\Big).
\end{equation}     
\end{corollary}

This bound is not sufficient for concluding that $\log |f|$ is BMO, in contrast to polynomials or analytic functions; 
see \cite{Brudnyi_1999}.
What is missing is control near the zero set of $f$, since the Bang degree depends strongly on the size of the function;
see \Cref{r:depb} and \Cref{r:BMO}.

\subsection{Inequalities of {\L}ojasiewicz type}

Functions definable in polynomially bounded o-minimal structures satisfy inequalities of {\L}ojasiewicz type 
(cf.\ \cite[4.14]{vandenDriesMiller96}), 
but the constants are typically unspecified.
In \Cref{sec:Lojasiewicz}, we establish several inequalities of {\L}ojasiewicz type with \emph{explicit} 
constants, depending on the Bang degree and the geometry of $K$.

Following \cite[Definition 1.16]{Pierzchala:2022aa}, we say that $K \in \sK(\R^n)$ is 
\emph{$(c_K,\th_K)$-thick} if there exist constants $c_K,\th_K>0$ such that 
\begin{equation}
    |K \cap B(x,r)| \ge c_K\, r^{\th_K}, 
\end{equation}
for all $x \in  K$ and $0<r\le 1$. 
Every compact fat nonempty set $K \subseteq \R^n$ that is definable in a polynomially bounded o-minimal structure 
is $(c_K,\th_K)$-thick for some $c_K,\th_K>0$, as a consequence of the mentioned {\L}ojasiewicz inequalities 
(see \cite[Theorem 1.18]{Pierzchala:2022aa}). 
Each convex body $K \subseteq \R^n$
is $(c_K,\th_K)$-thick with $\th_K = n$.

\begin{corollary}
    In the setup of \Cref{t:Remez_intro},
    assume additionally that $K$ is $(c_K,\th_K)$-thick.
    For any function $g : K \to [-1,1]$, 
    we have  
    \begin{equation} 
        \|f\|_{B(x,\frac{|g(x)|}2) \cap K} \ge \frac{\|f\|_K}{\Cnn}\,\Big(\frac{c_K\, \nu_K}{2^{\th_K} \, |K|}\Big)^{\nn}\,  |g(x)|^{\th_K \nn},
    \end{equation}
for all $x \in K$. 
\end{corollary}

Note that we do not require $f|_K^{-1}(0) \subseteq g^{-1}(0)$. See \Cref{t:Lojasiewicz} for a more general version.

In particular, this applies to the distance to the zero set of $f$,
\begin{equation}
g(x) := \min\{\dist(x,f|_K^{-1}(0)), 1\}.
\end{equation}

Taking $g \equiv r$ for $r \in (0,1)$, we get a bound on the order of vanishing.

\begin{corollary}
    The order of vanishing of $f$ at each $x \in K$ is bounded by $\th_K \nn$.
\end{corollary}

We deduce several inequalities for the gradient with completely explicit constants.
For instance, the following corollary is a special case of \Cref{t:gradptw1}.

\begin{corollary} 
Assume that $(\grad f|_K)^{-1}(0) \subseteq f|_K^{-1}(0)$ and $M_0\ge1$.
Then,
\begin{equation} 
    |\grad f(x)| \ge\frac{1}{2\sqrt{M_2}} \Big(\frac{\|f\|_K}{2\Cnn}\Big)^{\frac{1}{\th_K\nn}}\,\Big(\frac{c_K\, \nu_K}{|K|}\Big)^{\frac{1}{\th_K}}\, |f(x)|^{\frac{3}{2} -\frac{1}{\th_K\nn}}
\end{equation}
in a neighborhood of $f|_K^{-1}(0)$ in $K$.
\end{corollary}

It is well-known (see \Cref{l:BM}) that there exist $c>0$ and $\al \in [0,1)$ such that 
\[
   |\grad f(x)| \ge c \, |f(x)|^\al,
\]
but $c$ and $\al$ are not explicit. 
We have the following variant with an explicit exponent $<1$; see \Cref{r:gradexp}.

\begin{corollary}
Assume that $(\grad f|_K)^{-1}(0) \subseteq f|_K^{-1}(0)$ and $M_0\ge1$.
Then,
\begin{equation} 
    \|\grad f\|_{B_x} \ge \frac{1}{2} \Big(\frac{2\|f\|_K}{3\Cnn}\Big)^{\frac{1}{\th_K\nn}}\,\Big(\frac{c_K\, \nu_K}{|K|}\Big)^{\frac{1}{\th_K}}\,  |f(x)|^{1-\frac{1}{\th_K \nn}}
\end{equation}
for $x$ in a neighborhood of $f|_K^{-1}(0)$ in $K$, where $B_x$ is a neighborhood of $x$ in $L$ that shrinks 
to $x$ if $x$ tends to $f|_K^{-1}(0)$.
\end{corollary}

See \Cref{sec:Lojasiewicz} for further inequalities for the gradient.

\subsection{Inequalities of Harnack type}

If $f : L \to \R$ is nonvanishing on $K$ and $\min_{x \in K} |f(x)| \ge m_0>0$, then  
\Cref{t:Remez_intro} easily implies 
\begin{equation}
    \max_{x \in K} |f(x)| \le 2\, \Cnn\, \Big(\frac{|K|^{\nu_K}}{|K|^{\nu_K}-(|K|- |E_{m_0}|)^{\nu_K}}\Big)^{\nn} \, \min_{x \in K} |f(x)|,
\end{equation}
where $E_{m_0} = \{x \in K : |f(x)| \le  \min_{x \in K} |f(x)| + m_0\}$. But it is difficult to determine $|E_{m_0}|$. 

In \Cref{sec:Harnack}, we combine this inequality for $f$ with the one for $1/f$ in order to obtain the following 
inequality of Harnack type with completely explicit constants. 
This requires an additional regularity assumption on $\mu^{|\infty}$ which allows to conclude that 
$1/f : L \to \R$ is $(\be_0 \,\mu^{|\infty},\frac{2}{m_0})$-smooth, provided that $f : L \to \R$ is $(\mu^{|\infty},M_0)$-smooth and $\min_{x \in L} |f(x)| \ge m_0>0$,
where 
\begin{equation}
    \be_0 = \be_0(\mu^{|\infty},n,M_0,m_0)>1.
\end{equation}

\begin{corollary}
   Then, for all $K \in \sK(\R^n)$ such that $K \subseteq L$,
   \[
       \max_{x \in K} |f(x)| \le \Cnn^2 \Big(\frac{2^{\nu_K}}{2^{\nu_K}-1}\Big)^{2\nn} \min_{x \in K} |f(x)|,
   \]
   where 
   \[
       \nn = 2 \, \dd_{a_K \be_0\, \mu^{|\infty}}(\tfrac{m_0}{2M_0}).
   \]
\end{corollary}

See \Cref{t:Harnack2}, where $\be_0$ is made fully explicit.

\subsection{Inequalities of Markov type}

The Bang degree gives an upper bound for the total number of zeros of an $(\mu^{|\infty},M_0)$-smooth function 
on an interval, 
as proved in \cite{Rainer:2022aa}; see \Cref{p:zeros}. 
Using this fact and the reverse H\"older inequality in 
\Cref{c:rH_intro}, we obtain inequalities of Markov type in \Cref{sec:Markov}. 

Here is an exemplary result for univariate functions.
Since the zero bound is applied to the first derivative $f'$, we work with the shifted weight $\mu^{|\infty}_{+1} = (\mu_{j+1})_{j\ge 1}$.

\begin{corollary} \label[c]{c:Markov_intro}
   Let $I \subseteq \R$ be a compact interval.
   Let $b_1 \in (0,M_1)$.
   Every $(\mu^{|\infty}, M_0)$-smooth function $f : I \to \R$ such that $\|f'\|_I \ge b_1$
   satisfies
   \begin{equation}
   \|f'\|_{I} \le   \frac{2\,\Cnn}{|I|}  (\nn+1)^2\, \|f\|_I,   
   \end{equation}
   where
   \[
       \nn = 2 \, \dd_{|I|\, \mu^{|\infty}_{+ 1}}(\tfrac{b_1}{M_1}).
   \]
\end{corollary}

This has consequence for sampling the supremum norm; see \Cref{ssec:sampling}.

In \Cref{sec:Markov}, we also provide multivariate versions of \Cref{c:Markov_intro} on $K \in \sK(\R^n)$.

\subsection{Oscillatory integrals}

In \Cref{sec:oscillatory}, we combine
the estimate for the volume of sublevel sets in \Cref{c:sublevel_intro} and the bound for the number of zeros 
with the van der Corput lemma (which we recall in \Cref{l:Corput}) 
in order to obtain decay estimates for oscillatory integrals whose phase functions are 
$(\mu^{|\infty},M_0)$-smooth.

\begin{corollary} \label[c]{c:osc1_intro}
    Let $I^n \subseteq \R^n$ be a compact cube.
    Let $b_2 \in (0,M_2)$.
    Let $f : I^n \to \R$ be a $(\mu^{|\infty},M_0)$-smooth function.
    Assume, for some $1 \le i_0 \le n$, we have $\|\p_{i_0}^2f\|_\ell \ge b_2$ for each line segment $\ell$ in $I^n$ with direction $e_{i_0}$. 
    Let $g : I^n \to \C$ belong to $\cC^1(I^n)$. Then, 
    \begin{equation} 
        \Big|\int_{I^n} e^{i\la f(x)} g(x)\, dx\Big| \le B\, \la^{-\frac{1}{\nn+2}},\quad \text{ if } \la \ge \Big(\frac{2}{b_2}\Big)^{1 + \frac{2}{\nn}}, 
    \end{equation}
    where 
    \begin{gather}
        \nn = 2\, \dd_{\sqrt n\, |I|\, \mu^{|\infty}_{+2}}(\tfrac{\frac{b_2}2}{M_2 + \frac{b_2}2}),
        \\
        B =  \frac{ \Cnn^{1/\nn}\, n\, |I|^n\,\|g\|_{I^n}}{\|\p_{i_0}^2f\|_{I^n}^{1/\nn}}  + 2C\, (\nn+1)\, \big(|I|^{n-1}\,\|g\|_{I^n} +  \|\p_{i_0} g\|_{L^1(I^n)}\big),
    \end{gather}
    and $C>0$ is an absolute constant. 
\end{corollary}

If $K \in \sK(\R^d)$ is definable in $\R_{\cC_M}$, we can use uniform finiteness in o-minimal structures 
and obtain the following variant for integrals over $K$.

\begin{corollary}
    In the setting of \Cref{c:osc1_intro}, 
    assume additionally that $K \in \sK(\R^n)$ is contained in the interior of $I^n$ and definable 
    in $\R_{\cC_M}$. If $\|\p_{i_0}^2 f\|_K \ge b_2'>0$, then 
    \begin{equation} 
        \Big|\int_{K} e^{i\la f(x)} g(x)\, dx\Big| \le B\, \la^{-\frac{1}{\nn+2}},\quad \text{ for } \la >0,
    \end{equation}
    where 
    \begin{gather}
        \nn = 2\, \dd_{\sqrt n\, |I|\, \mu^{|\infty}_{+2}}(\tfrac{\frac{b_2}2}{M_2 + \frac{b_2}2}),
        \\
        B =  \frac{\Cnn^{1/\nn}\, |K|\,  \|g\|_{K}}{\nu_K\, \|\p_{i_0}^2f\|_{K}^{1/\nn}}  + 2C\, N_{f,K}\, \big(|I|^{n-1}\,\|g\|_{I^n} +  \|\p_{i_0} g\|_{L^1(I^n)}\big),
    \end{gather}
    $C>0$ is an absolute constant, and $N_{f,K} \in \N$ is an integer depending only on $f$ and $K$.
\end{corollary}

\subsection{Metric entropy of sublevel sets}

Let $\ep>0$ and $K \subseteq \R^n$ relatively compact.
The \emph{covering number} $N_\ep(K)$ of $K$ is the minimal number of closed balls of radius $\ep$ covering $K$
and $H_\ep(K) := \log_2 N_\ep(K)$ is the \emph{$\ep$-entropy} of $K$.

\begin{corollary}
    In the setup of \Cref{t:Remez_intro}, assume additionally that $K \in \sK(\R^n)$ is definable in  $\R_{\cC_M}$ and $L = B(0,r)$.
    Then the covering number of the sublevel set $K_t$ of $f|_K$ satisfies 
    \begin{equation}
        N_\ep(K_t) \le C(n) \frac{|K|}{\nu_K} \, \Big(\frac{\Cnn \,t}{\|f\|_K}\Big)^{1/\nn} \frac{1}{\ep^n} + C_{K_t} \Big(1+ \Big(\frac{r}{\ep}\Big)^{n-1}\Big), \quad t>0,
    \end{equation}
    for a constant $C_{K_t}>0$.  
\end{corollary}

\begin{proof}
    This follows from \Cref{c:sublevel_intro} and \cite[Proposition 5.6]{Yomdin:2004aa}.
\end{proof}

The constant $C_{K_t}$ is of the form 
\begin{equation}
    C_{K_t} = C(n) \max_{1\le j \le n-1} \om_{j}\cdot B_{n-j}(K_t),
\end{equation}
where $B_{n-j}(K_t)$ bounds the number of connected components of $P \cap K_t$, where $P$ is any affine $(n-j)$-plane in $\R^n$.
The bounds $B_{n-j}(K_t)$ exist because $K_t$ is definable in $\R_{\cC_M}$.
It would be interesting to find explicit bounds via the Bang degree depending on the Denjoy--Carleman complexity of $K_t$.

\subsection{Organization of the paper}

In \Cref{sec:weights}, we introduce the admissible weights that control the derivatives of 
quasianalytic functions and define the Bang degree. We also introduce terminology  that 
accounts for the fact that often differentiability of  finite order is sufficient.

In \Cref{sec:o-minimal}, we recall basics on o-minimal structures, in particular, the expansion 
of the real field by restricted function in a quasianalytic Denjoy--Carleman class.

In \Cref{sec:Remez}, we formulate the general version of the quasianalytic Remez inequality 
(\Cref{t:RemezKL}) and discuss the collection $\sK(\R^d)$ of sets that support this inequality.
Its proof comprises \Cref{sec:poly,sec:Remez_proof}.

In \Cref{sec:sublevel}, we deduce estimates for the volume of sublevel sets and obtain a  
number of consequences, including a reverse H\"older inequality and bounds on the mean 
oscillation of $\log |f|$.

\Cref{sec:Lojasiewicz} is devoted to several inequalities of {\L}ojasiewicz type. 
In \Cref{sec:Harnack}, inequalities of Harnack types are obtained. 
Combining the reverse H\"older inequality with a bound for the number of zeros in terms 
of the Bang degree, we obtain inequalities of Markov type in \Cref{sec:Markov}. 

Finally, in \Cref{sec:oscillatory}, the sublevel set approach together with the van der 
Corput lemma are used to provide decay effective estimates for oscillatory integrals. 

The paper ends with two appendices. 
In \Cref{sec:A}, we establish precise upper and lower bounds for the Bang degree 
associated with specific admissible weights (the analytic weight and Denjoy weights).
In \Cref{sec:B}, we reexamine the proof of the univariate quasianalytic Remez inequality 
of \cite{NazarovSodinVolberg04} and adjust it to the vector-valued case.

\subsection{Notation}

We use the Euclidean norm $|x| = (\sum_{j=1}^n x_j^2)^{1/2}$ in $\R^n$. 
Sometimes we also use the $1$-norm $|x|_1 = \sum_{j=1}^n |x_j|$.
Recall that $|x| \le |x|_1 \le \sqrt n \, |x|$.

For differentiable functions $f : \R^n \supseteq K \to \R$ we typically use (semi)norms based on the $1$-norm:
\begin{equation}
    \|f\|_{j,K} := \sup_{x \in K} \sum_{|\al|=j} \frac{j!}{\al!} |\p^\al f(x)|,
\end{equation}
in particular, 
\begin{equation}
    \|\grad f\|_{K} := \sup_{x \in K} \sum_{i=1}^n |\p_i f(x)| = \|f\|_{1,K}.
\end{equation}

The Lebesgue measure in $\R^n$ is denoted by $\cL^{n}$ and often we also write $|E| = \cL^n(E)$ if the 
ambient dimension $n$ is clear. Let $\om_n$ denote the volume of the unit ball $B(0,1) \subseteq \R^n$. 
The standard basis vectors in $\R^n$ are denoted by $e_1,\ldots,e_n$.

\section{Weights, quasianalytic functions, and the Bang degree} \label{sec:weights}

This section introduces the central objects of the paper and sets the terminology used throughout.

\subsection{Admissible weights and \texorpdfstring{$(\mu^{|N},M_0)$}{(mu,M0)}-smooth functions}

\begin{definition}[Admissible weights]
    An \emph{admissible weight} is an infinite increasing sequence $\mu^{|\infty}=(\mu_j)_{j \ge 1}$ of elements in $\R_{>0}$, i.e.,
    $0< \mu_1 \le \mu_2 \le \cdots$.
\end{definition}

\begin{definition}[Quasianalytic weights]
    An admissible weight $\mu^{|\infty} = (\mu_j)_{j \ge 1}$ is called \emph{quasianalytic} if 
\begin{equation}
    \sum_{j \ge 1}  \frac{1}{\mu_j} = \infty.
\end{equation}
\end{definition}

Given an admissible weight $\mu^{|\infty} = (\mu_j)_{j\ge 1}$
and any positive real number $M_0>0$, 
we define 
\begin{equation} \label{eq:M}
    M_j := M_0 \prod_{i = 1}^j \mu_i, \quad j \ge 1.
\end{equation}
Then the sequence $M=(M_j)_{j \ge 0}$ satisfies the logarithmic convexity property 
\begin{equation}
    M_j^2 \le M_{j-1}M_{j+1},\quad j \ge 1.
\end{equation}
Conversely, $M=(M_j)_{j \ge 0}$ uniquely determines the pair $(\mu^{|\infty},M_0)$ (setting $\mu_j = \frac{M_j}{M_{j-1}}$).

\begin{definition}[$(\mu^{|N},M_0)$-smooth functions and maps]
    Let $\mu^{|\infty} = (\mu_j)_{j\ge 1}$ be an admissible weight and $M_0>0$.
    Let $N \in \N_{\ge 1} \cup \{\infty\}$.
    A function $f : K \to \R$, defined on a nonempty compact set $K \subseteq \R^n$,
    is called \emph{$(\mu^{|N},M_0)$-smooth} if $f$ is of class $\cC^{N}$ (in some open neighborhood of $K$) 
    and 
    \begin{equation}
        \|f\|_{j,K} \le M_j, \quad 0 \le j < N +1,
    \end{equation}
    where 
    \[
        \|f\|_{j,K} :=\sup_{x \in K} \sum_{|\al|=j} \frac{j!}{\al!} |\p^\al f(x)|
    \]
    and $\|f\|_K = \|f\|_{0,K} = \sup_{x \in K} |f(x)|$ denotes the sup-norm. 

    Similarly, a map $f = (f_1,\ldots,f_m) : K \to \R^m$ is 
    called \emph{$(\mu^{|N},M_0)$-smooth} if $f$ is of class $\cC^{N}$ 
    and 
    \begin{equation}
        \|f\|_{j,K} \le M_j, \quad 0 \le j < N +1,
    \end{equation}
    where 
    \[
        \|f\|_{j,K} :=\sup_{x \in K} \sum_{i=1}^m\sum_{|\al|=j} \frac{j!}{\al!} |\p^\al f_i(x)|.
    \]
    We set $\|f\|_K := \|f\|_{0,K} = \sup_{x \in K} \sum_{i=1}^m |f_i(x)|$. 
\end{definition}

\begin{remark}
    In the single-valued case,
    $\|f\|_{j,K}$ is the supremum over $x \in K$ of the $1$-norm of the vector with entries $\p^\al f(x)$, 
    where $|\al|=j$ and the entry $\p^\al f(x)$ appears $\frac{j!}{\al!}$ times. 
    This is convenient because the directional derivatives satisfy
    \begin{align*}
        |d_v^j f(x)|=\big|\p_t^j|_{t=0} f(x+tv)\big| = \Big|\sum_{|\al|=j} \frac{j!}{\al!} \p^\al f(x)v^\al\Big| 
        \le |v|^j \sum_{|\al|=j} \frac{j!}{\al!} |\p^\al f(x)|.
    \end{align*}
    The vector-valued case behaves similarly.
\end{remark}

\subsection{The Bang degree}

Let $\mu^{|\infty}=(\mu_j)_{j \ge 1}$ be an admissible weight.
We consider the function $\Si_{\mu^{|\infty}} : \N_{\ge 1} \times \N  \to [0,\infty)$
defined by
\[
  \Si_{\mu^{|\infty}}(m,n) := \sum_{j =m}^n \frac{1}{\mu_j}.
\]
If $n < m$ then $\Si_{\mu^{|\infty}}(m,n)=0$ (being an empty sum). 
The function $(m,n)\mapsto \Si_{\mu^{|\infty}}(m,n)$ is increasing in $n$
and decreasing in $m$. 

For $b>0$, let $j_0(b)$ be the smallest integer $j\in \N$ with $j \ge \log (b^{-1})$, i.e.,
\[
  j_0(b):= \lceil \log (b^{-1})\rceil_{\N}.
\]

\begin{definition}[Bang degree]
    Let $\mu^{|\infty}=(\mu_j)_{j \ge 1}$ be an admissible weight and $b>0$.
    We define the \emph{Bang degree} $\dd_{\mu^{|\infty}}(b) \in \N \cup \{\infty\}$
by setting
\begin{align*}
  \dd_{\mu^{|\infty}}(b) := \sup\big\{n \in \N:  \Si_{\mu^{|\infty}}(j_0(b)+1,n)
  < e \big\}.
\end{align*}
\end{definition}

\begin{lemma} \label[l]{l:Bang}
    We have:
    \begin{enumerate}
        \item $\dd_{\mu^{|\infty}}(b) \ge j_0(b)$ and $\dd_{\mu^{|\infty}}(b) = j_0(b)$ if and only if $\mu_{j_0(b)+1}\le  1/e$.
        \item $\dd_{\mu^{|\infty}}(b) = \infty$ if and only if $\Si_{\mu^{|\infty}}(j_0(b)+1,n) < e$ for all $n$.
        \item The map $(a,b) \mapsto \dd_{a\,\mu^{|\infty}}(b)$ is increasing in $a>0$ and decreasing in $b>0$.
        \item If $\mu^{|\infty}$ is quasianalytic, then $\dd_{\mu^{|\infty}}(b)$ is finite.
    \end{enumerate}
\end{lemma}

\begin{proof}
    All properties follow easily from the definition.
\end{proof}

\begin{remark}
  Our notion of Bang degree differs slightly from the notion proposed in \cite{NazarovSodinVolberg04}, 
  where it is attached to 
  a given $(\mu^{|\infty},M_0)$-smooth function $f : [0,1] \to \R$
  and equals $\dd_{\mu^{|\infty}}(\|f\|_{[0,1]}/M_0)$ (in our notation). 
  The unspecified argument $b$ allows for more flexibility.
\end{remark}

\subsection{Shifted weights} 

Occasionally, we will deal with derivatives of $(\mu^{|\infty},M_0)$-smooth functions. This naturally requires 
a ``shift'' of the weight $\mu^{|\infty}$.

Given an admissible weight $\mu^{|\infty} = (\mu_j)_{j \ge 1}$, also the shifted sequence 
\begin{equation} \label{eq:shift}
 \mu^{|\infty}_{+i} = (\mu_{j+i})_{j \ge 1},   
\end{equation}
where $i \in \N_{\ge 1}$,
is an admissible weight.
Evidently, $\mu^{|\infty}$ is quasianalytic if and only if $\mu^{|\infty}_{+i}$ is quasianalytic.

The Bang degree for $\mu^{|\infty}_{+i}$ can always be reduced to the Bang degree for $\mu^{|\infty}$.

\begin{lemma}
    We have $\dd_{\mu^{|\infty}_{+i}}(b) = \dd_{\mu^{|\infty}}(e^{-i} b) - i$.
\end{lemma}

\begin{proof}
    This follows from $\Si_{\mu^{|\infty}_{+1}} (j_0(b)+1,n) = \Si_{\mu^{|\infty}} (j_0(b)+2,n+1)$.
\end{proof}

\subsection{Further definitions and conventions}

Let us adjust the definition of the Bang degree to its use in the Remez inequality.
We will not automatically assume that $\mu^{|\infty}$ is quasianalytic.
Instead we work with an admissibility condition that guarantees that the Bang degree is finite.

\begin{definition}[Associated degree] \label[d]{d:nn}
    Let $\mu^{|\infty}= (\mu_j)_{j\ge 1}$ be an admissible weight and $M_0,a,b>0$.
    We say that $(\mu^{|\infty},M_0,a,b)$ is \emph{admissible}, if 
    \begin{equation} \label{eq:condition}
        \sum_{j=j_0+1}^\infty \frac{1}{\mu_j} > a\, e, \quad \text{ where }  j_0 = \lceil \log (\tfrac{M_0}{b})\rceil_{\N}.
    \end{equation}
    In that case, we define the integer
    \begin{equation} \label{eq:nn}
        \nn= \nn(\mu^{|\infty},M_0,a,b):= 2\, \dd_{a\, \mu^{|\infty}}(\tfrac{b}{M_0})
    \end{equation}
    and say that $\nn$ is the \emph{degree associated with $(\mu^{|\infty},M_0,a,b)$}.
\end{definition}

\begin{remark} \label[r]{r:nnprop}
    (1) By definition, for $r,s>0$, $(r\, \mu^{|\infty},s\, M_0,a,b)$ is admissible if and only 
    if $(\mu^{|\infty},M_0,r\, a, \frac{b}{s})$ is admissible,
    and
    \[
        \nn(r \, \mu^{|\infty},s\, M_0,a,b) = \nn(\mu^{|\infty},M_0,r\, a, \tfrac{b}{s}).
    \]

    (2) Assuming admissibility, the integer $\nn(\mu^{|\infty},M_0,a,b)$ is increasing in $a$ and decreasing in $b$.

    (3)  If $\mu^{|\infty}$ is quasianalytic, then \eqref{eq:condition} is satisfied for all $M_0,a,b>0$.
\end{remark}

\begin{remark} \label[r]{r:Nnonzero}
    In our results, the integer $\nn$ will be useless if it is zero.
    We have $\nn(\mu^{|\infty},M_0,a,b) =0$ if and only if $ae \mu_1 \le 1$ and $M_0 \le b$ (by \Cref{l:Bang}(1)).
    Therefore, we will generally assume that $M_0 > b$ so that $\nn(\mu^{|\infty},M_0,a,b) \ge 2$.

    In \cite{Rainer:2022aa}, we occasionally used $M_0 = b = \|f\|_K$, for a $(\mu^{|\infty},M_0)$-smooth function $f$, 
    and took extra care for the case $\nn(\mu^{|\infty},M_0,a,b) =0$.
    These occasions can simply be adjusted to the convention $M_0 >b$ by setting $b = \|f\|_K$ and $M_0 = 2 \,\|f\|_K$.
    This possibly increases $\nn$ a little, while $f$ is still $(\mu^{|\infty},M_0)$-smooth.
\end{remark}

\begin{definition}[The function $\ga_{\tilde \mu}$] \label[d]{d:ga}
Let $\mu^{|\infty}= (\mu_j)_{j\ge 1}$ be an admissible weight.
Let $\tilde \mu : [1,\infty) \to (0,\infty)$ be an increasing continuous and (piecewise) $\cC^1$ 
function such that
\[
  \mu_j = \tilde \mu(j), \quad j\ge 1.
\]
For instance, we may take $\tilde \mu$ piecewise affine and work consistently with the left derivative
at points, where $\tilde \mu$ is not differentiable.

Then we define (following \cite{NazarovSodinVolberg04})
\begin{equation*}
    \ga_{\tilde \mu}(n) := \sup_{1 \le s\le n} \frac{s  \tilde \mu'(s)}{\tilde \mu(s)}, \quad n \in \N_{\ge 1}.  
\end{equation*}
Note that $\ga_{\tilde \mu}$ depends on the choice of $\tilde \mu$ which is not unique.
\end{definition}

\begin{definition}[The constant $\Cnn$] \label[d]{d:Cnn}
    Let $\mu^{|\infty}= (\mu_j)_{j\ge 1}$ be an admissible weight and $M_0,a,b>0$ with $M_0 >b$.
   Assume that $(\mu^{|\infty},M_0,a,b)$ is admissible and let $\nn=\nn(\mu^{|\infty},M_0,a,b)$.
   Suppose that $\tilde \mu$ has been chosen.
    We define 
\begin{equation}
    \Cnn :=  (4 e^{4 + \ga_{\tilde \mu}(\nn)})^{\nn}.
\end{equation}
\end{definition}

If $\sup_{1 \le s < \infty} \frac{s  \tilde \mu'(s)}{\tilde \mu(s)}$ is finite (see e.g.\ \Cref{p:anGa,p:DenjoyGa}), 
then $\Cnn$ can be replaced 
by the potentially larger constant $C^\nn$ for some suitable $C>0$.

\begin{remark}
    Our notation differs somewhat from the notation used in \cite{Rainer:2022aa}. For instance, 
    the integer $\nn$ defined in \eqref{eq:nn} is twice the integer denoted by the same symbol in \cite{Rainer:2022aa}.
\end{remark}

\section{O-minimal structures} \label{sec:o-minimal}

We recall the definition of an \emph{o-minimal structure} over the real ordered field and some background;
cf.\ \cite{vandenDriesMiller96} and \cite{vandenDries98}.

\begin{definition}[O-minimal structure over the real ordered field]
A \emph{structure}  over the real ordered field $(\R,+,\cdot,<)$ is a sequence $\sS = (\sS_n)_{n\ge 1}$
of collections of subsets $\sS_n$ of $\R^n$, such that for all $n,n'\ge 1$:
\begin{enumerate}
  \item $\sS_n$ is a Boolean algebra with respect to the usual set-theoretic operations;
  \item $\sS_n$ contains all semialgebraic subsets of $\R^n$;
  \item if $X \in \sS_n$ and $X' \in \sS_{n'}$, then $X \times X' \in \sS_{n+n'}$;
  \item if $n \ge n'$ and $X \in \sS_n$, then $\pi(X) \in \sS_{n'}$, where $\pi : \R^{n}\to \R^{n'}$
  is the projection on the first $n'$ coordinates.
\end{enumerate}
A subset $X \subseteq \R^n$ is said to be \emph{definable} in the structure $\sS$ if $X \in \sS_n$.
A map $f : X \to \R^{n'}$ is called \emph{definable} in $\sS$ if its graph is definable.
A structure $\sS$ is called \emph{o-minimal} if
\begin{enumerate}
    \item[(5)] every set in $\sS_1$ is a finite union of intervals and points.
\end{enumerate}
A structure $\sS$ is called \emph{polynomially bounded} if for every function $f : \R \to \R$ that is definable in $\sS$
there exists $N \in \N$ such that $f(t) = O(t^N)$ as $t \to \infty$.
An o-minimal structure $\sS$ either is polynomially bounded or the exponential function $\exp : \R \to \R$ is definable in $\sS$
(see \cite{Miller:1994ue}).
\end{definition}

\begin{example}[Semialgebraic sets]
    The family of all semialgebraic sets forms a polynomially bounded o-minimal structure.
\end{example}

\begin{example}[Globally subanalytic sets]
    The family of globally subanalytic sets forms a polynomially bounded o-minimal structure denoted $\R_{\on{an}}$.
  It is the smallest structure over the real ordered field containing all restricted analytic functions
  $f : \R^n \to \R$, i.e., $f|_{[-1,1]^n}$ is analytic and $f=0$ outside $[-1,1]^n$.
\end{example}

\begin{example}[The structure $\R_{\cC_M}$] \label[e]{e:RCM}
    Let $\mu^{|\infty} = (\mu_j)_{j \ge 1}$ be an admissible weight and $M_0 >0$. 
    Let $M = (M_j)_{j \ge 0}$ be defined by \eqref{eq:M}.
    Consider 
    the associated Denjoy--Carleman class $\cC_M$ (see \eqref{eq:DC}).

    Let us make the following additional assumptions on the admissible weight $\mu^{|\infty}$.
    \begin{enumerate}
        \item The sequence 
            \[
                a_j := \Big(\frac{\mu_1\cdots \mu_j}{j!}\Big)^{1/j},\quad j \ge 1,
            \]
            satisfies $\liminf_{j \to \infty} a_j >0$.
        \item The sequence $a_j$ is almost increasing, i.e., there exists a constant $H \ge 1$ such that 
            $a_j \le H  a_k$ whenever $j \le k$.
        \item There is a constant $C\ge 1$ such that $\mu_{j+1} \le C^j$, equivalently, 
                    $M_{j+1} \le C^j M_j$, for $j \ge 1$.
                \item $\mu^{|\infty}$ is quasianalytic.
    \end{enumerate}
    Under the assumptions (1)--(4), the class $\cC_M$ is quasianalytic, containes all real analytic functions, 
    and is stable under composition, differentiation, and inverse and implicit functions (cf.\ \cite{Rainer:2021aa}). 
    
    Under these assumptions, the expansion $\R_{\cC_M}$ 
    of the real ordered field by all restricted $\cC_M$-functions 
    is a polynomially bounded o-minimal structure, by \cite{RolinSpeisseggerWilkie03}. 
    A function $f : \R^n \to \R$ is called a \emph{restricted $\cC_M$-function} if 
    there exists $\tilde f \in \cC_M(U)$ for an open neighborhood $U$ of $[-1,1]^n$ 
    such that $f(x) = \tilde f(x)$ if $x \in [-1,1]^n$ and $f(x) = 0$ if  $x \not\in [-1,1]^n$.

    Note that (2) is satisfied with $H=1$ if the sequence $\mu_j/j$ is increasing (which means that $M_j/j!$ is logarithmically convex).
    Under this stronger condition, one can drop the assumption (3) (which guarantees that $\cC_M$ is stable under differentiation): 
    the expansion $\R_{\ol \cC_M}$ 
    of the real ordered field by all restricted $\ol \cC_M$-functions 
    is a polynomially bounded o-minimal structure, where 
    \[
        \ol \cC_M := \bigcup_{k\ge 0} \cC_{M_{+k}} \quad \text{ and } \quad M_{+k} := (M_{j+k})_{j\ge 0}.
    \]
\end{example}

\section{Quasianalytic Remez inequalities} \label{sec:Remez}

After reviewing quasianalytic Remez inequalities on intervals and convex bodies, 
we discuss a large collection of sets in $\R^n$ to which the quasianalytic Remez inequality can be extended; 
see \Cref{t:RemezKL}.

\subsection{Remez inequality for univariate functions}

The following theorem is a Remez inequality for univariate vector-valued functions.
It generalizes \cite[Theorem B]{NazarovSodinVolberg04} for single-valued functions; see also \cite[Theorem 5.4]{Rainer:2022aa}.
It is proved in \Cref{sec:B} for the reader's convenience. 

\begin{theorem} \label[t]{thm:NSV}
    Let $\mu^{|\infty}=(\mu_j)_{j\ge1}$ be an admissible weight and $M_0>b_0>0$.
    Assume that $(\mu^{|\infty},M_0,1,b_0)$ is admissible and let $\nn = \nn(\mu^{|\infty},M_0,1,b_0)$.
    Let $f: [0,1] \to \R^m$ be $(\mu^{|\nn},M_0)$-smooth such that $\|f\|_{[0,1]}\ge b_0$.
  Then for each interval $I \subseteq [0,1]$ and each Lebesgue measurable set $E \subseteq I$ with $|E|>0$ we have
  \begin{equation} \label{eq:NSV}
     \|f\|_I \le \Cnn \,\Big(\frac{m\, |I|}{|E|}\Big)^{\nn} \|f\|_E.
  \end{equation}
\end{theorem}

\subsection{Remez inequality on convex bodies}

\Cref{thm:NSV}, for $m=1$, was extended to convex bodies in $\R^n$ in \cite[Theorem 5.4]{Rainer:2022aa}. 
Here we state a vector-valued version.

\begin{theorem} \label[t]{thm:Remez}
    Let $\mu^{|\infty}=(\mu_j)_{j\ge 1}$ be an admissible weight and $M_0>b_0>0$.
    Let $K \subseteq L\subseteq \R^n$ be convex bodies. 
    Assume that $(\mu^{|\infty},M_0,\diam(K),b_0)$ is admissible
    and let $\nn=\nn(\mu^{|\infty},M_0,\diam(K),b_0)$. 
  Let $f : L \to \R^m$ be $(\mu^{|\nn},M_0)$-smooth 
  such that $\|f\|_{K} \ge b_0$.
  If $E \subseteq K$ is a Lebesgue measurable subset with $|E|>0$, then
  \begin{equation}
     \|f\|_{K} \le \Cnn \Big(
     \frac{m\, |K|^{1/n}}{|K|^{1/n} - (|K|-|E|)^{1/n}}
     \Big)^{\nn} \|f\|_E.
  \end{equation}
\end{theorem}

Even for $m=1$,
the statement of \Cref{thm:Remez} is slightly stronger than the one of \cite[Theorem 5.4]{Rainer:2022aa}. 
In fact,
in the latter, the potentially bigger integer $\nn(\mu^{|\infty},M_0,\diam(L),b_0)$ is 
used rather than $\nn=\nn(\mu^{|\infty},M_0,\diam(K),b_0)$.
The stronger vector-valued version is obtained by the same proof (based on the technique of \cite{Brudnyi:1973un}). 
Alternatively, we shall see that \Cref{thm:Remez} follows from \Cref{c:starlike} below.

\subsection{Star-shaped sets}

The following lemma, due to \cite[Lemma 4.1]{Pierzchala:2015aa},
will play an important role in the proof of \Cref{t:Remez_intro}.

\begin{lemma} \label[l]{l:ga}
    Let $K \subseteq \R^n$ be a compact set that is star-shaped with respect to $x_0 \in K$ and satisfies $|K|>0$. 
    Let $E \subseteq K$ be Lebesgue measurable. 
    Then there exists an affine curve $\ga : [0,1] \to K$ with $\ga(0) = x_0$ and such that $\ga^{-1}(E)$ contains a 
    Lebesgue measurable subset $F$ satisfying 
    \begin{equation} \label{eq:starlike}
        \cL^1(F) \ge 1 -\Big(1 - \frac{|E|}{|K|}\Big)^{1/n}.
    \end{equation}
\end{lemma}

We give a corollary which implies \Cref{thm:Remez}: for the convex body $K$, it suffices to choose $x_0 \in K$ such that $|f(x_0)|_1 = \|f\|_K$.

\begin{corollary} \label[c]{c:starlike}
    Let $\mu^{|\infty}=(\mu_j)_{j \ge 1}$ be an admissible weight and $M_0 >b_0>0$.
    Let $K \subseteq  \R^n$ be a compact set that is star-shaped with respect to $x_0 \in K$ and such that $|K|>0$. 
    Assume that $(\mu^{|\infty},M_0,\diam(K),b_0)$ is admissible
    and let $\nn=\nn(\mu^{|\infty},M_0,\diam(K),b_0)$. 
  Let $f : K \to \R^m$ be $(\mu^{|\nn},M_0)$-smooth 
  such that $|f(x_0)|_1 \ge b_0$.
  If $E \subseteq K$ is a Lebesgue measurable subset with $|E|>0$, then
  \begin{equation} 
     |f(x_0)|_1 \le \Cnn \Big(
     \frac{m\, |K|^{1/n}}{|K|^{1/n} - (|K|-|E|)^{1/n}}
     \Big)^{\nn} \|f\|_E.
  \end{equation}
\end{corollary}

\begin{proof}
    Let $\ga : [0,1] \to K$, $\ga(t) = x_0 + t (x_1 - x_0)$ with $x_1 = \ga(1)$, be the affine curve from \Cref{l:ga}.
    We apply \Cref{thm:NSV} to $g:= f \o \ga$ which is $(\diam(K)\, \mu^{|\nn},M_0)$-smooth
    and to the set $F$ from \Cref{l:ga},
    \begin{align}
        |f(x_0)|_1 = |g(0)|_1 &\le \Cnn \Big(\frac{m}{\cL^1(F)} \Big)^\nn\, \|g\|_F.
    \end{align}
    This implies the assertion, by \eqref{eq:starlike}.
\end{proof}

\subsection{Admissible sets} \label{ssec:admissiblesets}

We introduce a collection of sets in $\R^n$ admissible for the Remez inequality. 
The following definitions are based on \cite[Definition 9.1]{Pierzchala:2015aa}.

\begin{definition}[Admissible collection] \label[d]{d:admcoll}
    Suppose that a family $(K_w,a_w,\Up_w)_{w \in W}$ is given, where $K_w \subseteq [0,1]^n$ is compact, $a_w \in K_w$, 
    $\Up_w : [0,\infty)^n \to \R^n$ is a map, and $W$ is an arbitrary set. 
    This family is called an \emph{admissible collection} in $\R^n$ if the following holds.
    There exist constants $\de_1,\de_2,\de_3 >0$, $d,q \in \N_{\ge 1}$, and polynomial maps $R_w : \R^n \to \R^n$ of degree at most $d$ such that, 
    for all $w \in W$:
    \begin{enumerate}
        \item $K_w$ is star-shaped with respect to $a_w$ and $|K_w| \ge \de_1$;
        \item $\Up_w(x_1,\ldots,x_n) = R_w(x_1^{1/q},\ldots,x_n^{1/q})$;
        \item $\de_2 \le \sum_\ka |c_{w,\ka}| \le \de_3$, where $c_{w,\ka}$ are the coefficients of the polynomial $\on{Jac}(R_w)$.
        \item $a_{w,i} = 0$ whenever $\Up_w$ is not a polynomial map with respect to the variable $x_i$, where $a_w = (a_{w,1},\ldots,a_{w,n})$.
    \end{enumerate}
    We also say that $(d,q,\de_1,\de_2,\de_2)$ is the \emph{characteristic} of $(K_w,a_w,\Up_w)_{w \in W}$. 
\end{definition}

\begin{definition}[Admissible set]
    Let $K \subseteq \R^n$ be a nonempty compact set.
    We say that $K$ is an \emph{admissible set} if 
    there exists an admissible collection $(K_w,a_w,\Up_w)_{w \in K}$ in $\R^n$ such that 
    \begin{enumerate}
        \item $\Up_w(a_w) = w$ for each $w \in K$;
        \item $K = \bigcup_{w \in K} \Up_w(K_w)$.
    \end{enumerate}
    We also say that $K$ has the \emph{characteristic} $(d,q,\de_1,\de_2,\de_2)$ if $(K_w,a_w,\Up_w)_{w \in K}$ has this characteristic. 

    Let us denote by $\sK(\R^n)$ the family of all admissible sets in $\R^n$.
\end{definition}

In particular, every compact set $K \subseteq \R^n$ with the \emph{special parameterization property}
as defined in \cite[Definition 1.3]{Pierzchala:2015aa}
is admissible.

\begin{example}
Each convex body $K \subseteq \R^n$ is an admissible set with characteristic $(1,1,\de_1,\de_2,\de_3)$.
Indeed, we may use the admissible collection $(K_w,a_w,\Up_w)_{w \in K}$, where
$a_w := (\frac{1}{2},\ldots,\frac{1}{2})$, $T_w(x) := a_w + \frac{1}{2\diam(K)}(x-w)$, 
$K_w := T_w(K)$, and $\Up_w := T_w^{-1}$.
\end{example}

\begin{definition}[UPC set] \label[d]{d:UPC}
    A subset $K \subseteq \R^n$ is \emph{UPC} if there exists a map $P : \ol K \times \R \to \R^n$ and 
    constants $A,B>0$ and $d\in \N_{\ge 1}$ such that, for each $x \in \ol K$, 
    \begin{itemize}
        \item $P(x,\cdot)$ is a polynomial curve of degree at most $d$;
        \item $P(x,0)=x$;
        \item $\dist(P(x,t), \R^n \setminus K) \ge A\, t^B$, for $t \in [0,1]$.
    \end{itemize}
\end{definition}

\begin{proposition}[{\cite[Lemma 3.8]{Pierzchala:2015aa}}]
    Each compact UPC set is admissible with characteristic $(\max\{d,\lceil B \rceil\},1,\de_1,\de_2,\de_3)$, 
    where $d \in \N_{\ge 1}$ and $B>0$ are the constants from \Cref{d:UPC}.
\end{proposition}

\begin{proof}
    The characteristic follows easily from the proof of Lemma 3.8 in \cite{Pierzchala:2015aa}.
\end{proof}

\begin{example}
    Each compact fat subanalytic set is UPC and hence admissible, by \cite{PawluckiPlesniak86}.
    More generally, each compact fat set that is definable in $\R_{\cC_M}$ is UPC and hence admissible, 
    where $\cC_M$ is a suitable quasianalytic Denjoy--Carleman class; see \Cref{e:RCM}. 
    This is due to \cite{Pierzchal-a:2005uk}. 
\end{example}    

\begin{example}
    The closure of a bounded Lipschitz domain is UPC and thus admissible with characteristic 
    $(1,1,\de_1,\de_2,\de_3)$. 
    More generally, the closure of a bounded H\"older domain with H\"older exponent $\al \in (0,1]$ 
    is UPC and thus admissible with characteristic 
    $(\lceil \al^{-1}\rceil,1,\de_1,\de_2,\de_3)$. Cf.\ \cite[Section 1.3]{Rainer:2024aa}. 
\end{example}

\subsection{Remez inequality on admissible sets} \label{ssec:RemezKL}

\begin{definition}[The constant $a_K$]
    Let $K \in \sK(\R^n)$. 
    Assume that $K$ has the characteristic $(d,q,\de_1,\de_2,\de_3)$. 
    Then we associate the positive number 
\begin{equation} \label{eq:aK}
    a_K :=
    \begin{cases}
        4\,\frac{(d+1)^{d+1}}{(d-1)^{d-1}} \cdot \max\{\diam(K),1\}, & \text{ if } d \ge 2,
        \\
        \diam(K), & \text{ if } d=1.
    \end{cases}
\end{equation}
\end{definition}

We are ready to state the general Remez inequality.

\begin{theorem} \label[t]{t:RemezKL}
    Let $L\subseteq \R^n$ be a convex body and $K \in \sK(\R^n)$ such that $K \subseteq L$. 
    There exists $\nu_K \in (0,1]$, depending only on $K$, such that the following holds.
    Let $\mu^{|\infty} = (\mu_j)_{j \ge 1}$ be an admissible weight satisfying $\mu_j\ge j$ for $j \ge 1$. 
    Let $M_0>b_0>0$.
    Let $m \in \N_{\ge 1}$.
    Assume that $(\mu^{|\infty},m\,M_0,a_K,b_0)$ is admissible and let 
    \[
        \nn = \nn(\mu^{|\infty},m\,M_0,a_K,b_0).
    \]
    Then, for each $(\mu^{|\nn},M_0)$-smooth map $f : L \to \R^m$ with $\|f\|_K\ge b_0$ and 
    each Lebesgue measurable set $E \subseteq K$ with $|E| > 0$, we have
\[
    \|f\|_K \le \Cnn\, \Big(\frac{m\, |K|^{\nu_K}}{|K|^{\nu_K}-(|K|-|E|)^{\nu_K}}\Big)^{\nn} \, \|f\|_E. 
\]
\end{theorem}

Clearly, \Cref{t:Remez_intro} is a direct consequence of \Cref{t:RemezKL}.
The proof of \Cref{t:RemezKL} comprises \Cref{sec:poly} and \Cref{sec:Remez_proof}. 

\begin{remark}
    It is possible to explicitly compute a number $\nu_K$ for which \Cref{t:RemezKL} holds;
    this requires the full characteristic $(d,q,\de_1,\de_2,\de_3)$ of $K$ as well as other 
    specifics of $K$ such as the ambient dimension $n$ and $|K|$.
\end{remark}

\begin{remark}
    \Cref{t:RemezKL} extends \Cref{thm:Remez}; in fact, if $K$ is a convex body we may 
    take $d=1$, $a_K=\diam(K)$, $\nu_K = 1/n$, and $\nn = \nn(\mu^{|\infty},M_0,a_K,b_0)$ (even if $m>1$).
\end{remark}

\begin{remark}
Noting that
\[
    \frac{(d+1)^{d+1}}{(d-1)^{d-1}} = \Big(\frac{d+1}{d-1}\Big)^{d-1} (d+1)^2 = \Big(1+ \frac{2}{d-1} \Big)^{d-1} (d+1)^2 
\]
we see that the growth of $a_K$ (as defined in \eqref{eq:aK}) is quadratic in $d$, 
\[
    a_K \sim 4 e^2 (d+1)^2 \cdot \max\{\diam(K),1\} \quad \text{ as } d \to \infty.
\]
\end{remark}

\begin{remark} 
    The proof will show that
    \Cref{t:RemezKL} remains valid in the following more general setup: 
    let $a>0$, $\nn = \nn (\mu^{|\infty},m\,M_0,a,b_0)$, and assume that 
    the maps $f : L \to \R^m$ with $\|f\|_K \ge b_0$ satisfy 
    \begin{itemize}
        \item $f|_\ell$ is $(a\, \mu^{|\nn},M_0)$-smooth for each line segment $\ell \subseteq L$;
        \item $f\o p$ is $(a\, \mu^{|\nn},M_0)$-smooth for each polynomial curve $p : [0,1] \to K$ of degree at most $d$.
    \end{itemize}
\end{remark}

\section{Bounds along polynomial curves} \label{sec:poly}

In this section, we establish bounds for the composition of a $(\mu^{|N},M_0)$-smooth function 
with a polynomial curve of degree at most $d$.

We will use the following corollary of the Bernstein--Walsh lemma; see e.g.\ \cite[Corollary 12.1.6]{RS02}.

\begin{lemma}[Bernstein's lemma]
    Let $p$ be a polynomial of degree at most $d$.
    For $r>1$, let $E_r \subseteq \C$ be the ellipse with foci $~\pm 1$ and semiaxes $\frac{1}{2}(r+r^{-1})$ 
    and $\frac{1}{2}(r-r^{-1})$.
    Then 
    \begin{equation} \label{eq:Bernstein}
    |p(z)| \le r^d \, \|p\|_{[-1,1]}, \quad z \in E_r.
    \end{equation}
\end{lemma}

The next proposition is the goal of this section.

\begin{proposition} \label[p]{p:comp}
Let $\mu^{|\infty}$ be an admissible weight satisfying
$\mu_j \ge j$, for $j \ge 1$,
and $M_0 >0$.
Let $K \subseteq \R^n$ be a compact nonempty set containing the origin. 
Let $f : K \to \R$ be a $(\mu^{|N},M_0)$-smooth function and $p : [-1,1] \to K$ a polynomial curve of degree at most $d$.
Then $f\o p : [-1,1] \to \R$ is $(a\, \mu^{|N},M_0)$-smooth, 
where
\begin{equation} \label{eq:a}
    a := \begin{cases}
        2\,\frac{(d+1)^{d+1}}{(d-1)^{d-1}} \cdot \max\{\diam(K),1\}, & \text{ if } d \ge 2,
        \\
        \frac{1}{2}\diam(K) , & \text{ if } d =1.
    \end{cases}
\end{equation}
\end{proposition}

\begin{proof}
    If $d=1$ then we may assume that $p(t) = \frac{1}{2}(x_0+x_1) + \frac{t}{2}(x_1-x_0)$, where $x_0,x_1 \in K$ 
    and the assertion follows easily.

    Let us assume $d \ge 2$. 
    Write $p = (p_1,\ldots,p_n)$.
    Since  $0 \in K$, we have
\begin{equation} \label{eq:diam}
   \|p_i\|_{[-1,1]} \le \max_{t \in [-1,1]} |p(t)| \le \diam(K), \quad 1 \le i \le n.
\end{equation}

Obviously, $\|f \o p\|_{[-1,1]} \le M_0$.
The derivative $(f \o p)^{(k)}(t)$, for $1 \le k \le N$ and $t \in [-1,1]$, coincides with the $k$-th derivative of 
\begin{equation} \label{eq:poly}
    z \mapsto \sum_{|\al|\le k} (\p^\al f)(p(t)) \frac{(p(z)-p(t))^\al}{\al!}
\end{equation}
at $z = t$. For fixed $t \in [-1,1]$, this is a polynomial in $z$ and we may assume that $z \in \C$.
By \eqref{eq:Bernstein} and \eqref{eq:diam}, 
we have 
\begin{align*}
    |(p(z)-p(t))^\al| &= \prod_{i=1}^n |p_i(z)-p_i(t)|^{\al_i} 
    \\
                      &\le \prod_{i=1}^n ((r^d+1) \|p_i\|_{[-1,1]})^{\al_i} \le (r^d+1)^{|\al|} \diam(K)^{|\al|}.
\end{align*}
Thus,
the modulus of the polynomial \eqref{eq:poly} can be estimated on $E_r$ by
\begin{align*}
  \MoveEqLeft  \sum_{|\al|\le k} \frac{1}{\al!} |\p^\al f(p(t))| (r^d+1)^{|\al|} \diam(K)^{|\al|}
\\  
    &= \sum_{j=0}^k \sum_{|\al| = j} \frac{1}{\al!} |\p^\al f(p(t))| (r^d+1)^{j} \diam(K)^{j}
    \\
    &\le \sum_{j=0}^k \frac{\|f\|_{j,K}}{j!} (r^d+1)^{j} \diam(K)^{j}
    \\
    &\le \sum_{j=0}^k \frac{M_j}{j!} ((r^d+1) \diam(K))^{j}
    \\
    &\le  \frac{M_k}{k!}\sum_{j=0}^k  ((r^d+1) \diam(K))^{j},
\end{align*}
where we used $\mu_j \ge j$ in the last step.

For each $t \in [-1,1]$ the ball $B(t,\de)$ with radius
\begin{equation}
    \de = \frac{1}{2}(r+r^{-1}) -1 = \frac{(r-1)^2}{2r}
\end{equation}
is contained in $E_r$. 
Thus, by Cauchy's estimates, we find
\begin{equation} \label{eq:Cauchy}
    |(f\o p)^{(k)}(t)| \le \frac{M_k}{\de^k} \sum_{j=0}^k  ((r^d+1) \diam(K))^{j},
\end{equation}
for $1 \le k \le N$ and $t \in [-1,1]$. 

Let us estimate 
\begin{equation} \label{eq:toest}
    \frac{1}{\de^k}  \sum_{j=0}^k  ((r^d+1) \diam(K))^{j} = \frac{2^k r^k}{(r-1)^{2k}} \sum_{j=0}^k  ((r^d+1) \diam(K))^{j}.
\end{equation}
We distinguish two cases:

(1) $\diam(K) < 1$: Then \eqref{eq:toest} is bounded by
\begin{align*}
    \frac{2^k r^k}{(r-1)^{2k}} \sum_{j=0}^k (2 r^d)^j &= \frac{2^k r^k}{(r-1)^{2k}} \frac{(2 r^d)^{k+1}-1}{2 r^d -1}
    \\
                                                           &\le \frac{2^k r^k}{(r-1)^{2k}} \frac{(2r^d)^{k+1}}{2 r^d -1} \le \Big(\frac{8 r^{d+1}}{(r-1)^{2}} \Big)^k,
\end{align*}
since $\frac{2r^d}{2 r^d -1} \le 2$. 
As $d \ge 2$, 
the minimum of $\frac{r^{d+1}}{(r-1)^{2}}$ for $r>1$ is attained at $r_0 = \frac{d+1}{d-1}$ and amounts to 
$\frac{1}{4} \frac{(d+1)^{d+1}}{(d-1)^{d-1}}$. Thus \eqref{eq:toest} is bounded by
\begin{equation}
  \Big(\frac{2(d+1)^{d+1}}{(d-1)^{d-1}} \Big)^k.  
\end{equation}

(2) $\diam(K) \ge 1$: In that case, $r^d \diam(K) \ge 1$ and hence \eqref{eq:toest} is bounded by 
\begin{align*}
    \frac{2^k r^k}{(r-1)^{2k}} \sum_{j=0}^k (2 r^d \diam(K))^j &= \frac{2^k r^k}{(r-1)^{2k}} \frac{(2r^d \diam(K))^{k+1}-1}{2 r^d \diam(K) -1}
    \\
                                                                 &\le \frac{2^kr^k}{(r-1)^{2k}} \frac{(2r^d \diam(K))^{k+1}}{2 r^d\diam(K) -1} 
                                                                 \\
                                                                 &\le \Big(\frac{8r^{d+1} \diam(K)}{(r-1)^{2}} \Big)^k,
\end{align*}
since $\frac{2r^d \diam(K)}{2 r^d \diam(K) -1} \le 2$. 
As in case (1), we conclude that \eqref{eq:toest} is bounded by
\begin{equation}
  \Big(\frac{2(d+1)^{d+1}}{(d-1)^{d-1}} \cdot \diam(K) \Big)^k.  
\end{equation}

In view of \eqref{eq:Cauchy}, we conclude that, for $1 \le k \le N$ and $t \in [-1,1]$,
\begin{equation}
    |(f\o p)^{(k)}(t)| \le a^k M_k, 
\end{equation}
where $a$ is given by \eqref{eq:a}. 
This completes the proof.
\end{proof}

\section{Proof of the Remez inequality on admissible sets} \label{sec:Remez_proof}

This section is dedicated to the proof of \Cref{t:RemezKL}.

\subsection{Preparatory results}

We will make crucial use of the following result on admissible sets, due to \cite{Pierzchala:2015aa}.

\begin{proposition}[{\cite[Corollary 9.11]{Pierzchala:2015aa}}] \label[p]{l:P_9.11}
    Let $K \in \sK(\R^n)$ be an admissible set with admissible collection $(K_w,a_w,\Up_w)_{w \in K}$.
    There exist constants $\si, \et>0$, and $\vr \in [0,1)$
    such that for each Lebesgue measurable $E \subseteq K$ and each $w \in  K$,
    \begin{equation}
        \frac{|\Up_w^{-1}(E) \cap K_w|}{|K_w|} \ge \frac{1}{\et (1-\la)^{1/\si} +1}
        \quad \text{ if } \la := \frac{|E|}{|K|} > \vr.
    \end{equation}
\end{proposition}

The following lemma is a slight modification of \cite[Lemma 8.2]{Pierzchala:2015aa}.

\begin{lemma} \label[l]{l:comp}
    Let $a,b,c >0$. Consider the functions $\vh : (- \infty, 1] \to (0,1]$ and $\ps : (0,1]  \to [1,\infty)$ given by 
    \begin{equation}
        \vh(t) = \frac{1}{a(1-t)^b + 1} \quad \text{ and } \quad
        \ps(t) = \frac{1}{1-(1-t)^c}.
    \end{equation}
    There exists $d = d(a,b,c) >0$ such that 
    \begin{equation}
        (\ps \o \vh)(t) \le \frac{1}{1-(1-t)^d}, \quad t \in (0,1].
    \end{equation}
\end{lemma}

\begin{proof}
    We claim that there exists $r = r(a) \in (0,1]$ such that 
    \begin{equation} \label{eq:claim}
        as \le (1+as)s^r, \quad \text{ for all } s \ge 0.
    \end{equation}
    Therefore, setting $s = (1-t)^b$,  
    \[
        \vh(t) \ge 1-(1-t)^{br}, \quad t \in (-\infty,1], 
    \]
    and,
    consequently, as $\ps$ is decreasing,
    \[
        (\ps \o \vh)(t) \le \ps\big(1-(1-t)^{br}\big) = \frac{1}{1-(1-t)^{bcr}}
    \]
    for $t \in (0,1]$.

    It remains to prove \eqref{eq:claim}. If $a \in (0,1]$, it suffices to set $r:= 1$.
    Let us assume that $a>1$. The inequality is always true for $s=0$ and $s \ge 1$. So let $s \in (0,1)$.
    By substituting $u=s^r$, \eqref{eq:claim} is equivalent to 
    \[
        u^{\frac{1-r}r} (1-u) \le \frac{1}a. 
    \]
    The maximum of the left-hand side for $u \in (0,1)$ as attained at $u=1-r$. Hence, for $a>1$, letting 
    $r$ be the unique solution of 
    \[
        r(1-r)^{\frac{1-r}r} = \frac{1}a 
    \]
    in the interval $(0,1)$,
    implies \eqref{eq:claim}.
    (Such $r$ exists because the left-hand side defines a strictly increasing surjective function $(0,1) \to (0,1)$.)
\end{proof}

\subsection{A preliminary Remez inequality}

The following Remez inequality holds if the ratio $|E|/|K|$ is not too small.

\begin{theorem} \label[t]{t:RemezK}
    Let $K \in \sK(\R^n)$. 
    Then there exist $\nu \in (0,1]$ and $\vr \in [0,1)$ such that the following holds.
    Let $\mu^{|\infty} = (\mu_j)_{j \ge 1}$ be an admissible weight satisfying $\mu_j\ge j$ for $j \ge 1$. 
    Let $M_0 >b_0>0$.
    Let $m \in \N_{\ge 1}$.
Assume that $(\mu^{|\infty},m\, M_0,a_K,b_0)$ is admissible and let $\nn = \nn(\mu^{|\infty},m\, M_0,a_K,b_0)$.
Let $f : K \to \R^m$ be $(\mu^{|\nn},M_0)$-smooth such that $\|f\|_K \ge b_0$. 
Then for each Lebesgue measurable set $E \subseteq K$ 
we have
\[
    \|f\|_K \le \Cnn\, \Big(\frac{m}{1-(1-\la)^\nu}\Big)^{\nn} \, \|f\|_E \quad \text{ if } \la := \frac{|E|}{|K|} > \vr.
\]
\end{theorem}

\begin{proof}
    We may assume that $K$ contains the origin; indeed, the statement of \Cref{t:RemezK} is invariant under translations.
  Let $(K_w,a_w,\Up_w)_{w \in K}$ be an admissible collection for $K$ in $\R^n$ with 
  characteristic $(d,q,\de_1,\de_2,\de_3)$. 
  Let $\si,\et>0$ and $\vr\in [0,1)$ be the constants provided by \Cref{l:P_9.11}.
Assume that $E \subseteq K$ is Lebesgue measurable and $\la := \frac{|E|}{|K|} > \vr$.

Fix $w \in K$ such that $|f(w)|_1 = \|f\|_K$.
The preimage $\Up_w^{-1}(E) \subseteq [0,\infty)^n$ is Lebesgue measurable 
(indeed, $E$ is the union of a $F_\si$-set and a set 
of measure zero and the preimage under $\Up_w$ of the latter is again of measure zero).

By \Cref{l:ga}, there exists an affine curve $\ga_w : [0,1] \to K_w$ such that $\ga_w(0)=a_w$ 
and $\ga_w^{-1}(\Up_w^{-1}(E) \cap K_w)$ contains a Lebesgue measurable subset $F_0$ satisfying 
\begin{equation}
    \cL^1(F_0) \ge 1 - \Big(1 - \frac{|\Up_w^{-1}(E) \cap K_w|}{|K_w|} \Big)^{1/n}.
\end{equation}
By \Cref{l:P_9.11}, we conclude that 
\begin{equation} \label{eq:admkey}
    \cL^1(F_0) \ge 1 - \Big(1 -  \frac{1}{\et (1-\la)^{1/\si} +1}\Big)^{1/n}.
\end{equation}

Consider the function $G_q : [0,1] \to [0,1]$, $G_q(t) = t^q$, and $F_1 := G_q^{-1}(F_0)$. Then 
(by \cite[Theorem 1.13]{LiebLoss01} and \cite[Lemma 7.1]{Pierzchala:2015aa})
\[
    \cL^1(F_0) \le 1 - (1-\cL^1(F_1))^q
\]
and, hence, using \eqref{eq:admkey}, 
\begin{align*}
    \cL^1(F_1) \ge 1 - (1- \cL^1(F_0))^{1/q} \ge 1 -  \Big(1 -  \frac{1}{\et (1-\la)^{1/\si} +1}\Big)^{1/(nq)}.
\end{align*}
Thus, by \Cref{l:comp}, there exists $\nu >0$ such that 
\begin{equation} \label{eq:F1}
    \frac{1}{\cL^1(F_1)} 
    \le \frac{1}{1 -  (1 - \la)^\nu}.
\end{equation}
Replacing $\nu$ by $\min\{\nu,1\}$, we may assume that $\nu \in (0,1]$.
    
The composite $p := \Up_w \o \ga_w \o G_q : [0,1] \to K$ is a polynomial curve of degree at most $d$, 
by \Cref{d:admcoll}(4)
(cf.\ \cite[Lemma 9.6]{Pierzchala:2015aa}). 

We claim that $f \o p : [0,1] \to \R^m$ is $(a_K \, \mu|^{\nn},m\, M_0)$-smooth.
Let $\ell : [-1,1] \to [0,1]$ be given by $\ell(t) = \frac{1}2 (t+1)$.
Then $p \o \ell : [-1,1] \to K$ is a polynomial curve of degree at most $d$. 
Since $\|f\|_K  = |f(w)|_1 = |f(p(0))|_1$,
we have 
\begin{equation}
   b_0 \le  \|f \o p\|_{[0,1]} = \|f\|_K \le M_0. 
\end{equation}
By \Cref{p:comp},
for $1 \le k \le  \nn$, 
\begin{align*}
    \|(f\o p)^{(k)}\|_{[0,1]} &= \sup_{x \in [0,1]} \sum_{i=1}^m |(f_i\o p)^{(k)}(x)| 
    = \sup_{x \in [0,1]} \sum_{i=1}^m |(f_i\o p\o \ell \o \ell^{-1})^{(k)}(x)| 
    \\
&= 2^k \sup_{x \in [-1,1]} \sum_{i=1}^m |(f_i\o p\o \ell)^{(k)}(x)| 
\le m\,M_0\, a_K^k \, \mu_1 \mu_2 \cdots \mu_k,
\end{align*}
where $a_K$ is defined in \eqref{eq:aK}.
This implies the claim.

In view of \Cref{r:nnprop}(1),
we may apply \Cref{thm:NSV} to $f \o p : [0,1] \to \R^m$ and $F_1 \subseteq [0,1]$. 
Therefore, 
\begin{align*}
    \|f\|_K = \|f \o p\|_{[0,1]} &\le \Cnn \, \Big(\frac{m}{\cL^1(F_1)}\Big)^{\nn} \, \|f \o p\|_{F_1}.
\end{align*}
By \eqref{eq:F1} and $f(p(F_1)) = f(\Up_w(\ga_w(F_0))) \subseteq f(\Up_w( \Up_w^{-1}(E) \cap K_w )) \subseteq f(E)$,
the proof is complete.
\end{proof}

\subsection{Proof of \texorpdfstring{\Cref{t:RemezKL}}{Proof of Theorem 3.10}}

Let $L\subseteq \R^n$ be a convex body and
$K \in \sK(\R^n)$
such that $K \subseteq L$. 
Let $\mu^{|\infty} = (\mu_j)_{j \ge 1}$ be an admissible weight satisfying $\mu_j\ge j$ for $j \ge 1$. 
Let $M_0 > b_0>0$. Let $m \in \N_{\ge 1}$.
Assume that $(\mu^{|\infty},m\,M_0,a_K,b_0)$ is admissible and let 
$\nn = \nn(\mu^{|\infty},m\, M_0,a_K,b_0)$. 
Let $f : L \to \R^m$ be $(\mu^{|\nn},M_0)$-smooth such that $\|f\|_K\ge b_0$.
Let $E \subseteq K$ be Lebesgue measurable with $|E| > 0$ and set $\la := \frac{|E|}{|K|}$. 

We may assume without loss of generality that $L$ is the convex hull of $K$ and hence $\diam(L)=\diam(K)$. 
Indeed the convex hull $\on{co}(K)$ of $K$ is contained in $L$ and the restriction $f|_{\on{co}(K)}$ 
is $(\mu^{|\nn},M_0)$-smooth.

Let $\nu\in (0,1]$ and $\vr \in [0,1)$ be the constants associated with $K$ in \Cref{t:RemezK}.
Assume that $\vr>0$ (otherwise we are done).
If $\la \in (\vr,1]$, then the statement follows from \Cref{t:RemezK}.
Suppose that $\la \in (0,\vr]$.

Set $\mu := \frac{|K|}{|L|} \in (0,1]$.
There is a unique $\ep_0 \in (0,1]$ such that $1-\mu \vr = (1-\vr)^{\ep_0}$.
Thus $u(t) := (1-t)^{\ep_0} -1 +\mu t$ vanishes at $t=0$ and $t=\vr$ and is concave in $[0,\vr]$ 
since $u''(t) = \ep_0 (\ep_0-1) (1-t)^{\ep_0-2}\le 0$.
It follows that 
\begin{equation} \label{eq:ep}
    (1- \mu t) \le (1-t)^{\ep_0} \quad \text{ for } t \in [0,\vr].
\end{equation}

Since $\diam(L) = \diam(K) \le a_K$ (see \eqref{eq:aK}), $(\mu^{|\infty},m\,M_0,\diam(L),b)$ is admissible and
\[
    \nn' = \nn(\mu^{|\infty},m\,M_0,\diam(L),b) \le \nn(\mu^{|\infty},m\,M_0,a_K,b) =\nn
\]
as well as $\mathbf C_{\nn'} \le \Cnn$ (see \Cref{r:nnprop}(2)).
Thus, by \Cref{thm:Remez} and \eqref{eq:ep}, 
\begin{align*}
    \|f\|_K \le \|f\|_L 
    &\le \mathbf C_{\nn'}\, \Big( \frac{m}{1-(1-\mu\la)^{1/n}}\Big)^{\nn'}\, \|f\|_E
    \le \Cnn \, \Big( \frac{m}{1-(1-\la)^{\ep_0/n}}\Big)^{\nn}\, \|f\|_E.
\end{align*}

It follows that \Cref{t:RemezKL} holds with $\nu_K := \min\{\nu,\ep_0/n\}$,
which depends only on $K$, because we assumed that $L= \on{co}(K)$.

\section{Sublevel sets and applications} \label{sec:sublevel}

In this section, we deduce from \Cref{t:RemezKL} a precise bound for the volume growth of sublevel sets
and several applications.

For the rest of the paper, let $a_K$ be defined by \eqref{eq:aK} and let $\nu_K \in (0,1)$ be the constant provided by  
\Cref{t:RemezKL} with the convention that $a_K = \diam(K)$ and $\nu_K = 1/n$, in the case that 
$K \subseteq \R^n$ is a convex body.

The following setup applies throughout this section.

\subsection{Setup} \label{ssec:setup}

Let $L\subseteq \R^n$ be a convex body and
$K \in \sK(\R^n)$ 
such that $K \subseteq L$. 

Let $\mu^{|\infty} = (\mu_j)_{j \ge 1}$ be an admissible weight satisfying $\mu_j\ge j$ for $j \ge 1$. 
Let $M_0>b_0>0$.
Assume that $(\mu^{|\infty},M_0,a_K,b_0)$ is admissible and let 
\[
\nn = \nn(\mu^{|\infty},M_0,a_K,b_0).
\]
Let $f : L \to \R$ be a $(\mu^{|\nn},M_0)$-smooth function  satisfying $\|f\|_K\ge b_0$.
We only consider single-valued functions in this section (i.e.\ $m=1$).

\begin{remark} \label[r]{r:convexbody}
    Most of the results in this section were obtained in \cite{Rainer:2022aa} in the special case that $K \subseteq \R^n$ 
    is a convex body. In that case, the condition $\mu_j \ge j$ for $j \ge 1$ is not needed.
\end{remark}

\subsection{Sublevel sets}

Let $K_t := \{x \in K: |f(x)| \le t\}$ denote the sublevel set of $f|_K$.

\begin{theorem} \label[t]{t:sublevel}
    We have
    \begin{equation}
        |K_t| \le \frac{|K|}{\nu_K} \, \Big(\frac{\Cnn \,t}{\|f\|_K}\Big)^{1/\nn}, \quad t>0.
    \end{equation}
\end{theorem}

\begin{proof}
    We may assume that $\la_t := \frac{|K_t|}{|K|}>0$.
    By \Cref{t:RemezKL}, 
    \begin{align*}
        \|f\|_K \le  \Cnn\, \Big(\frac{1}{1-(1-\la_t)^{\nu_K}}\Big)^{\nn} \, t
    \end{align*}
    and hence
    \[
        1-(1-\la_t)^{\nu_K} \le \Big(\frac{\Cnn \,t}{\|f\|_K}\Big)^{1/\nn}.
    \]
    So the assertion follows from the fact that $\nu\,s \le 1-(1-s)^\nu$, for $s\in [0,1]$ and $\nu \in (0,1]$. 
\end{proof}

\subsection{Comparison of \texorpdfstring{$L^p$}{Lp}-norms}

Recall that the \emph{distribution function} $d_f$ of 
a Borel measurable function $f : \R^n \supseteq X \to \R$ is defined by 
\[
d_f(t) :=  |\{x \in X : |f(x)|>t\}| 
\]
and the \emph{decreasing rearrangement} $f^*$ of $f$ by 
\[
    f^*(y) := \inf \{t>0 : d_f(t) \le y\}.
\]

\begin{lemma} \label[l]{l:rearrange}
    Let $E \subseteq K$ be a Lebesgue measurable set with $|E|>0$.
    Then
    \begin{equation} \label{eq:rearrange}
        (f|_E)^*(|E|\,\la) \ge \frac{\|f\|_K}{\Cnn} \Big(\frac{\nu_K\,|E|\, (1-\la)}{|K|}\Big)^{\nn}, 
        \quad \la \in (0,1).
    \end{equation}
\end{lemma}

\begin{proof}
    For $E_t := \{x \in E : |f(x)| \le t\}$,
    we have 
    \[
        d_{f|_E}(t) = |E|\, \la_t, \quad \text{ where } \la_t := 1- \frac{|E_t|}{|E|}.
    \]
    Let $s_\la$ denote the right-hand side of \eqref{eq:rearrange}.
    If $t \in (0,s_\la)$ then, by \Cref{t:sublevel}, 
    \begin{align*}
        |E| (1-\la_t) =   |E_t| \le \frac{|K|}{\nu_K} \, \Big(\frac{\Cnn \,t}{\|f\|_K}\Big)^{1/\nn} 
        < \frac{|K|}{\nu_K} \, \Big(\frac{\Cnn \,s_\la}{\|f\|_K}\Big)^{1/\nn} = |E| (1-\la),
    \end{align*}
    implying that $d_{f|_E}(t) = |E|\, \la_t > |E|\, \la$, for $t \in (0,s_\la)$. 
    Thus, $(f|_E)^*(|E|\,\la) \ge s_\la$ and \eqref{eq:rearrange} is proved.
\end{proof}

Let us write
\begin{align*}
  \|f\|_{L^p(E)}^\sharp &:= \Big( \frac{1}{|E|} \int_E |f(x)|^p \, dx \Big)^{1/p}, \quad 0 < p <\infty,
  \\
  \|f\|_{L^\infty(E)}^\sharp &:= \esssup_{E}|f|,
\end{align*}
for the normalized $L^p$-norms (or quasinorms if $0<p<1$) of $f$ on a Lebesgue measurable set $E$ with $0<|E|<\infty$.
Then, as a consequence of H\"older's inequality,
\[
  \|f\|_{L^q(E)}^\sharp \le \|f\|_{L^p(E)}^\sharp, \quad \text{ if } 0< q \le p \le \infty.
\]

In the setup of \Cref{ssec:setup}, also opposite inequalities hold.

\begin{theorem} \label[t]{t:comparison}
   Let $E \subseteq K$ be a Lebesgue measurable subset with $|E|>0$.
  Then, for all $0< q< p \le \infty$,
  \begin{equation} \label{eq:LpLqmaster}
      \|f\|_{L^p(K)}^\sharp \le \Big(\Cnn \Big(\frac{|K|}{\nu_K \,|E|}\Big)^{\nn} \Big)^{1-\frac{q}{p}}
        (q\, \nn+1)^{\frac{1}{q}-\frac{1}{p}}  
        (\|f\|_{L^q(K)}^\sharp)^{\frac{q}p} (\|f\|_{L^q(E)}^\sharp)^{1-\frac{q}{p}}.
    \end{equation}
  In particular, 
  \begin{equation} \label{eq:LpLqK}
      \|f\|_{L^p(K)}^\sharp \le \Big(\frac{\Cnn}{\nu_K^{\nn}} \Big)^{1-\frac{q}{p}}
        (q\, \nn+1)^{\frac{1}{q}-\frac{1}{p}}  
        \|f\|_{L^q(K)}^\sharp
  \end{equation}
  and
 \begin{equation} \label{eq:LpLqE}
       \|f\|_{L^p(K)}^\sharp \le  \Cnn \Big(\frac{|K|}{\nu_K \,|E|}\Big)^{\nn} \,
       (q\, \nn+1)^{\frac{1}{q}}\,  \|f\|_{L^q(E)}^\sharp.
    \end{equation}
\end{theorem}

\begin{proof}
    Both \eqref{eq:LpLqK} and \eqref{eq:LpLqE} are simple consequences of \eqref{eq:LpLqmaster}:
    it suffices to specialize to $E=K$ 
    and to use $\|f\|_{L^q(K)}^\sharp \le \|f\|_{L^p(K)}^\sharp$, respectively.

    Let us prove \eqref{eq:LpLqmaster}.
    By \Cref{l:rearrange},  
    \begin{align} 
        \frac{1}{|E|} \int_{E} |f(x)|^q \, dx &= \frac{1}{|E|} \int_0^{|E|} \big((f|_E)^*(y)\big)^q \, dy
       =  \int_0^1 \big((f|_E)^*(|E|\, \la)\big)^q \, d\la
       \\
                                              &\ge   \frac{\|f\|_K^q}{\Cnn^q}  \Big(\frac{\nu_K\, |E|}{ |K|}\Big)^{q \nn}\int_0^{1} (1-\la)^{q \nn} \, d\la 
       \\ \label{eq:compc}
                               &= \frac{\|f\|_K^q}{\Cnn^q}  \Big(\frac{\nu_K\, |E|}{ |K|}\Big)^{q \nn} \frac{1}{q \nn+1}
    \end{align}
    and thus 
    \begin{equation} \label{eq:LinftyLqmaster}
       \|f\|_K \le \Cnn \Big(\frac{ |K|}{\nu_K\,|E|}\Big)^{\nn}\,
        (q\, \nn+1)^{\frac{1}{q}}  \,
        \|f\|_{L^q(E)}^\sharp
    \end{equation}
    which is \eqref{eq:LpLqmaster} for $p=\infty$.

    If $0< q < p < \infty$, then \eqref{eq:LinftyLqmaster} implies
     \begin{align*}
     \MoveEqLeft 
     \frac{1}{|K|} \int_{K} |f(x)|^p \,dx
        \le  \|f\|_{K}^{p-q} \frac{1}{|K|}\int_{K} |f(x)|^q \,dx = \|f\|_{K}^{p-q}\, (\|f\|_{L^q(K)}^\sharp)^q
        \\
        &\le \Big(\Cnn \Big(\frac{|K|}{\nu_K\, |E|}\Big)^{\nn}
        (q\, \nn+1)^{\frac{1}{q}}\Big)^{p-q} \, (\|f\|_{L^q(K)}^\sharp)^q (\|f\|_{L^q(E)}^\sharp)^{p-q}
     \end{align*}
     from which \eqref{eq:LpLqmaster} follows easily.
\end{proof}

\subsection{Integrability of \texorpdfstring{$|f|^{-\al}$}{|f|{-alpha}}}

\begin{theorem} \label[t]{t:integrability}
   The integral 
   \[
       \int_K |f(x)|^{-\al} \, dx
   \]
   converges, provided that $0<\al < 1/\nn$.
\end{theorem}

\begin{proof}
   By \cite[Theorem 4.6]{Rainer:2022aa}, the zero set $Z_f$ of $f : L \to \R$ has Lebesgue measure zero.
   Thus $K \ni x \mapsto |f(x)|^{-\al}$ is defined almost everywhere and measurable. 
   A computation similar to \eqref{eq:compc} (using \Cref{l:rearrange} for $E=K$) shows
   \begin{align*}
       \int_K |f(x)|^{-\al} \, dx 
       &= |K|  \int_0^1 \big((f|_K)^*(|K|\, \la)\big)^{-\al} \, d\la 
       \\
       &\le |K|\, \Big(  \frac{\|f\|_K \, \nu_K^\nn}{\Cnn} \Big)^{-\al}  \int_0^{1} (1-\la)^{-\al \nn} \, d\la.
   \end{align*}
   The integral on the right-hand side converges precisely if $\al \nn < 1$. 
\end{proof}

\subsection{On the mean oscillation of \texorpdfstring{$\log |f|$}{log |f|}}

Recall that the \emph{mean oscillation} of a locally integrable function $g : \R^n \supseteq X \to \R$ 
over a Lebesgue measurable set $E \subseteq X$ with $0<|E|<\infty$ is given by
\[
  \on{mo}_E(g) := \frac{1}{|E|} \int_E |g(x)-g_E|\, dx,
\]
where $g_E$ is the average of $g$ over $E$,
\[
  g_E := \frac{1}{|E|} \int_E g(x)\, dx.
\]
In the Setup \ref{ssec:setup}, we have the following bound for the mean oscillation of $\log |f|$.

\begin{corollary} \label[c]{cor:Remez6}
  For each Lebesgue measurable set $E \subseteq K$ with $|E|>0$,
  \begin{equation} \label{eq:mo}
     \on{mo}_E(\log |f|) \le
        2\log \Cnn + 2\nn \, \Big(1+ \log\Big(\frac{|K|}{\nu_K\, |E|}\Big)\Big).
  \end{equation}
\end{corollary}

\begin{proof}
Observe that
\begin{equation} \label{eq:mo1}
  \on{mo}_E(\log |f|) \le  \frac{2}{|E|} \int_E \Big|\log \frac{|f(x)|}{\|f\|_K}\Big| \, dx.
\end{equation}
Indeed,
\begin{align*}
   \Big| \log |f| - \frac{1}{|E|} \int_E \log |f|\, dx \Big| 
   &\le \Big| \log |f| - \log \|f\|_K \Big| + \Big| \log \|f\|_K  - \frac{1}{|E|} \int_E \log |f|\, dx \Big|
   \\
   &= \Big| \log \frac{|f|}{\|f\|_K} \Big| + \Big| \frac{1}{|E|} \int_E \log \|f\|_K  -  \log |f|\, dx \Big|
   \\
   &\le \Big| \log \frac{|f|}{\|f\|_K} \Big| + \frac{1}{|E|} \int_E \Big| \log \frac{|f|}{\|f\|_K} \Big| \, dx.
\end{align*}
By \Cref{l:rearrange},
\begin{align}
   \frac{1}{|E|}\int_E \Big|\log \frac{|f(x)|}{\|f\|_K}\Big| \, dx
   &= \frac{1}{|E|} \int_0^{|E|} \Big|\log \frac{(f|_E)^*(y)}{\|f\|_K}\Big| \, dy
    =  \int_0^{1} \Big|\log \frac{(f|_E)^*(|E|\, \la)}{\|f\|_K}\Big| \, d\la
    \\
    &\le \int_0^1  -\log\Big(\frac{1}{\Cnn}\Big) - \nn \, \log \Big( \frac{\nu_K\, |E|\,(1-\la)}{|K|}\Big) \, d\la
    \\
    &= \log \Cnn + \nn \, \Big( 1 +  \log\Big(\frac{|K|}{\nu_K\, |E|}\Big)\Big) \label{eq:compR6}
\end{align}
and \eqref{eq:mo} follows.
\end{proof}

\Cref{cor:Remez6} provides an explicit bound on how fast $\on{mo}_E(\log |f|)$ can grow 
as $|E|$ tends to zero. In the following corollary, we derive a different bound, where 
the dependence is integrated in the integer $\nn$. We only consider complex bodies and balls therein, for simplicity.

\begin{corollary} \label[c]{c:nnB}
    Let $K\subseteq \R^n$ be a convex body. 
    Let $\mu^{|\infty} = (\mu_j)_{j \ge 1}$ be an admissible quasianalytic weight. 
    Let $f : K \to \R$ be a $(\mu^{|\infty},M_0)$-smooth function which is not identically zero and where $M_0>0$ is such that $M_0 > \|f\|_K$. Then, 
    for each ball $B \subseteq K$,
    \begin{equation}
        \on{mo}_B(\log |f|) \le 2 \log \mathbf{C}_{\nn_B} + 2(1+ \log n)\, \nn_B,
    \end{equation}
    where $\nn_B := \nn(\mu^{|\infty},M_0,\diam(K),\|f\|_B)$.
\end{corollary}

\begin{proof}
    Since $M_0 > \|f\|_K$, we have $\nn_B \in \N_{\ge 2}$, by \Cref{r:Nnonzero}, 
    and thus the constant $\mathbf{C}_{\nn_B}>0$ is well-defined. By \Cref{l:rearrange} (applied to $E=K =L = B$), 
    \[
        (f|_B)^*(|B|\, \la) \ge \frac{\|f\|_B}{\mathbf{C}_{\nn_B}} \Big(\frac{1-\la}{n}\Big)^{\nn_B}.
    \]
    Then, the computation \eqref{eq:compR6} 
    implies the assertion.
\end{proof}

\begin{remark} \label[r]{r:nnB}
    We refer to \Cref{p:an}, \Cref{p:Denjoy1}, and \Cref{p:Denjoys} for 
   precise growth estimates of $\nn(\mu^{|\infty},1,a,b)$ as a function of $a$ and $b$, 
   for particular admissible weights $\mu^{|\infty}$. 
   For instance, in the case $\mu_{\on{an}}^{|\infty} = (j)_{j\ge 1}$, 
   $\nn_B$ is of the same order as
   \[
         \Big\lceil \log \frac{M_0}{\|f\|_B}\Big\rceil_\N.
   \]
\end{remark}

We cannot infer that $\log |f|$ is BMO, but we have the following.

\begin{corollary}
    Let $K\subseteq \R^n$ be a convex body. 
    Let $\mu^{|\infty} = (\mu_j)_{j \ge 1}$ be an admissible quasianalytic weight. 
    Let $M_0> b_0>0$.
    Let $f : K \to \R$ be a $(\mu^{|\infty},M_0)$-smooth function such that $M_0 > \|f\|_K$. Then, 
    for each ball $B \subseteq K$ satisfying $\|f\|_B \ge b_0$,
    \begin{equation}
        \on{mo}_B(\log |f|) \le 2 \log \Cnn + 2 (1+\log n) \, \nn,
    \end{equation}
    where $\nn := \nn(\mu^{|\infty},M_0,\diam(K),b_0)$.
    In particular, for each $b_0>0$,
    the restriction of $\log |f|$ to the set $\{x \in K : |f(x)|\ge b_0\}$ is BMO.
\end{corollary}

\begin{proof}
    Since $\nn(\mu^{|\infty},M_0,\diam(K),\|f\|_B) \le \nn(\mu^{|\infty},M_0,\diam(K),b_0)$,
    this is a consequence of \Cref{c:nnB}.
\end{proof}

\begin{remark} \label[r]{r:BMO}
    A locally integrable function $g : \R^n\to \R$ is BMO if and only if $g= C_1 \, \log \om$ 
    for some constant $C_1$ and some nonnegative $\om : \R^n \to \R$ that satisfies a reverse H\"older inequality
    in the sense that there exist $p=p(\om)>1$ and $C_2 = C_2(\om)>0$ such that 
    \[
        \|\om\|^\sharp_{L^p(B)} \le C_2 \, \|\om\|^\sharp_{L^1(B)},
    \]
    for all balls $B$. See \cite[Chapter V]{Stein93}. In our reverse H\"older inequality \eqref{eq:LpLqK} for $K=B$,
    the constant depends on $B$ (through $\nn$).
\end{remark}

\section{Inequalities of {\L}ojasiewicz type} \label{sec:Lojasiewicz}

The goal of this section is to establish several inequalities of {\L}ojasiewicz type with explicit constants, which 
depend on $\nn$ and the geometry of the underlying admissible set $K$.

\subsection{Thickness of a set}

Let us recall the definition of thickness,
following \cite[Definition 1.16]{Pierzchala:2022aa}.

\begin{definition}[Thick sets]
Let $K \subseteq \R^n$ be a Lebesgue measurable nonempty set.
We say that $K$ is \emph{$(c_K,\th_K)$-thick} if there exist constants 
$c_K,\th_K > 0$ such that 
\begin{equation}
    |K \cap B(x,r)| \ge c_K \,  r^{\th_K}
\end{equation}
for all $x \in K$ and $0 < r \le 1$.    
\end{definition}

Every compact fat nonempty $K \subset \R^n$ that is definable in a polynomially bounded o-minimal 
expansion of the real field is \emph{$(c_K,\th_K)$-thick} for some $c_K,\th_K > 0$. This follows 
easily from the {\L}ojasiewicz inequality that is available for polynomially bounded o-minimal structures; 
see e.g.\ the proof of \cite[Theorem 1.18]{Pierzchala:2022aa}.

\begin{remark} \label[r]{r:cbthick}
    Each convex body $K \subseteq \R^n$ is $(c_K,\th_K)$-thick with 
    \begin{equation}
        \th_K = n \quad \text{ and } \quad c_k = \om_n \Big(\frac{\rh}{H}\Big)^n,  
    \end{equation}
    where $H := \max\{ \diam(K),1\}$ and $\rh>0$ is the radius of the largest ball contained in $K$.
    Indeed, there exists a ball $B(x_0,\rh) \subseteq K$.
    For given $x \in K$ and $0<r\le 1$, consider the ball $B:= \{x+\frac{r}{H}(y-x) : y \in B(x_0,\rh) \}$.
    Then $B \subseteq K$, by convexity, and $B \subseteq B(x,r)$ since 
    $\tfrac{r}{H}|y-x| \le  \tfrac{r}{H} \diam(K) \le r$. It follows that 
    \[
     |K \cap B(x,r)| \ge |B| =\om_n (\tfrac{\rh}{H})^n \, r^n.
    \]
\end{remark}

\subsection{Inequalities of {\L}ojasiewicz type with explicit constants}

The following theorem is formulated for vector-valued maps, since among other things we will apply it to the gradient.

\begin{theorem} \label[t]{t:Lojasiewicz}
Let $L\subseteq \R^n$ be a convex body and
$K \in \sK(\R^n)$ 
with $K \subseteq L$. 
Assume that $K$ is $(c_K,\th_K)$-thick.
Let $\mu^{|\infty} = (\mu_j)_{j \ge 1}$ be an admissible weight satisfying $\mu_j \ge j$ for $j \ge 1$. 
    Let $M_0>b_0>0$.
    Let $m \in \N_{\ge 1}$.
    Assume that $(\mu^{|\infty},m\, M_0,a_K,b_0)$ is admissible and let 
    \[
        \nn = \nn(\mu^{|\infty},m\,M_0,a_K,b_0).
    \]
    Let $f : L \to \R^m$ be a $(\mu^{|\nn},M_0)$-smooth map 
    satisfying $\|f\|_K\ge b_0$.
    For any function $g : K \to [-1,1]$, 
    we have 
    \begin{equation} \label{eq:L}
        \|f\|_{B(x,\frac{|g(x)|}2) \cap K} \ge \frac{\|f\|_K}{\Cnn}\,\Big(\frac{c_K\, \nu_K}{m\,2^{\th_K} \, |K|}\Big)^{\nn}\,  |g(x)|^{\th_K \nn},
    \end{equation}
    for all $x \in K$. 
\end{theorem}

Note that we do not require that $f|_K^{-1}(0) \subseteq g^{-1}(0)$.

\begin{proof}
    If $x \in g^{-1}(0)$, then \eqref{eq:L} holds trivially; in that case, the left-hand side of \eqref{eq:L} 
    equals $|f(x)|_1$.
    Fix $x \in K \setminus g^{-1}(0)$ and set 
    $B :=  B(x,\frac{r}{2}) \cap K$ for $r:= |g(x)|$. Note that $r \in (0,1]$.

    We formulated \Cref{t:sublevel} for single-valued function, but it remains true in the vector-valued setting 
    with the same proof. Thus, for $t>0$,
    \begin{equation} \label{eq:L2}
        |K_t|\le  \frac{m\,|K|}{\nu_K} \, \Big(\frac{\Cnn \,t}{\|f\|_K}\Big)^{1/\nn},
    \end{equation}
    where $K_t := \{x\in K : |f(x)|_1 \le t\}$.
    
    Let $T$ be the set of all $t>0$ such that 
    \begin{equation} \label{eq:L3}
        \frac{m\, |K|}{\nu_K} \, \Big(\frac{\Cnn \,t}{\|f\|_K}\Big)^{1/\nn} < c_K \Big(\frac{r}{2}\Big)^{\th_K}.
    \end{equation}
    That means $t \in T$ if and only if 
    \begin{equation} \label{eq:L4}
        0<  t <\frac{\|f\|_K}{\Cnn} \, \Big(\frac{c_K\, \nu_K}{m\,2^{\th_K}\, |K|}\Big)^{\nn} \, r^{\th_K \nn}.
    \end{equation}
    Since $|B| \ge c_K (\frac{r}{2})^{\th_K}$ because $K$ is $(c_K,\th_K)$-thick, 
    we have $B\not \subseteq K_t$ and hence $\|f\|_B > t$
    for all $t \in T$, by \eqref{eq:L2} and \eqref{eq:L3}.
    In view of \eqref{eq:L4}, this implies \eqref{eq:L}.
\end{proof}

Let us apply this to the distance function.

\begin{corollary}
    In the setup of \Cref{t:Lojasiewicz} with $m=1$, we have, for all $x \in K$, 
    \begin{equation} 
        \|f\|_{B(x,\frac{r(x)}2) \cap K} \ge \frac{\|f\|_K}{\Cnn}\,\Big(\frac{c_K\, \nu_K}{2^{\th_K} \, |K|}\Big)^{\nn}\,  r(x)^{\th_K \nn},
    \end{equation}
    where 
    \begin{equation}
        r(x) := \min\{ \on{dist}(x,f|_K^{-1}(0)),1\}, \quad x \in K,
    \end{equation}
    or, alternatively,
    \begin{equation}
        r(x) := \frac{\on{dist}(x,f|_K^{-1}(0))}{\max\{\diam(K),1\}}, \quad x \in K. s
    \end{equation}
\end{corollary}

\subsection{Order of vanishing}

    The order of vanishing of a continuous function $f : K \to \R$ at $x\in K$ is defined by 
    \begin{equation}
        \on{ord}_x f = \limsup_{r\to 0} \frac{\log \|f\|_{B(x,r)\cap K}}{\log r}.
    \end{equation}
    If $f$ is smooth and $K$ is open, then $\on{ord}_x f$ is the order of the smallest nonzero partial derivative of $f$ at $x$.

\begin{corollary}
    In the setup of \Cref{t:Lojasiewicz} with $m=1$, 
    we have $\on{ord}_x f \le \th_K \nn$ for all $x \in K$.
\end{corollary}

\begin{proof}
    Applying \Cref{t:Lojasiewicz} to $g\equiv r$ for $r\in (0,1)$, gives
    \[
        \|f\|_{B(x,\frac{r}2) \cap K} \ge  \frac{\|f\|_K}{\Cnn}\,\Big(\frac{c_K\, \nu_K}{|K|}\Big)^{\nn}\,  \Big(\frac{r}{2}\Big)^{\th_K \nn},
    \]
    for all $x \in K$.
    This implies the assertion.
\end{proof}

\subsection{Inequalities for the gradient}

We will deduce several inequalities for the gradient.

\subsubsection{First gradient inequality}

The first is a direct application of \Cref{t:Lojasiewicz} for $\grad f$ and $f$ in place of $f$ and $g$.
It requires a ``shift'' of the weight $\mu^{|\infty}$, see \eqref{eq:shift}.

\begin{corollary} \label[c]{c:grad}
Let $L\subseteq \R^n$ be a convex body and
$K \in \sK(\R^n)$ 
such that $K \subseteq L$. 
Assume that $K$ is $(c_K,\th_K)$-thick.
Let $\mu^{|\infty} = (\mu_j)_{j \ge 1}$ be an admissible weight satisfying $\mu_j \ge j$ for $j \ge 1$. 
    Let $M_1 = M_0\mu_1 > b_1>0$. 
    Assume that $(\mu^{|\infty}_{+1},n\,M_1,a_K,b_1)$ is admissible and let 
    \[ 
        \nn = \nn(\mu^{|\infty}_{+1},n\,M_1,a_K,b_1).
    \]
    Let $f : L \to \R$ be a $(\mu^{|\nn+1},M_0)$-smooth function satisfying $\|f\|_K \le 1$ and
    $\|\grad f\|_K\ge b_1$.  
    Then, for all $x \in K$, 
    \begin{equation} \label{eq:Lgrad}
        \|\on{grad} f\|_{B(x,\frac{|f(x)|}2) \cap K} \ge \frac{\|\grad f\|_K}{\Cnn}\,\Big(\frac{c_K\, \nu_K}{n\, 2^{\th_K} \, |K|}\Big)^{\nn}\,  |f(x)|^{\th_K \nn}.
    \end{equation}
\end{corollary}

We do not require $(\grad f|_K)^{-1}(0) \subseteq f|_K^{-1}(0)$.

\begin{proof}
    It is easy to see that $\grad f : L \to \R^n$ is $(\mu^{|\infty}_{+1},M_1)$-smooth.
    Thus the statement follows from \Cref{t:Lojasiewicz}.
\end{proof}

\begin{remark}
    Alternatively, using that $\|\grad f\|_{K} \ge b_1$ implies $\|\p_{j_0} f\|_K\ge  b_1/n$ for some $j_0$ 
    and \Cref{t:Lojasiewicz} for $\p_{j_0} f$ implies the following version of \eqref{eq:Lgrad}: 
   \begin{equation}
        \|\on{grad} f\|_{B(x,\frac{|f(x)|}2) \cap K} \ge \frac{b_1}{\Cnn \, n}\,\Big(\frac{c_K\, \nu_K}{2^{\th_K} \, |K|}\Big)^{\nn}\,  |f(x)|^{\th_K \nn}.
   \end{equation}
   with $\nn = \nn(\mu^{|\infty}_{+1},n\,M_1,a_K,b_1)$.
\end{remark}

\subsubsection{Second gradient inequality}

The constants in \eqref{eq:Lgrad} are completely explicit, but the exponent $\th_K \nn$ of $|f(x)|$ 
is typically large. In the analytic or quasianalytic setting (see \Cref{l:BM}), 
one has an (unspecified) exponent $\al \in (0,1)$.
In the following, we discuss gradient inequalities with explicit constants and \emph{small} exponents of $|f(x)|$.

\begin{lemma}[{\cite[Theorem 6.3]{BM04}}] \label[l]{l:BM}
    Let $\cC_M$ be a quasianalytic Denjoy--Carleman class satisfying the properties listed in \Cref{e:RCM}.
    Let $U \subseteq \R^n$ be open and $f \in \cC_M(U)$. 
    Let $K \subseteq U$ be compact and assume that $(\grad f|_K)^{-1}(0) \subseteq f|_K^{-1}(0)$. 
    Then there exist $c>0$ and $\al \in [0,1)$ such that 
    \[
        |\grad f(x)| \ge c\, |f(x)|^{\al}
    \]
    in a neighborhood of $K$.
\end{lemma}

\begin{theorem} \label[t]{t:gradptw1}
Let $L\subseteq \R^n$ be a convex body and $K \in \sK(\R^n)$ such that $K \subseteq \on{int}(L)$. 
Assume that $K$ is $(c_K,\th_K)$-thick.
Let $\mu^{|\infty} = (\mu_j)_{j \ge 1}$ be a quasianalytic admissible weight satisfying $\mu_j \ge j$ for $j \ge 1$
and all the properties listed in \Cref{e:RCM}. 
    Let $M_0 \ge 1$ and $M_0>b_0>0$
    and
    \[
        \nn = \nn(\mu^{|\infty},M_0,a_K,b_0).
    \]
    Let $f : L \to \R$ be a $(\mu^{|\nn},M_0)$-smooth function 
    satisfying $\|f\|_K\ge b_0$.
Assume that $(\grad f|_K)^{-1}(0) \subseteq f|_K^{-1}(0)$.
Then,
\begin{equation} \label{eq:gradptw1}
    |\grad f(x)| \ge\frac{1}{2\sqrt{M_2}} \Big(\frac{\|f\|_K}{2\Cnn}\Big)^{\frac{1}{\th_K\nn}}\,\Big(\frac{c_K\, \nu_K}{|K|}\Big)^{\frac{1}{\th_K}}\, |f(x)|^{\frac{3}{2} -\frac{1}{\th_K\nn}}
\end{equation}
in a neighborhood of $f|_K^{-1}(0)$ in $K$.
\end{theorem}

\begin{proof}
By \Cref{l:BM}, there exist $c>0$ and $\al \in [0,1)$ such that, for $x \in K$,
\[
    \frac{|f(x)|}{|\grad f(x)|}   =  \frac{|f(x)|^\al|f(x)|^{1-\al}}{|\grad f(x)|}   \le \frac{|f(x)|^{1-\al}}{c}.
\]
Thus there is a neighborhood $U$ in $K$ of $f|_K^{-1}(0)$ such that 
\begin{equation}
   |f(x)| \le 1 \quad \text{ and } \quad  \frac{|f(x)|}{|\grad f(x)|} \le 1 \quad \text{ for } x \in U.
\end{equation}
In particular, the function $g : K \to \R$ which is $0$ on $(\grad f|_K)^{-1}(0)$ and  
\begin{equation} \label{eq:g2}
    g(x) := \frac{|f(x)|^{3/2}}{\sqrt{M_2} \,|\grad f(x)|} \quad \text{ for } x \not\in (\grad f|_K)^{-1}(0),  
\end{equation}
is well-defined on $K$ and satisfies $0 \le g \le 1$ on $U$. 

Let $x \in U$ and $y \in B(x,\frac{g(x)}{2}) \cap L$. Then 
\begin{align*}
    |f(y) - f(x)| &\le   |\langle \grad f(x), y-x \rangle| + \frac{M_2}{2} |x-y|^2 
    \\
                  &\le \frac{1}{2 \sqrt{M_2}} |f(x)|^{3/2} + \frac{1}{8} |f(x)|   
\end{align*}
and consequently (noting that $M_2 \ge 1$),
\begin{equation} \label{eq:uu}
    |f(y)| \le 2 \, |f(x)|.
\end{equation}

By \Cref{t:Lojasiewicz} applied to the choice \eqref{eq:g2} of $g$ and \eqref{eq:uu},
for $x \in U$, 
\begin{align*}
    2\, |f(x)| \ge  \|f\|_{B(x,\frac{g(x)}2) \cap K} \ge \frac{\|f\|_K}{\Cnn}\,\Big(\frac{c_K\, \nu_K}{2^{\th_K} \, |K|}\Big)^{\nn}\,  \Big( \frac{|f(x)|^{3/2}}{\sqrt{M_2} \,|\grad f(x)|}\Big)^{\th_K \nn},
\end{align*}
implying
\begin{equation}
    |\grad f(x)|^{\th_K \nn} \ge \frac{\|f\|_K}{2\Cnn \, M_2^{\th_K\nn/2}}\,\Big(\frac{c_K\, \nu_K}{2^{\th_K} \, |K|}\Big)^{\nn}\, |f(x)|^{\frac{3}{2}\th_K \nn -1}.
\end{equation}
This gives \eqref{eq:gradptw1}.
\end{proof}

\begin{corollary} \label[c]{r:gradexp}
    In the setup of \Cref{t:gradptw1}, 
\begin{equation}
    \|\grad f\|_{B_x} \ge \frac{1}{2} \Big(\frac{2\|f\|_K}{3\Cnn}\Big)^{\frac{1}{\th_K\nn}}\,\Big(\frac{c_K\, \nu_K}{|K|}\Big)^{\frac{1}{\th_K}}\,  |f(x)|^{1-\frac{1}{\th_K \nn}}
\end{equation}
for $x$ in a neighborhood of $f|_K^{-1}(0)$ in $K$, where $B_x$ is a neighborhood of $x$ 
in $L$ that shrinks to $x$ if $x$ tends to $f|_K^{-1}(0)$.
\end{corollary}

\begin{proof}
    We follow the proof of \Cref{t:gradptw1} 
    and replace \eqref{eq:g2} by 
    \begin{equation}   
    g(x) := \frac{|f(x)|}{|\grad f(x)|}.
    \end{equation}
    Then the function $g : K \to \R$ is well-defined and satisfies $0 \le g \le 1$ on $U$.
    Moreover, $g(x) \to 0$ as $x$ tends to $f|_K^{-1}(0)$.
    
    Set $B_x :=  B(x,\frac{g(x)}{2}) \cap L$.
    For $x \in U$ and $y \in B_x'$, we have
\begin{align*}
    |f(y) - f(x)| &\le  \sup_{z \in B_x}  |\grad f(z)|  \frac{|f(x)|}{2\,|\grad f(x)|}
                 \le \frac{\|\grad f\|_{B_x}}{2 \, |\grad f(x)|} \,|f(x)|,
\end{align*}
(using $|\grad f(z)| \le |\grad f(z)|_1$)
and thus
\begin{equation}
    |f(y)| \le \Big(1+\frac{\|\grad f\|_{B_x}}{2 \, |\grad f(x)|}\Big) \, |f(x)| 
    \le \frac{3}{2}\, \|\grad f\|_{B_x}\,\frac{|f(x)|}{ |\grad f(x)|}.
\end{equation}
Then \Cref{t:Lojasiewicz} applied to this choice of $g$ gives
\begin{equation}
    \frac{3}{2}\, \|\grad f\|_{B_x}\,\frac{|f(x)|}{ |\grad f(x)|}  \ge \frac{\|f\|_K}{\Cnn}\,\Big(\frac{c_K\, \nu_K}{2^{\th_K} \, |K|}\Big)^{\nn}\,  \Big( \frac{|f(x)|}{|\grad f(x)|}\Big)^{\th_K \nn},
\end{equation}
and hence
\begin{equation}
    \|\grad f\|_{B_x}^{\th_K \nn} \ge \frac{2\|f\|_K}{3\Cnn}\,\Big(\frac{c_K\, \nu_K}{2^{\th_K} \, |K|}\Big)^{\nn}\,  |f(x)|^{\th_K \nn-1}.
\end{equation}
which implies the statement.
\end{proof}

\subsubsection{Third gradient inequality}

We close this section with another inequality of {\L}ojasiewicz type based on the coarea formula, which we recall 
in the following lemma in the form we will use (cf.\ \cite{Federer69}).

\begin{lemma} \label[l]{l:coarea}
   Let $f : \R^n \to \R$ be Lipschitz.
   Then, for all Borel measurable functions $g : \R^n \to [0,\infty]$,
  \[
      \int_{\R^n} g(x)\, |\grad f(x)| \, dx = \int_{\R} \int_{f^{-1}(y)} g \, d\cH^{n-1} \, dy.
  \] 
\end{lemma}

\begin{theorem} \label[t]{t:grad2}
Let $L\subseteq \R^n$ be a convex body and $K \in \sK(\R^n)$ such that $K \subseteq L$. 
Let $\mu^{|\infty} = (\mu_j)_{j \ge 1}$ be an admissible weight satisfying $\mu_j \ge j$ for $j \ge 1$. 
    Let $M_0 > b_0>0$. 
    Assume that $(\mu^{|\infty},M_0,a_K,b_0)$ is admissible and let 
    \[ 
        \nn = \nn(\mu^{|\infty},M_0,a_K,b_0).
    \]
    Let $f : L \to \R$ be a $(\mu^{|\nn},M_0)$-smooth function satisfying $\|f\|_K \ge b_0$.  
    Assume that there exist $\al\ge 0$ and $l,t_0>0$ such that  
    \begin{equation}
        \cH^{n-1}(\{|f|=y\}\cap K) \ge l\, y^{\al} \quad \text{ for almost every } y \in [0,t_0].
    \end{equation}
    Then, for $K_t := \{x \in K : |f(x)| \le t\}$, 
    \begin{equation}
        \|\grad f\|_{K_t} \ge \frac{l\, \nu_K}{(\al+1)\,|K|} \Big(\frac{\|f\|_K}{\Cnn}\Big)^{1/\nn} \, \|f\|_{K_t}^{\al+1-\frac{1}{\nn}}, \quad  0 \le t \le t_0.
    \end{equation}
\end{theorem}

\begin{proof}
    First we observe that $|\grad |f(x)|| = |\grad f(x)|$ if $f(x) \ne 0$. The zero set of $f$ has Lebesgue measure zero, 
    by \cite[Theorem 4.6]{Rainer:2022aa}.
    Set $K_t' := \{x \in K : 0<|f(x)| \le t\}$.
    By \Cref{l:coarea}, 
\begin{align*}
    |K_t'| = \int_0^{t} \int_{\{|f| = y\}\cap K} \frac{\mathbf 1_{K_t'}}{|\grad f|}\, d\cH^{n-1}\, dy.
\end{align*}
We conclude that, for $0 \le t \le t_0$,
\begin{align*}
    |K_t'|&\ge \frac{1}{\|\grad f\|_{K_t'}}  \int_0^{t} \int_{\{|f|=y\}\cap K} \mathbf 1_{K_t'}\, d\cH^{n-1}\, dy 
    \\
          &= \frac{1}{\|\grad f\|_{K_t'}}  \int_0^{t} \cH^{n-1}(\{|f|=y\} \cap K) \, dy 
          \ge \frac{l \, t^{\al+1}}{(\al+1)\, \|\grad f\|_{K_t'}}. 
\end{align*}
Combined with \Cref{t:sublevel}, this implies the assertion.
\end{proof}

\section{Inequalities of Harnack type} \label{sec:Harnack}

In this section, we prove two inequalities of Harnack type for $(\mu^{|N},M_0)$-smooth functions.
The first is a simple consequence of the Remez inequality, but the involved constant is somewhat implicit 
and not easy to compute.
The second inequality has a completely explicit constant but requires more work.

\subsection{First inequality of Harnack type}

\begin{proposition} \label[p]{p:Harnack}
Let $L\subseteq \R^n$ be a convex body and
$K \in \sK(\R^n)$ 
such that $K \subseteq L$. 
Let $\mu^{|\infty} = (\mu_j)_{j \ge 1}$ be an admissible weight satisfying $\mu_j\ge j$ for $j \ge 1$. 
Let $M_0>b_0>0$.
Assume that $(\mu^{|\infty},M_0,a_K,b_0)$ is admissible and let 
\[
\nn = \nn(\mu^{|\infty},M_0,a_K,b_0).
\]
Let $f : L \to \R$ be $(\mu^{|\nn},M_0)$-smooth with $\|f\|_K\ge b_0$ 
and $\min_{x \in K} |f(x)| \ge m_0>0$.
    Then
    \begin{equation}
        \max_{x \in K} |f(x)| \le 2\,\Cnn\, \Big(\frac{|K|^{\nu_K}}{|K|^{\nu_K}-(|K|-|E_{m_0}|)^{\nu_K}}\Big)^{\nn}\, \min_{x \in K} |f(x)|,
    \end{equation}
    where $E_{m_0}:= \{x \in K : |f(x)| \le \min_{x \in K} |f(x)| +m_0 \}$.
\end{proposition}

\begin{proof}
    The set $E_{m_0}$ is closed and $|E_{m_0}|>0$, by the continuity of $f$. 
    The statement follows from \Cref{t:RemezKL} and 
    the fact that $\|f\|_{E_{m_0}} \le 2\,\min_{x \in K} |f(x)|$.
\end{proof}

\subsection{Bounds for \texorpdfstring{$1/f$}{1/f}}

Let $\mu^{|\infty} = (\mu_j)_{j \ge 1}$ be an admissible weight and $M_0>0$.
Let us assume that the sequence
\begin{equation} \label{eq:aj}
a_j :=\Big(\frac{\mu_1 \cdots \mu_j}{j!}\Big)^{1/j}, \quad  j\ge 1, 
\end{equation}
is almost increasing 
(cf.\ \Cref{e:RCM}), i.e.,
there exists a constant $H \ge 1$ such that 
\begin{equation} \label{eq:rai}
    a_j \le H \, a_k \quad \text{ whenever } j \le k.
\end{equation}

\begin{lemma} \label[l]{l:inv1}
    Let $\mu^{|\infty} = (\mu_j)_{j \ge 1}$ be an admissible weight satisfying \eqref{eq:rai} with constant $H \ge 1$ and let $M_0 \ge m_0 >0$.
    Let $I \subseteq \R$ be a compact interval and $f : I \to \R$ 
    a $(\mu^{|N},M_0)$-smooth function such that $\min_{x \in I} |f(x)| \ge m_0$. 
    Then the function $1/f : I \to \R$ is well-defined 
    and $\big(H(1+\frac{2 M_0}{m_0})\, \mu^{|N},\frac{2}{m_0}\big)$-smooth.
\end{lemma}

\begin{proof}
    We follow the proof of \cite[Theorem 2.1]{Rudin62}. Recall that $N \in \N_{\ge 1} \cup \{\infty\}$.
    For $1 \le n < N+1$, define $r_n$ by 
    \begin{equation} \label{eq:rn}
        H\, a_n  r_n = \frac{m_0}{m_0+2M_0}.
    \end{equation}
    Fix $x_0 \in I$ and consider the polynomial $P(z) := \sum_{j = 0}^n \frac{f^{(j)}(x_0)}{j!} z^j$. 
    Using \eqref{eq:rai}, we find that, if $|z| \le r_n$, 
    \begin{align*}
        |P(z)| &\ge |f(x_0)| - \sum_{j = 1}^n \frac{|f^{(j)}(x_0)|}{j!} |z|^j        
        \\
               &\ge m_0 - \sum_{j = 1}^n \frac{M_j}{j!} r_n^j = m_0 - M_0 \sum_{j = 1}^n (a_j r_n)^j  
               \ge m_0 - M_0 \sum_{j = 1}^n (H a_n r_n)^j
               \\
               &=m_0 - M_0 \sum_{j = 1}^n \Big(\frac{m_0}{m_0+2M_0}\Big)^j > m_0 - \frac{m_0 M_0}{m_0+2M_0} \frac{m_0+2M_0}{2M_0} = \frac{m_0}2.
    \end{align*}
    Since $f^{(j)}(x_0) = P^{(j)}(0)$, for $0 \le j \le n$, we have $(1/f)^{(n)}(x_0) = (1/P)^{(n)}(0)$ 
    and hence, by \eqref{eq:aj} and \eqref{eq:rn},
    \begin{align}
        |(1/f)^{(n)}(x_0)| &= \Big|\frac{n!}{2\pi i} \int_{|z|=r_n} \frac{dz}{z^{n+1} P(z)}\Big| 
        \\
                           &\le \frac{2}{m_0} \frac{n!}{r_n^n}
        = \frac{2}{m_0} \Big(H \Big(1+ \frac{2M_0}{m_0}\Big)\Big)^n \mu_1 \cdots \mu_n.
    \end{align}
    This implies the assertion.
\end{proof}

\begin{lemma} \label[l]{l:inv2}
    Let $T : (\R^n)^j \to \R$, 
    \[
        T(v_1,\ldots,v_j) = \sum_{i_1,\ldots,i_j=1}^n T_{i_1,\ldots,i_j} v_1^{i_1} \cdots v_j^{i_j}
    \]
    be a $j$-linear symmetric map, where $v_k = (v_k^{1},\ldots,v_k^n)$.
   Then 
   \[
       \|T\|_2 = \Big( \sum_{i_1,\ldots,i_j=1}^n T_{i_1,\ldots,i_j}^2 \Big)^{1/2} 
       \le n^{\frac{j-1}{2}} \sup_{|v| \le 1} |T(v,\ldots,v)|.
   \]
\end{lemma}

\begin{proof}
    We proceed by induction on $j$. For $j=1$, we have equality.
    Assume $j \ge 2$. By induction hypothesis, for the symmetric $(j-1)$-linear maps $T_i := T(e_i,\cdots)$,
    \[
        \|T_i\|_2  \le n^{\frac{j-2}{2}} \sup_{|v| \le 1} |T_i(v,\ldots,v)|.
    \]
    Therefore,
    \begin{align*}
        \|T\|_2^2 =  \sum_{i_1=1}^n \|T_{i_1}\|_2^2 
        \le \sum_{i_1=1}^n  n^{j-2} \sup_{|v| \le 1} |T_{i_1}(v,\ldots,v)|^2 \le n^{j-1} \sup_{|v| \le 1} |T(v,\ldots,v)|^2,
    \end{align*}
    where the last inequality is a consequence of
    \[
        \sup_{|v_1| \le 1,\ldots,|v_j|\le 1} |T(v_1,\ldots,v_j)|= \sup_{|v| \le 1} |T(v,\ldots,v)|,
    \]
    see \cite[Proposition 1]{Bochnak_1971}.
\end{proof}

\begin{lemma} \label[l]{l:inv3}
    Let $\mu^{|\infty} = (\mu_j)_{j \ge 1}$ be an admissible weight satisfying \eqref{eq:rai}  with constant $H \ge 1$ and let $M_0 \ge m_0 >0$.
    Let $K \subseteq \R^n$ be a nonempty fat compact set and $f: K \to \R$ a  $(\mu^{|N},M_0)$-smooth function such that 
    $\min_{x \in K} |f(x)| \ge m_0$.
    Then the function $1/f : K \to \R$ is well-defined and $\big(nH(1+\frac{2M_0}{m_0})\, \mu^{|N},\frac{2}{m_0}\big)$-smooth.
\end{lemma}

\begin{proof}
    Fix $x \in \on{int}(K)$ and $v \in \mathbb S^{n-1}$.  Let $g(t) := f(x+t v)$ for $t\in I$,
    where $I\subseteq \R$ is a compact interval such that $x \in x + I v \subseteq K$.
    Then $g : I \to \R$ is $(\mu^{|N},M_0)$-smooth. By \Cref{l:inv1}, 
    \begin{equation}
        |(1/g)^{(j)}(0)| = |d_v^j(1/f)(x)| \le \frac{2}{m_0} H^j \Big(1+ \frac{2M_0}{m_0}\Big)^j \mu_1 \cdots \mu_j. 
    \end{equation}
    Since $\sum_{|\al|= j} \frac{j!}{\al!} = n^j$, we have 
    \[
        \sum_{|\al|= j} \frac{j!}{\al!} |\p^\al (1/f)(x)| \le n^{j/2}\, \Big(\sum_{|\al|= j} \frac{j!}{\al!} |\p^\al (1/f)(x)|^2 \Big)^{1/2}
    \]
    and, by \Cref{l:inv2}, the right-hand side is bounded by
    \[
        n^{j/2} \cdot n^{(j-1)/2} \sup_{|v| \le 1} \Big| \sum_{|\al|= j} \frac{j!}{\al!} \p^\al (1/f)(x) v^\al \Big| 
        = \frac{n^j}{\sqrt n} \sup_{|v| \le 1}  |d_v^j (1/f)(x)|.
    \]
    It follows that, for $j\ge 1$,
    \[
       \sum_{|\al|= j} \frac{j!}{\al!} |\p^\al (1/f)(x)| \le \frac{2}{\sqrt n\, m_0} (n H)^j \Big(1+ \frac{2M_0}{m_0}\Big)^j \mu_1 \cdots \mu_j,
    \]
    for all $x \in \on{int}(K)$. By continuity and since $K$ is fat, the estimate holds for all $x \in K$.
    Clearly, $\|1/f\|_K \le 1/m_0$.
\end{proof}

\subsection{Second inequality of Harnack type}

\begin{theorem} \label[t]{t:Harnack2}
Let $L\subseteq \R^n$ be a convex body and
$K \in \sK(\R^n)$ 
such that $K \subseteq L$. 
Let $\mu^{|\infty} = (\mu_j)_{j \ge 1}$ be an admissible weight satisfying $\mu_j\ge j$ for $j \ge 1$ 
as well as \eqref{eq:rai} with constant $H\ge 1$. 
Let $M_0\ge m_0>0$.
Set 
\begin{equation}
    \be_0 := n H (1+\tfrac{2M_0}{m_0}).
\end{equation}
Assume that $(\mu^{|\infty},2M_0,a_K \be_0,m_0)$ is admissible and let 
\[
    \nn = \nn(\mu^{|\infty},2M_0,a_K \be_0,m_0).
\]
If $f : L \to \R$ is $(\mu^{|\nn},M_0)$-smooth and satisfies
    $\min_{x \in L} |f(x)| \ge m_0$, then 
    \begin{equation}
        \max_{x \in K} |f(x)| \le \Cnn^2\, \Big(\frac{2^{\nu_K}}{2^{\nu_K}-1}\Big)^{2\nn}\, \min_{x \in K} |f(x)|.
    \end{equation}
\end{theorem}

\begin{proof}
    Let us set $m_f:=\min_{x \in K} |f(x)|$ and $M_f := \max_{x \in K} |f(x)|$.
    Clearly, $m_0 \le m_f \le M_f \le M_0$.
    The set $E_-:= \{x \in K : |f(x)| \le \sqrt{M_fm_f} \}$ has positive measure.
    By \Cref{t:RemezKL}, 
    \begin{equation}
        M_f \le \mathbf{C}_{\nn'}\, \Big(\frac{|K|^{\nu_K}}{|K|^{\nu_K}-(|K|-|E_-|)^{\nu_K}}\Big)^{\nn'}\, \sqrt{M_fm_f},
    \end{equation}
    where $\nn' := \nn(\mu^{|\infty},M_0,a_K,m_0)$. 
    Since $\nn' \le \nn$ (see \Cref{r:nnprop}), we conclude
    \begin{equation}
        M_f \le \Cnn^2\, \Big(\frac{|K|^{\nu_K}}{|K|^{\nu_K}-(|K|-|E_-|)^{\nu_K}}\Big)^{2 \nn}\, m_f.
    \end{equation}

    By \Cref{l:inv3}, the function $1/f : L \to \R$ is $\big(\be_0\, \mu^{|\nn},\frac{2}{m_0}\big)$-smooth
    and satisfies $\min_{x \in L} |1/f(x)| \ge  1/M_0$. Note that $m_{1/f} := \min_{x \in K} |1/f(x)| = 1/M_f$ and 
    $M_{1/f} := \max_{x \in K} |1/f(x)| = 1/m_f$.
    Observe that 
\[
    \nn(\mu^{|\infty},\tfrac{2}{m_0},a_K \be_0,\tfrac{1}{M_0}) =  \nn(\mu^{|\infty},2M_0,a_K \be_0,m_0).
\]

    Thus, \Cref{t:RemezKL} applied to $1/f$, gives
    \begin{equation}
        M_{1/f}  \le \Cnn\, \Big(\frac{|K|^{\nu_K}}{|K|^{\nu_K}-(|K|-|E_+|)^{\nu_K}}\Big)^{\nn}\,   \sqrt{M_{1/f} m_{1/f}},
    \end{equation}
    where $E_+:= \{x \in K : |1/f(x)| < \sqrt{M_{1/f} m_{1/f}}\}$, which implies 
    \begin{equation}
        M_{f}  \le \Cnn^2\, \Big(\frac{|K|^{\nu_K}}{|K|^{\nu_K}-(|K|-|E_+|)^{\nu_K}}\Big)^{2\nn}\, m_{f}.
    \end{equation}

    Now $E_+ = \{x \in K : \sqrt{M_f m_f} <  |f(x)|\}$ so that $E_-$ and $E_+$ form a partition of $K$. 
    Therefore, $|E_-| \ge |K|/2$ or $|E_+| \ge |K|/2$. Consequently,
    \begin{equation}
        M_{f}  \le \Cnn^2\, \Big(\frac{2^{\nu_K}}{2^{\nu_K}-1}\Big)^{2\nn}\, m_{f}.
    \end{equation}
    The proof is complete.
\end{proof}

\section{Inequalities of Markov type} \label{sec:Markov}

Here we discuss a few inequalities of Markov type which are consequences 
of an upper bound for the number of zeros in terms of $\nn$ which we recall in \Cref{p:zeros} 
and the comparison of the $L^1$- and $L^\infty$-norm, see \Cref{t:comparison}.

\subsection{Number of zeros}

\begin{proposition}[{\cite[Corollary 3.6 and Remark 3.7]{Rainer:2022aa}}] \label[p]{p:zeros}
   Let $I \subseteq \R$ be a compact interval.
   Let $\mu^{|\infty} = (\mu_j)_{j \ge 1}$ be an admissible weight and $M_0>b_0>0$.
   Assume that $(\mu^{|\infty},M_0,|I|,b_0)$ is admissible and let 
    \[ 
        \nn = \nn(\mu^{|\infty},M_0,|I|,b_0).
    \]
   Let $f : I \to \R$ be a $(\mu^{|\nn}, M_0)$-smooth function satisfying $\|f\|_I \ge b_0$.
   Then the total number of zeros of $f$ in $I$ (counted with multiplicities) is finite and 
   bounded by $\nn$.
\end{proposition}

\subsection{Inequalities of Markov type for univariate functions}

The following lemma is a variant of \cite[Lemma 8.1]{Pierzchala:2022aa}. 

\begin{lemma} \label[l]{l:Markov1}
   Let $I \subseteq \R$ be a compact interval.
   Let $\mu^{|\infty} = (\mu_j)_{j \ge 1}$ be an admissible weight and $M_1 = M_0\mu_1>b_1>0$.
   Assume that $(\mu^{|\infty}_{+1},M_1,|I|,b_1)$ is admissible and let 
    \[ 
        \nn = \nn(\mu^{|\infty}_{+1},M_1,|I|,b_1).
    \]
   Let $f : I \to \R$ be a $(\mu^{|\nn+1}, M_0)$-smooth function satisfying $\|f'\|_I \ge b_1$.
   Then, for each nonempty Lebesgue measurable subset $E \subseteq I$,  
   \begin{equation}
       \|f'\|_{L^1(E)} \le 2 (\nn+1)\, \|f\|_E.   
   \end{equation}
\end{lemma}

\begin{proof}
    By \Cref{p:zeros} 
    $f'$ has at most $\nn$ zeros in $I$ (counted with multiplicities). 
    Thus there is sequence $t_1 < \cdots <t_n$, with $n \le \nn$, of points in $I$ such that the sign of $f'$ 
    is constant on each interval $I_j := (t_j,t_{j+1})$, for $j=1, \ldots, n-1$, $I_0 := I \cap (-\infty,t_1)$ and 
    $I_n := I \cap (t_n,\infty)$. 
    Then 
    \begin{align*}
        \int_E |f'(x)| \, dx &= \sum_{j=0}^n \int_{E \cap I_j}  |f'(x)| \, dx 
        \\
                             &\le \sum_{j=0}^n |f(\sup E \cap I_j) - f(\inf E \cap I_j)| \le 2(n+1)\, \|f\|_E 
    \end{align*}
    which implies the assertion. 
\end{proof}

Combining \Cref{l:Markov1} with \Cref{t:comparison}, we obtain the following 
inequality of Markov type.

\begin{proposition} \label[p]{p:Markov2}
    In the setting of \Cref{l:Markov1}, for each Lebesgue measurable subset $E \subseteq I$ with $|E|>0$,  
   \begin{equation}
   \|f'\|_{I} \le 2\,\Cnn \, \frac{|I|^{\nn}}{|E|^{\nn+1}}  (\nn+1)^2\, \|f\|_E.   
   \end{equation}
\end{proposition}

\begin{proof}
    Combine \Cref{l:Markov1} with \eqref{eq:LpLqE} for $f'$. 
    The condition $\mu_j \ge j$ for $j\ge 1$ is not needed; see \Cref{r:convexbody}.
\end{proof}

\begin{corollary}
   Let $I \subseteq \R$ be a compact interval.
   Let $\mu^{|\infty} = (\mu_j)_{j \ge 1}$ be a quasianalytic admissible weight and $M_0>0$.
   Let $f : I \to \R$ be a $(\mu^{|\infty}, M_0)$-smooth function.
   Then, for all $j \ge 1$,
   \begin{equation}
       \|f^{(j)}\|_{I} \le \Big(\frac{2}{|I|} \Big)^j\, \prod_{i=1}^j \mathbf{C}_{\nn_i}  (\nn_i+1)^2\, \|f\|_I,
   \end{equation}
   where $\nn_i = \nn(\mu^{|\infty}_{+i},M_i,|I|,\frac{\|f^{(i)}\|_I}2)$.
\end{corollary}

\begin{proof}
    Iterate \Cref{p:Markov2}. In the definition of $\nn_i$ we put $\tfrac{\|f^{(i)}\|_I}2$ 
    to guarantee that $M_i > \frac{\|f^{(i)}\|_I}2$ and hence $\nn_i \in \N_{\ge 2}$. 
\end{proof}

\subsection{Sampling the supremum norm} \label{ssec:sampling}

We need an easy consequence of \Cref{p:Markov2}.

\begin{corollary} \label[c]{c:sampling1}
   Let $\mu^{|\infty} = (\mu_j)_{j \ge 1}$ be a quasianalytic admissible weight, $M_0>0$, and $\mu_1 >1$.
   Let $f : [0,1] \to \R$ be a nonzero $(\mu^{|\infty}, M_0)$-smooth function.
   Then
   \begin{equation}
       \|f'\|_{[0,1]} \le 2 \, \Cnn\, (\nn+1)^2\, \|f\|_{[0,1]},
   \end{equation}
   where $\nn = \nn(\mu^{|\infty}_{+1},M_1,1,\|f\|_{[0,1]})$.
\end{corollary}

\begin{proof}
    The assertion is trivially true if $\|f'\|_{[0,1]} \le \|f\|_{[0,1]}$.
    If $\|f'\|_{[0,1]} > \|f\|_{[0,1]}$, then the statement follows from \Cref{p:Markov2}.
    The assumption $\mu_1>1$ guarantees that $\|f\|_{[0,1]} < M_1$ so that $\nn \in \N_{\ge 2}$.
\end{proof}

Recall that an \emph{$\ep$-net} for $[0,1]$ is a subset $X \subseteq [0,1]$ such that for each $y \in [0,1]$ there
exists $x \in X$ with $|x-y|\le \ep$.

\begin{corollary} \label[c]{c:sampling2}
   In the setting of \Cref{c:sampling1},
   let $X$ be any $\ep$-net for $[0,1]$, where 
   \begin{equation} \label{eq:sampling}
       \ep := \frac{1}{4 \, \Cnn\, (\nn+1)^2}.
   \end{equation}
   Then 
   \begin{equation} \label{eq:sampling_aim}
      \|f\|_X \le \|f\|_{[0,1]} \le 2\,  \|f\|_X.
   \end{equation}
\end{corollary}

\begin{proof}
    Let $y \in [0,1]$. There exists $x \in X$ such that $|x-y| \le \ep$. By \Cref{c:sampling1},
    \begin{align*}
        |f(y)| \le |f(x)| + |x-y| \, \|f'\|_{[0,1]} 
        \le \|f\|_X +  2 \ep\, \Cnn\, (\nn+1)^2\, \|f\|_{[0,1]}.
    \end{align*}
    With the choice \eqref{eq:sampling} for $\ep$, the statement follows easily.
\end{proof}

So given a nonzero $(\mu^{|\infty}, M_0)$-smooth function $f : [0,1] \to \R$,
we may compute $|f(x_0)|$ at a random point $x_0$. 
If $|f(x_0)| \ne 0$, then \Cref{c:sampling2} implies that, if 
$X$ consists of $4 \, \Cnn\, (\nn+1)^2 +1$ equidistant points in $[0,1]$, where 
$\nn = \nn(\mu^{|\infty}_{+1},M_1,1,|f(x_0)|)$, then $X$ satisfies \eqref{eq:sampling_aim}.

\subsection{Inequalities of Markov type for multivariate functions}

Let $f : I^n \to \R$ be a function defined on a compact cube $I^n \subseteq \R^n$.
If $n=1$ we set $V(f) := \sup_{x,y \in I} |f(x)-f(y)|$.
In general, for $1 \le i \le n$,
\begin{equation}
  V_i(f) := \inf_\ell V(f|_\ell),  
\end{equation}
where $\ell$ ranges over all line segments in $I^n$ with direction $e_i$. 

\begin{corollary}
Let $I^n \subseteq \R^n$ be a compact cube.
    Let $\mu^{|\infty} = (\mu_j)_{j \ge 1}$ be a quasianalytic admissible weight and $M_0>0$.
    Let $f : I^n \to \R$ be a $(\mu^{|\infty}, M_0)$-smooth function.
    If $V_i(f) >0$, then
    \begin{equation} \label{eq:cube1}
        \|\p_i f\|_{I^n} \le  \frac{2\,\Cnn}{|I|} \, (\nn+1)^2\, \|f\|_{I^n}
    \end{equation}
    where $\nn = \nn(\mu^{|\infty}_{+1}, M_1, |I|, \frac{V_i(f)}{2|I|})$.
\end{corollary}

\begin{proof}
    If $\ell \subseteq I$ is a line segment with direction $e_i$,
    then $\|\p_i f\|_\ell \ge V_i(f)/|I|$ for all $x \in \ell$. 
    The statement follows from \Cref{p:Markov2}. 
    (We have $\nn \in \N_{\ge 2}$ since $M_1 > \frac{V_i(f)}{2|I|}$.)
\end{proof}

The next lemma is a variant of \cite[Lemma 8.1]{Pierzchala:2022aa}. 
If $K \subseteq \R^n$ is a nonempty compact set, let $K_i \subseteq \R^{n-1}$ 
denote the image of $K$ under the projection in direction $e_i$. 
In the following, the measure $\cL^{n-1}(K_i)$ will appear; 
for $n=1$, we set $\cL^0(K_1) := 1$, by convention.

\begin{lemma} \label[l]{l:Markov4}
   Let $I^n \subseteq \R^n$ be a compact cube.
   Let $\mu^{|\infty} = (\mu_j)_{j \ge 1}$ be a quasianalytic admissible weight and $M_0 >0$.
   Let $f : I^n \to \R$ be a $(\mu^{|\infty}, M_0)$-smooth function.
   If $V_i(f)>0$, then 
   for each nonempty compact subset $K \subseteq I^n$,  
   \begin{equation} \label{eq:cube2}
       \|\p_i f\|_{L^1(K)} \le 2 \cL^{n-1}(K_i) \,(\nn+1) \, \|f\|_K,   
   \end{equation}
   where $\nn = \nn(\mu^{|\infty}_{+1}, M_1,|I|,\frac{V_i(f)}{2|I|})$. 
\end{lemma}

\begin{proof}
    The case $n=1$ follows from \Cref{l:Markov1}.
    Let $n\ge 2$ and $i=1$. Then, by Fubini's theorem and the case $n=1$, 
    \begin{align*}
        \int_{K} |\p_1 f(x_1,x')|\, dx \le \int_{K_1} 2\, (\nn+1) \, \|f\|_{K} \, dx' = 2 \cL^{n-1}(K_1) \,(\nn+1) \, \|f\|_K.
    \end{align*}
    For $i \ge 2$, we argue analogously.
\end{proof}

Let us combine \Cref{l:Markov4} with \Cref{t:comparison}.

\begin{corollary}
   Let $I^n \subseteq \R^n$ be a compact cube.
   Let $\mu^{|\infty} = (\mu_j)_{j \ge 1}$ be a quasianalytic admissible weight satisfying $\mu_j \ge j$ for $j\ge 1$.
   Let $M_1=M_0\mu_1> b_1>0$.
   Let $K\in \sK(\R^n)$ be a subset of $I^n$.
   Let $f : I^n \to \R$ be a $(\mu^{|\infty}, M_0)$-smooth function 
   such that $\|\p_i f\|_K \ge b_1$.
   If $V_i(f)>0$, then 
   \begin{equation} \label{eq:cube3}
       \|\p_i f\|_{K} \le     \frac{2\Cnn\, \cL^{n-1}(K_i)}{\nu_K^\nn \,|K|}    \,(\nn+1)^2 \, \|f\|_K,   
   \end{equation}
   where $\nn = \nn(\mu^{|\infty}_{+1}, M_1,\max\{a_K,|I|\},\min\{b_1,\frac{V_i(f)}{|I|}\})$. 
\end{corollary}

\begin{proof}
    The assertion is immediate from
    \Cref{l:Markov4} and \eqref{eq:LpLqK} applied to $\p_i f$.
    (Note that $\nn \in \N_{\ge 2}$, because $M_1>b_1$.)
\end{proof}

\section{Oscillatory integrals} \label{sec:oscillatory}

The sublevel set estimates of \Cref{t:sublevel} and the bound for the numbers of zeros of \Cref{p:zeros}
can be combined with the van der Corput lemma to provide estimates for oscillatory integrals 
whose phase functions are $(\mu^{|\infty},M_0)$-smooth.
A similar approach was pursued by \cite{Loi:2021aa} for phase functions that 
are definable in an o-minimal structure.

\subsection{Van der Corput lemma}

\begin{lemma}[{\cite[Proposition 2.2]{Carbery:1999aa}, \cite{Stein93}}] \label[l]{l:Corput} 
    There exists an absolute constant $C>0$ such that for any compact interval $I \subseteq \R$, 
    any $k \ge 2$, and any integrable function $f : I \to \R$ satisfying $|f^{(k)}| \ge 1$ in the sense of distributions,
    \begin{equation}
        \Big|\int_I e^{i\la f(x)} \, dx\Big| \le C \, k\, \la^{-1/k}, \quad \la > 0.
    \end{equation}
\end{lemma}

\subsection{Univariate phase and amplitude}

\begin{proposition} \label[p]{p:oscillatory}
    Let $I \subseteq \R$ be a compact interval.
    Let $\mu^{|\infty} = (\mu_j)_{j\ge 1}$ be an admissible weight and $M_2=M_0\mu_1 \mu_2\ge b_2>0$.
Assume that $(\mu^{|\infty}_{+2},M_2+\frac{b_2}{2},|I|,\frac{b_2}2)$ is admissible and let 
    \[ 
        \nn = \nn(\mu^{|\infty}_{+2},M_2+\tfrac{b_2}{2},|I|,\tfrac{b_2}{2}).
    \]
    Let $f : I \to \R$ be a $(\mu^{|\nn+2},M_0)$-smooth function satisfying $\|f''\|_I \ge b_2$.
    Let $g : I \to \C$ belong to $\cC^1(I)$. Then, 
    \begin{equation} \label{eq:osc1}
        \Big|\int_I e^{i\la f(x)} g(x)\, dx\Big| \le B\, \la^{-\frac{1}{\nn+2}},
        \quad \text{ if } \la \ge \Big(\frac{2}{b_2}\Big)^{1 + \frac{2}{\nn}}, 
    \end{equation}
    where 
    \begin{equation}
        B = \frac{\Cnn^{1/\nn}\, |I|\, \|g\|_I}{\|f''\|_I^{1/\nn}} + 2C\, (\nn+1) \big( \|g\|_I + \|g'\|_{L^1(I)} \big)
    \end{equation}
    and $C>0$ is an absolute constant.
\end{proposition}

\begin{proof}
    Consider the sublevel set $I_t := \{x \in I : |f''(x)| \le t\}$ of $f''$. By \Cref{t:sublevel} and \Cref{r:convexbody}, we have 
    \begin{equation}
        |I_t| \le |I| \Big(\frac{\Cnn \, t}{\|f''\|_I}\Big)^{1/\nn}, \quad t >0.
    \end{equation}
    Consequently,
    \begin{equation}
        \Big|\int_{I_t} e^{i\la f(x)} g(x)\, dx\Big| \le \|g\|_I\, |I|\, \Big(\frac{\Cnn \, t}{\|f''\|_I}\Big)^{1/\nn}, \quad t >0.
    \end{equation}

    For $0< t \le \frac{b_2}2$, we have 
    \begin{align*}
        \|f'' \pm t\|_I &\le \|f''\|_I + t \le M_2 + \tfrac{b_2}2
        \intertext{and}
        \|f'' \pm t\|_I &\ge \|f''\|_I - t \ge b_2-\tfrac{b_2}2 =\tfrac{b_2}2
    \end{align*}
    so that, by \Cref{p:zeros}, $|f''| - t$ has at most $\nn$ zeros in $I$. We conclude that, for $0<t\le \frac{b_2}2$, 
    the set $I \setminus I_t = \{x \in I : |f''(x)|>t\}$ consists of at most $\nn+1$ intervals $(a_i,b_i)$ (the intervals containing the endpoints of $I$ are half-open). 
    By partial integration, 
    \begin{align*}
        \int_{a_i}^{b_i} e^{i\la f(x)} g(x)\, dx
        &= \int_{a_i}^{b_i} \p_x \Big(\int_{a_i}^x e^{i\la f(u)}\, du\Big) g(x)\, dx 
        \\
        &= \Big(\int_{a_i}^{b_i} e^{i\la f(u)} \, du\Big) g(b_i) 
        -\int_{a_i}^{b_i} \Big(\int_{a_i}^x e^{i\la f(u)} \, du\Big) g'(x)\, dx.
    \end{align*}
    By \Cref{l:Corput}, for $\la >0$ and $x\in [a_i,b_i]$,
    \begin{align*}
       \Big| \int_{a_i}^{x} e^{i \la f(u)} \, du \Big| 
       = \Big| \int_{a_i \sqrt t}^{x \sqrt t} e^{i \la f(y/\sqrt t)} \, \frac{dy}{\sqrt t} \Big| \le 2C\, (\la t)^{-1/2}.
    \end{align*}
    Thus,
    \begin{align*}
       \Big| \int_{a_i}^{b_i} e^{i \la f(x)} g(x) \, dx \Big|
       \le 2C\, (\la t)^{-1/2}\big(\|g\|_I + \|g'\|_{L^1(I)}\big). 
    \end{align*}
    It follows that, for  $0< t \le \frac{b_2}2$ and $\la >0$,
    \begin{align*}
        \MoveEqLeft \Big|\int_{I} e^{i\la f(x)} g(x)\, dx\Big| 
        \\
        &\le \|g\|_I\, |I|\, \Big(\frac{\Cnn \, t}{\|f''\|_I}\Big)^{1/\nn} + 2C\,(\nn+1)\, (\la t)^{-1/2}\big(\|g\|_I + \|g'\|_{L^1(I)}\big) 
    \end{align*}
    and setting $t = \la^{-\frac{\nn}{\nn+2}}$, we conclude \eqref{eq:osc1}. 
\end{proof}

\subsection{Multivariate phase and amplitude}

The next theorem generalizes \Cref{p:oscillatory} to several variables.

\begin{theorem} \label[t]{t:oscillatory}
    Let $I^n \subseteq \R^n$ be a compact cube.
    Let $\mu^{|\infty} = (\mu_j)_{j\ge 1}$ be an admissible weight and $M_2 = M_0\mu_1 \mu_2\ge b_2>0$.
Assume that $(\mu^{|\infty}_{+2},M_2+\frac{b_2}{2}, \sqrt n \,|I|,\frac{b_2}2)$ is admissible and let 
    \[ 
        \nn = \nn(\mu^{|\infty}_{+2},M_2+\tfrac{b_2}{2},\sqrt n \, |I|,\tfrac{b_2}{2}).
    \]
    Let $f : I^n \to \R$ be a $(\mu^{|\nn+2},M_0)$-smooth function.
    Assume, for some $1 \le i_0 \le n$, we have $\|\p_{i_0}^2f\|_\ell \ge b_2$ for each line segment $\ell$ in $I^n$ with direction $e_{i_0}$. 
    Let $g : I^n \to \C$ belong to $\cC^1(I^n)$. Then, 
    \begin{equation} \label{eq:osc2}
        \Big|\int_{I^n} e^{i\la f(x)} g(x)\, dx\Big| \le B\, \la^{-\frac{1}{\nn+2}},
        \quad \text{ if } \la \ge \Big(\frac{2}{b_2}\Big)^{1 + \frac{2}{\nn}}, 
    \end{equation}
    where 
    \begin{equation}
        B =  \frac{ \Cnn^{1/\nn}\, n\, |I|^n\,\|g\|_{I^n}}{\|\p_{i_0}^2f\|_{I^n}^{1/\nn}}  + 2C\, (\nn+1)\, \big(|I|^{n-1}\,\|g\|_{I^n} +  \|\p_{i_0} g\|_{L^1(I^n)}\big),
    \end{equation}
    and $C>0$ is an absolute constant. 
\end{theorem}

\begin{proof}
    Let us modify the proof of \Cref{p:oscillatory}.
    For simplicity of notation assume $i_0 =1$.
    For $I^n_t := \{x \in I^n : |\p_1^2 f(x)| \le t\}$, \Cref{t:sublevel} and \Cref{r:convexbody} give
    \begin{equation}
    |I^n_t| \le n\,|I|^n \Big(\frac{\Cnn \, t}{\|\p_1^2f \|_{I^n}}\Big)^{1/\nn}, \quad t >0.
    \end{equation}
    Consequently,
    \begin{equation}
    \Big|\int_{I^n_t} e^{i\la f(x)} g(x)\, dx\Big| \le \|g\|_{I^n}\,n\, |I|^n\, \Big(\frac{\Cnn \, t}{\|\p_1^2 f\|_{I^n}}\Big)^{1/\nn}, \quad t >0.
    \end{equation}

    For $0< t \le \frac{b_2}2$, the function $I \ni s \mapsto |\p_1^2 f(s,x')| - t$ has at most $\nn$ zeros, 
    uniformly in $x' \in I^{n-1}$, by \Cref{p:zeros}. 
    So each line with direction $e_1$ intersects $I^n \setminus I^n_t$ in a union of at most $\nn+1$ segments.
    For each such segment, say with endpoints $(a_i,x')$ and $(b_i,x')$ (where $a_i,b_i$ depend on $x'$),
    \Cref{l:Corput} implies, similarly as in the proof of \Cref{p:oscillatory}, that
    \begin{align} \label{eq:osc20}
       \Big| \int_{a_i}^{b_i} e^{i \la f(s,x')} g(s,x') \, ds \Big|
       \le 2C\, (\la t)^{-1/2}\Big(\|g\|_{I^n} + \int_{a_i}^{b_i} |\p_1 g(s,x')| \, ds\Big). 
    \end{align}
    Thus, for  $0< t \le \frac{b_2}2$ and $\la >0$,
    \begin{align}
                                                                         \label{eq:osc21}
        \Big| \int_{I^n \setminus I^n_t} e^{i \la f(x)} g(x) \, ds \Big| &\le \int_{I^{n-1}} \sum_i \Big|\int_{a_i}^{b_i} e^{i \la f(s,x')} g(s,x') \, ds\Big| \, dx'
            \\
                                                                         &\le 2C\, (\nn+1)\, (\la t)^{-1/2}\Big(|I|^{n-1}\,\|g\|_{I^n} +  \|\p_1 g\|_{L^1(I^n)}\Big). 
    \end{align}
    Setting $t = \la^{-\frac{\nn}{\nn+2}}$, we conclude \eqref{eq:osc2}. 
\end{proof}

In the following theorem, we consider oscillatory integrals over sets $K \in  \sK(\R^n)$ 
that are definable in the polynomially bounded o-minimal structure $\R_{\cC_M}$; see \Cref{e:RCM}.
We assume that $\mu^{|\infty} = (\mu_j)_{j\ge 1}$ is a quasianalytic admissible weight satisfying $\mu_j \ge j$ 
as well as all properties listed in \Cref{e:RCM}.

\begin{theorem}
    Let $I^n \subseteq \R^n$ be a compact cube.
    Let $K \in \sK(\R^n)$ be definable in $\R_{\cC_M}$ 
    such that $K \subseteq \on{int}(I^n)$.
    Let $M_2 = M_0\mu_1 \mu_2, b_2,b_2'>0$.
    Let
    \[ 
        \nn = \nn(\mu^{|\infty}_{+2},M_2+\tfrac{b_2}{2},\max\{a_K,|I|\},\min\{\tfrac{b_2}{2},b_2'\}).
    \]
    Let $f : I^n \to \R$ be a $(\mu^{|\infty},M_0)$-smooth function.
    Assume, for some $1 \le i_0 \le n$, we have $\|\p_{i_0}^2f\|_\ell \ge b_2$ 
    for each line segment $\ell$ in $I^n$ with direction $e_{i_0}$ 
    as well as $\|\p_{i_0}^2 f\|_K \ge b_2'$. 
    Let $g : I^n \to \C$ belong to $\cC^1(I^n)$. Then, 
    \begin{equation} \label{eq:osc3}
        \Big|\int_{K} e^{i\la f(x)} g(x)\, dx\Big| \le B\, \la^{-\frac{1}{\nn+2}},\quad \text{ if } \la>0,   
    \end{equation}
    where 
    \begin{equation}
        B =  \frac{\Cnn^{1/\nn}\, |K|\, \|g\|_{K}}{\nu_K \,\|\p_{i_0}^2f\|_{K}^{1/\nn}}  + 2C N_{f,K}\, \big(|I|^{n-1}\,\|g\|_{I^n} +  \|\p_{i_0} g\|_{L^1(I^n)}\big),
    \end{equation}
    $C>0$ is an absolute constant, and $N_{f,K} \in \N$ is an integer depending only on $f$ and $K$. 
\end{theorem}

\begin{proof}
    We follow the proof of \Cref{t:oscillatory}.
    Let $i_0 =1$ and 
    $K_t := \{x \in K : |\p_1^2 f(x)| \le t\}$. Then \Cref{t:sublevel} gives
    \begin{equation}
        \Big|\int_{K_t} e^{i\la f(x)} g(x)\, dx\Big| \le \|g\|_{K}\, \frac{|K|}{\nu_K}\, \Big(\frac{\Cnn \, t}{\|\p_1^2 f\|_{K}}\Big)^{1/\nn}, \quad t >0.
    \end{equation}

    Since $K$ and $f|_K$ are definable in $\R_{\cC_M}$, so is the set
    \[
        \cK := \{(x_1,x',t) \in \R \times \R^{n-1} \times \R : (x_1,x') \in K, \, |\p_1^2 f(x_1,x')| > t\}.
    \]
    Hence there exists $N_{f,K} \in \N$ such that 
    \[
        \cK_{(x',t)} = \{x_1 \in \R : (x_1,x',t) \in \cK\} =  \{x_1 \in \R : (x_1,x') \in K, \, |\p_1^2 f(x_1,x')| > t\}
    \]
    has at most $N_{f,K}$ connected components, for all $(x',t) \in \R^{n-1} \times \R$ (see e.g.\ \cite[4.4]{vandenDriesMiller96}). 

    Let $t>0$ and $x' \in I^{n-1}$. If a connected component of $\cK_{(x',t)}$ has the endpoints $(a_i,x')$ and $(b_i,x')$, 
    then \Cref{l:Corput} implies as in the proof of \Cref{t:oscillatory} that \eqref{eq:osc20} holds. Consequently, for $t,\la>0$,
    \begin{align}
        \Big| \int_{K \setminus K_t} e^{i \la f(x)} g(x) \, ds \Big| 
                                                                     &\le 2C N_{f,K}\, (\la t)^{-1/2}\big(|I|^{n-1}\,\|g\|_{I^n} +  \|\p_1 g\|_{L^1(I^n)}\big). 
    \end{align}
    Setting $t = \la^{-\frac{\nn}{\nn+2}}$, we conclude \eqref{eq:osc3}. 
\end{proof}

\appendix

\section{Examples of admissible weights} \label{sec:A}

In this section, we establish upper and lower bounds for the Bang degree $\dd_{a\, \mu^{|\infty}}(b)$, 
as a function in $a$ and $b$,
for particular weights $\mu^{|\infty}$, namely the analytic weight and the Denjoy weights.
Moreover, we estimate the associated function $\ga_{\tilde \mu}$.

\subsection{The analytic weight}
Let us define $\mu_{\on{an}}^{|\infty} :=(j)_{j\ge 1}$.
Clearly, $\mu_{\on{an}}^{|\infty}$ is an admissible quasianalytic weight which satisfies all 
properties listed in \Cref{e:RCM}. 
In this case, $\Si_{\mu_{\on{an}}^{|\infty}}(1,n)$ is the partial harmonic sum $H_n := \sum_{j=1}^n \frac{1}{j}$,
and, more generally, 
\[
    \Si_{\mu_{\on{an}}^{|\infty}}(m,n) = H_n - H_{m-1}
\] 
if $1 \le m \le n$, where $H_0:= 0$. 

\begin{proposition} \label[p]{p:an}
    Let $a,b>0$ and set $j_0(b) = \lceil \log(b^{-1}) \rceil_{\N}$. 
    If $0 < ae + H_{j_0(b)} \le 1$ (consequently, $j_0(b) =0$), then
    \begin{equation} \label{eq:an1}
        \dd_{a\, \mu_{\on{an}}^{|\infty}}(b) = 0.
    \end{equation}
    If $ae + H_{j_0(b)} > 1$, then
    \begin{equation} \label{eq:an2}
    (j_0(b)+1) \, e^{ae-1} -1 \le  \dd_{a\, \mu_{\on{an}}^{|\infty}}(b) \le (j_0(b)+1)\, e^{ae+1} -1.
    \end{equation}
\end{proposition}

\begin{proof}
    \Cref{l:Bang} implies \eqref{eq:an1}.
    The function $g(t) = \log t$, for $t>0$, is concave so that 
    \[
        g(j+1)-g(j) \le g'(j) \le g(j)-g(j-1),
    \]
    where the left inequality holds for all integers $j\ge 1$, the right for all integers $j\ge 2$. Summing over $j$, we find 
    \begin{equation} \label{eq:Hn}
        \log(n+1) \le H_n \le \log n +1, \quad n \ge 1.
    \end{equation}

Let $x > 1$ and define the positive integer $n(x)$ by
\begin{equation} \label{eq:Comtet}
    H_{n(x)} < x \le H_{n(x)+1}.
\end{equation}
By \eqref{eq:Hn}, we conclude
\begin{equation} \label{eq:nx}
    e^{x-1}-1 \le n(x) < e^x-1.
\end{equation}
By \eqref{eq:Comtet}, we have 
\begin{equation}
    \Si_{\mu_{\on{an}}^{|\infty}}(m,n(x)) < x- H_{m-1} \le \Si_{\mu_{\on{an}}^{|\infty}}(m,n(x)+1).
\end{equation}
Setting $x = ae +H_{m-1}$, we get 
\[
    \Si_{\mu_{\on{an}}^{|\infty}}(m,n) < ae \le \Si_{\mu_{\on{an}}^{|\infty}}(m,n+1)
\]
with $n=n(ae + H_{m-1})$, provided $ae+ H_{m-1}> 1$.
Thus, by \eqref{eq:Hn} and \eqref{eq:nx}, 
\begin{align}
  m\, e^{ae-1} -1  \le  n \le m\, e^{ae+1}-1,
\end{align}
which also holds if $m-1=0$ (by direct verification).
This implies \eqref{eq:an2}. 
\end{proof}

It is easy to compute $\ga_{\tilde \mu_{\on{an}}}$; see \Cref{d:ga}.

\begin{proposition} \label[p]{p:anGa}
    For $\tilde \mu_{\on{an}}(t) = t$,
    we get $\ga_{\tilde \mu_{\on{an}}}  \equiv 1$. \qed
\end{proposition}

\subsection{Denjoy weights}

Let $s \in (0,1]$. Define $\mu_j := j (\log (ej))^s$ for $j\ge 1$.
Then, for each $s \in (0,1]$, $\mu_s^{|\infty} := (\mu_j)_{j \ge 1}$ is a quasianalytic admissible weight satisfying 
all properties listed in \Cref{e:RCM}.
Set
$A^s_n := \sum_{j=1}^n \frac{1}{\mu_j}$, for $n \ge 1$, and $A^s_0 := 0$.

\begin{proposition}[The case $s=1$] \label[p]{p:Denjoy1}
    Let $a,b >0$ and set $j_0(b) = \lceil \log(b^{-1})\rceil_\N$.
    If $0< ae+ A^1_{j_0(b)} \le 1$ (consequently, $j_0(b) =0$), then 
    \begin{equation}\label{eq:D1}
        \dd_{a\, \mu_{1}^{|\infty}}(b) = 0.
    \end{equation}
    If $ae+ A^1_{j_0(b)} > 1$, then 
    \begin{equation} \label{eq:D2}
        (j_0(b)+1)^{e^{ae-1}} e^{e^{ae-1}-1}-1  \le   \dd_{a\, \mu_{1}^{|\infty}}(b) \le (j_0(b)+1)^{e^{ae+1}} e^{e^{ae+1}-1}-1.
    \end{equation}
\end{proposition}

\begin{proof}
    \Cref{l:Bang} implies \eqref{eq:D1}.
    Consider $g(t) = \log \log (et)$, for $t \ge 1$. Then
    \[
        g'(t) = \frac{1}{t \log (et)} \quad \text{ and } \quad  g''(t) = - \frac{1+ \log (et)}{(t \log (et))^2}
    \]
    so that $g$ is concave. Thus
    \[
        g(j+1) - g(j) \le g'(j) \le g(j) - g(j-1), 
    \]
    where the left inequality holds for $j\ge 1$ and the right for $j \ge 2$.
    Summing over $j$, we find 
    \[
        g(n+1) - g(1) \le A^1_n \quad \text{ and } \quad A^1_n -A^1_1 \le g(n)-g(1)   
    \]
    which implies 
    \begin{equation} \label{eq:A1}
        \log\log(e(n+1)) \le  A^1_n 
        \le \log\log (en) + 1, \quad n \ge 1.
    \end{equation}

    For $x > 1$ define the positive integer $n(x)$ by 
    \begin{equation} \label{eq:A1x}
        A^1_{n(x)} < x \le A^1_{n(x)+1}.
    \end{equation}
    By \eqref{eq:A1} and \eqref{eq:A1x},
    \begin{equation} \label{eq:A1n}
        e^{e^{x-1}-1}- 1  \le  n(x) < e^{e^x-1}- 1.
    \end{equation}

    By \eqref{eq:A1x}, we have 
    \begin{equation}
        \Si_{\mu_1^{|\infty}}(m,n(x)) < x- A^1_{m-1} \le \Si_{\mu_{1}^{|\infty}}(m,n(x)+1)
    \end{equation}
    so that, setting $x = ae +A^1_{m-1}$, we get 
    \[
        \Si_{\mu_1^{|\infty}}(m,n) < ae \le \Si_{\mu_1^{|\infty}}(m,n+1)
    \]
    with $n=n(ae + A^1_{m-1})$, provided $ae+ A^1_{m-1}> 1$.
    Thus, by \eqref{eq:A1} and \eqref{eq:A1n},
    \begin{equation}
        m^{e^{ae-1}} e^{e^{ae-1}-1}-1  \le n \le m^{e^{ae+1}} e^{e^{ae+1}-1}-1
    \end{equation}
    which also holds if $m-1=0$ (by direct verification).
    This implies \eqref{eq:D2}
\end{proof}

\begin{proposition}[The case $s \in (0,1)$] \label[p]{p:Denjoys}
    Let $a,b >0$ and set $j_0(b) = \lceil \log(b^{-1})\rceil_\N$. 
    If $ae+ A^s_{j_0(b)}\le 1$ (consequently, $j_0(b) =0$), then 
    \begin{equation} \label{eq:Ds1}
        \dd_{a\, \mu_{s}^{|\infty}}(b) = 0.
    \end{equation}
    If $ae+ A^s_{j_0(b)}> 1$, then 
    \begin{equation} \label{eq:Ds2}
        e^{(f_s(a,b)-(1-s))^{\frac{1}{1-s}}-1}-1 \le   \dd_{a\, \mu_{s}^{|\infty}}(b) 
        \le e^{(f_s(a,b)+(1-s))^{\frac{1}{1-s}}-1}-1
    \end{equation}
    where 
    \[
        f_s(a,b) := (1-s)ae + (1+\log(j_0(b)+1))^{1-s}.
    \]
\end{proposition}

\begin{proof}
    \Cref{l:Bang} implies \eqref{eq:Ds1}.
    Fix $s \in (0,1)$ and consider $g_s(t) = (\log (et))^{1-s}$ for $t \ge 1$. 
    Then
    \begin{align*}
        g_s'(t) = \frac{1-s}{t (\log (et))^{s}}
        \quad \text{ and } \quad
        g_s''(t) = -\frac{(1-s) (\log (et))^{s} + s(1-s) (\log (et))^{s-1}}{(t (\log (et))^{s})^2},
    \end{align*}
    so that $g_s$ is concave. Thus 
    \[
        g_s(j+1) - g_s(j) \le g_s'(j) \le g_s(j) - g_s(j-1), 
    \]
    where the left inequality holds for $j\ge 1$ and the right for $j \ge 2$.
    Summing over $j$, we find
    \[
        g_s(n+1) - g_s(1) \le (1-s) A^s_n  \quad \text{ and } \quad 
        (1-s)(A^s_n - A^s_1) \le g_s(n) -g_s(1)
    \]
    which implies
    \begin{equation} \label{eq:As}
        (\log (e(n+1)))^{1-s} - 1 \le    (1-s) A^s_n \le (\log (en))^{1-s} -s, \quad n \ge 1.
    \end{equation}

    For $x > 1$ define the positive integer $n(x)$ by 
    \begin{equation} \label{eq:Asx}
        A^s_{n(x)} < x \le A^s_{n(x)+1}.
    \end{equation}
    By \eqref{eq:As} and \eqref{eq:Asx},
    \begin{equation} \label{eq:Asn}
        e^{((1-s)x+s)^{\frac{1}{1-s}}-1} -1 \le n(x) < e^{((1-s)x+1)^{\frac{1}{1-s}}-1} -1.
    \end{equation}
    By \eqref{eq:Asx},
    \begin{equation}
        \Si_{\mu_s^{|\infty}}(m,n(x)) < x- A^s_{m-1} \le \Si_{\mu_{s}^{|\infty}}(m,n(x)+1).
    \end{equation}
    Setting $x=ae+A^s_{m-1}$, we get 
    \begin{equation}
        \Si_{\mu_s^{|\infty}}(m,n)<ae \le \Si_{\mu_s^{|\infty}}(m,n+1)
    \end{equation}
    with $n=n(ae+A^s_{m-1})$, provided $ae+A^s_{m-1}>1$.
    Thus, by \eqref{eq:As} and \eqref{eq:Asn},
    \begin{equation}
        e^{((1-s)(ae-1) + (\log(em))^{1-s})^{\frac{1}{1-s}}-1}-1 \le   n \le e^{((1-s)(ae+1) + (\log(em))^{1-s})^{\frac{1}{1-s}}-1}-1
    \end{equation}
    which also holds if $m-1=0$ (by direct verification).
    This implies \eqref{eq:Ds2}
\end{proof}

Let us determine $\ga_{\tilde \mu_s}$.

\begin{proposition} \label[p]{p:DenjoyGa}
    Let $s \in (0,1]$. For $\tilde \mu_s(t) = t (\log (et))^s$, we
    have $\ga_{\tilde \mu_s} \equiv 1+s$.
\end{proposition}

\begin{proof}
    We have 
    \begin{align*}
        \ga_{\tilde \mu_s}(n) &= \sup_{1 \le t \le n} \frac{t\tilde \mu_s'(t)}{\tilde \mu_s(t)} 
        = \sup_{1 \le t \le n} \frac{t ((\log (et))^s + s(\log (et))^{s-1})}{t (\log (et))^s}
        \\
                              &= \sup_{1 \le t \le n} \Big(1 + \frac{s}{\log(et)}\Big) = 1+s 
    \end{align*}
    as claimed.
\end{proof}

\section{Univariate Remez inequality} \label{sec:B}

The goal of this section is to prove the following vector-valued version of \cite[Theorem B]{NazarovSodinVolberg04},
by adjusting the original proof.

\begin{theorem} \label[t]{t:BNSV}
    Let $\mu^{|\infty}=(\mu_j)_{j\ge1}$ be an admissible weight and $M_0>b_0>0$.
    Assume that $(\mu^{|\infty},M_0,1,b_0)$ is admissible and let $\nn = \nn(\mu^{|\infty},M_0,1,b_0)$.
    Let $f: [0,1] \to \R^m$ be $(\mu^{|\nn},M_0)$-smooth such that $\|f\|_{[0,1]}\ge b_0$.
  Then for each interval $I \subseteq [0,1]$ and each Lebesgue measurable set $E \subseteq I$ with $|E|>0$ we have
  \begin{equation} \label{eq:BNSV}
     \|f\|_I \le \Cnn \,\Big(\frac{m\, |I|}{|E|}\Big)^{\nn} \|f\|_E.
  \end{equation}
\end{theorem}

Let us set $\dd := \dd_{\mu^{|\infty}}(\frac{b_0}{M_0})$ so that $\nn = 2\dd$ and note that $\dd \in \N_{\ge 1}$.
See \Cref{d:Cnn} for the definition of $\Cnn$.

Recall that a function $f = (f_1,\ldots,f_m) :[0,1] \to \R^m$ is $(\mu^{|\nn},M_0)$ smooth if $f \in \cC^\nn([0,1])$ and 
\begin{equation}
    \|f\|_{j,[0,1]} = \sup_{x \in [0,1]} \sum_{i=1}^m |f_i^{(j)}(x)| \le M_j, \quad 0 \le j \le \nn.
\end{equation}
We also use $\|f\|_{[0,1]} = \|f\|_{0,[0,1]} = \sup_{x \in [0,1]} \sum_{i=1}^m |f_i(x)|$.

In the following, we write $m_j := M_j/j!$, for $j\ge 1$, and assume  $M_0 = m_0 = 1$.

\begin{proof}
    We may assume that $M_0 = 1$, by replacing $f$ by $f/M_0$ and $b_0$ by $b_0/M_0$.
  Choose points $x_1 < x_2 < \cdots < x_{\nn}$ in $E$ such that $\min_j (x_{j+1}-x_j) \ge \frac{|E|}{\nn-1}$.
  Put $Q(x) := (x-x_1)(x-x_2)\cdots (x-x_{\nn})$. Then, by the Lagrange interpolation formula, for $x \in I$,
  \[
      f_i(x) = \sum_{j=1}^{\nn} \frac{Q(x)f_i(x_j)}{Q'(x_j)(x-x_j)} + \frac{Q(x) f_i^{(\nn)}(\xi)}{\nn!}, \quad \xi = \xi_{i,x} \in I.
  \]
  It follows that
  \[
      \|f\|_I \le \Big(\sum_{j=1}^{\nn} \frac{1}{|Q'(x_j)|}\Big) \|f\|_E |I|^{\nn-1}  + m_{\nn} |I|^{\nn}.
  \]
  Now
  \begin{align*}
     |Q'(x_j)| &= (x_j-x_{j-1})\cdots (x_j -x_1)(x_{j+1}-x_j) \cdots (x_{\nn}-x_j)
     \\
               &\ge \frac{(j-1)!(\nn-j)!}{(\nn-1)^{\nn-1}} |E|^{\nn-1} > \frac{(j-1)!(\nn-j)!}{(\nn-1)!} \Big(\frac{|E|}{e}\Big)^{\nn-1}
  \end{align*}
  so that
  \[
      \sum_{j=1}^{\nn}\frac{1}{|Q'(x_j)|} < \Big(\frac{2e}{|E|}\Big)^{\nn-1}.
  \]
  Thus
  \begin{equation} 
      \|f\|_I \le \Big(\frac{2e|I|}{|E|}\Big)^{\nn-1} \|f\|_E   +  m_{\nn} |I|^{\nn}.
  \end{equation}

  We claim that 
  \begin{equation} \label{eq:NSV1}
        m_{\nn} |I|^{\nn} \le e^{-\nn (\log m+ 3+\ga_{\tilde \mu}(\nn))} \implies m_{\nn} |I|^{\nn} \le \frac{2}{3} \|f\|_I.
  \end{equation}

    If \eqref{eq:NSV1} is satisfied, then we conclude 
  \begin{equation} \label{eq:firstcase}
      \|f\|_I \le  \Big(\frac{2e|I|}{|E|}\Big)^{\nn} \|f\|_E .
  \end{equation}

  Let us suppose that \eqref{eq:NSV1} is not satisfied.
    Choose a subinterval $I_1 \subseteq I$ such that \eqref{eq:NSV1} is satisfied for $I_1$, as well as
    \begin{equation}
    |E \cap I_1| \ge |E|\frac{|I_1|}{|I|},
    \end{equation}
    and 
    \[
        m_\nn |I_1|^\nn \ge 2^{-\nn} e^{-\nn(\log m + 3 + \ga_{\tilde \mu}(\nn))}.
    \]
    Then, using \eqref{eq:NSV1} and \eqref{eq:firstcase} for $I_1$, 
    \begin{align*}
        1 &\le  
          \Big(2e^{\log m + 3 + \ga_{\tilde \mu}(\nn)} \Big)^\nn m_\nn |I_1|^\nn
          \le
      \frac{2}{3}\Big(2e^{\log m + 3 + \ga_{\tilde \mu}(\nn)} \Big)^\nn \|f\|_{I_1}
      \\
          &\le
          \frac{2}{3}\Big(2e^{\log m + 3 + \ga_{\tilde \mu}(\nn)} \Big)^\nn \Big(\frac{2e |I_1|}{|E \cap I_1|}\Big)^\nn \|f\|_{E\cap I_1}
          \le 
          \Big(4e^{\log m + 4 + \ga_{\tilde \mu}(\nn)} \frac{ |I|}{|E|}\Big)^\nn \|f\|_{E}.
    \end{align*}
    This implies \eqref{eq:BNSV} because $\|f\|_I \le \|f\|_{[0,1]} \le M_0 =1$.
\end{proof}

The proof of the claim \eqref{eq:NSV1} comprises the Lemmas \ref{l:beta}---\ref{l:claim3}.

\begin{lemma}[{\cite[Lemma 3.1]{Rainer:2022aa}}] \label[l]{l:beta}
    The function 
    \[
        \be(x) := \sup_{0 \le j \le \nn} \frac{\sum_{i=1}^m |f_i^{(j)}(x)|}{e^j M_j}, \quad  x \in [0,1],
    \]
    is continuous. More precisely, 
    for all $1 \le j \le \nn$ and all distinct $x,y \in [0,1]$,
            \begin{equation} \label{eq:Bang1}
                \be(x) < \max\{\be(y),e^{-j}\} e^{e |x-y| \mu_j}.
            \end{equation}
\end{lemma}

\begin{proof}
    \cite[Lemma 3.1]{Rainer:2022aa} is formulated and proved in the case $m=1$, but the argument works for general $m$.
\end{proof}

\begin{lemma} \label[l]{l:lb}
    $\min_{x \in [0,1]} \be(x) > e^{-\dd -1}$.
\end{lemma}

\begin{proof}
    Suppose 
    $\min_{x \in [0,1]} \be(x) \le e^{-\dd -1}$ and let $x_0 \in [0,1]$ be such that $\be(x_0) = e^{-\dd-1}$. 
    Since $\|f\|_{[0,1]}\ge b_0$, we have $\max_{x \in [0,1]} \be(x) \ge \frac{b_0}{M_0} \ge e^{-j_0}$,
    where $j_0 := \lceil \log \frac{M_0}{b_0}\rceil_\N$.
    Since $\be$ is continuous, there is a monotonic sequence $x_j \in [0,1]$ such that $\be(x_j) = e^{-j}$ 
    for all $j_0 \le j \le \dd+1$.
    By \eqref{eq:Bang1}, we have 
    \begin{equation}
        |x_j-x_{j-1}| > \frac{1}{e\, \mu_j} \quad \text{ for } j_0+1 \le j \le \dd+1,
    \end{equation}
    so that, summing over $j$, we find 
    \begin{equation}
        e > \sum_{j=j_0+1}^{\dd+1}  \frac{1}{\mu_j},
    \end{equation}
    contradicting the definition of $\dd$.
\end{proof}

\begin{lemma} \label[l]{l:NSVint}
    For $0\le k \le \dd$, 
    \[
        \Big( \frac{1}{\nn-k}-\frac{1}{\nn}\Big) \log m_\nn - \frac{1}{\nn-k} \log m_k   <  \ga_{\tilde \mu}(\nn)-1.
    \]
\end{lemma}

\begin{proof}
   The left-hand side equals 
    \begin{align*}
        &\Big( \frac{1}{\nn-k}-\frac{1}{\nn}\Big) \sum_{j=1}^\nn\log(\tfrac{\mu_j}{j}) - \frac{1}{\nn-k} \sum_{j=1}^k\log(\tfrac{\mu_j}{j})    
        \\
        &= \frac{1}{\nn-k} \sum_{j=k+1}^\nn \big(\log(\tfrac{\mu_j}{j})- \log \mu_1\big) 
        - \frac{1}{\nn} \sum_{j=1}^\nn  \big( \log(\tfrac{\mu_j}{j}) - \log \mu_1 \big)    
        \\
        &= \int_1^\nn \frac{\hat \mu'(t)}{\hat \mu(t)} S(t) \, dt,
    \end{align*}
    where $\hat \mu(t) := \tilde \mu(t)/t$ and 
    \[
        S(t) := \frac{1}{\nn-k} \sum_{j=k+1}^\nn \mathbf 1_{[1,j]}(t) 
        - \frac{1}{\nn} \sum_{j=1}^\nn  \mathbf 1_{[1,j]}(t).
    \]
    Since $0 \le S(t) < t/\nn$, we 
    find
    \begin{align*}
        \int_1^\nn \frac{\hat \mu'(t)}{\hat \mu(t)} S(t) \, dt 
        \le \sup_{1 \le t \le \nn} \frac{t \hat \mu'(t)}{\hat \mu(t)} 
        = \sup_{1 \le t \le \nn} \frac{t \tilde  \mu'(t)}{\tilde  \mu(t)}-1  = \ga_{\tilde \mu}(\nn) -1.
    \end{align*}
\end{proof}

\begin{lemma} \label[l]{l:claim1}
    Let $I$ be a subinterval of $[0,1]$.
    If 
    \begin{equation} \label{eq:B2}
        m_{\nn} |I|^{\nn} \le e^{-\nn (\log m+ 3+\ga_{\tilde \mu}(\nn))}
    \end{equation}
    then, for $0 \le k \le \dd$, 
    \begin{equation} \label{eq:NSV2}
        m_{\nn} |I|^{\nn} \le \frac{e^{-2\nn}}{m} m_k |I|^k.
    \end{equation}
\end{lemma}

\begin{proof}
    By \eqref{eq:B2},
    \begin{equation}
        \frac{m_{\nn} |I|^{\nn}}{m_k |I|^k} \le e^{-(\nn-k)(\log m+ 3+\ga_{\tilde \mu}(\nn))}  \frac{m_{\nn}^{k/\nn}}{m_k}.
    \end{equation}
    By \Cref{l:NSVint},
    \begin{align*}
        \Big(\frac{m_{\nn}^{k/\nn}}{m_k}\Big)^{1/(\nn-k)} = \frac{m_\nn^{\frac{1}{\nn-k}-\frac{1}{\nn}}}{m_k^{\frac{1}{\nn-k}}}
        < e^{\ga_{\tilde \mu}(\nn)-1}.
    \end{align*}
    It follows that 
    \begin{align*}
        \frac{m_{\nn} |I|^{\nn}}{m_k |I|^k} &\le e^{-(\nn-k)(\log m+ 3+\ga_{\tilde \mu}(\nn)) + (\nn-k) (\ga_{\tilde \mu}(\nn)-1)} 
        \\
                                            &= e^{- (\nn-k)(\log m + 4)} \le e^{-2\nn- \log m} = \frac{e^{-2\nn}}{m},
    \end{align*}
    if $k \le \dd = \nn/2$.
\end{proof}

\begin{lemma}\label[l]{l:claim2}
    For some $0 \le k \le \dd$, 
    \begin{equation}
        m_k \Big(\frac{|I|}{2}\Big)^k \le m e^{(2+\frac{1}{e})\dd} \Big( \|f\|_I + m_{\nn} \Big(\frac{|I|}{2}\Big)^\nn\Big). 
    \end{equation}
\end{lemma}

\begin{proof}
    Let $c$ be the center of $I$. Then, for $x \in I$, 
    \[
        f_i(x) = T^{\nn-1}_cf_i(x) + \frac{f_i^{(\nn)}(\xi)}{\nn !} (x-c)^\nn,
    \]
    where $T^{\nn-1}_cf_i(x)$ is the Taylor polynomial of degree $\nn-1$ of $f_i$ at $c$ and $\xi_{i,x} \in I$.
    By \cite[6.4.III]{Mandelbrojt52},
    \begin{align*}
        |f_i^{(k)}(c)| = |(T^{\nn-1}_c f_i)^{(k)}(c)| &\le \Big(\frac{2}{|I|}\Big)^k (\nn-1)^{k}\,  \| T^{\nn-1}_cf_i\|_I
\\
                                                      &\le \Big(\frac{2}{|I|}\Big)^k (\nn-1)^{k} \Big( \|f_i\|_I +  m_{\nn} \Big(\frac{|I|}2\Big)^\nn\Big).
    \end{align*}
    Thus 
    \begin{align*}
        \frac{  |f_i^{(k)}(c)|}{e^kM_k} &\le \frac{1}{m_k} \Big(\frac{2}{|I|}\Big)^k \Big( \frac{\nn-1}{e} \Big)^k\frac{1}{k!} \Big( \|f_i\|_I +  m_{\nn} \Big(\frac{|I|}2\Big)^\nn\Big)
        \\
                                        &\le \frac{1}{m_k} \Big(\frac{2}{|I|}\Big)^k e^{\frac{\nn-1}{e}} \Big( \|f_i\|_I +  m_{\nn} \Big(\frac{|I|}2\Big)^\nn\Big).
    \end{align*}
    By \Cref{l:lb}, 
    $e^{-\dd-1} < \be(c)$. 
    Therefore, there is some $0 \le k \le \dd$ such that 
    \[
        e^{-\dd-1} < \sum_{i=1}^m \frac{  |f_i^{(k)}(c)|}{e^kM_k}.
    \]
     This implies the assertion.
\end{proof}

\begin{lemma} \label[l]{l:claim3}
    If \eqref{eq:B2}, then $m_\nn |I|^\nn < \frac{2}3 \|f\|_I$.
\end{lemma}

\begin{proof}
    By \Cref{l:claim1,l:claim2}, there exists $0 \le k \le \dd$ such that 
    \begin{align*}
        m_\nn |I|^\nn \le \frac{e^{-2\nn}}{m} m_k |I|^k &\le e^{k-2\nn}  e^{(2+\frac{1}{e})\dd} \Big( \|f\|_I + m_{\nn} \Big(\frac{|I|}{2}\Big)^\nn\Big)
        \\
                                                        &\le  e^{-(1-\frac{1}{e})\dd} \Big( \|f\|_I + \frac{1}4 m_{\nn} |I|^\nn\Big).
    \end{align*}
    This implies the assertion, 
    since $\frac{3}2 < e^{1-\frac{1}e} - \frac{1}4$.
\end{proof}

\subsection*{Acknowledgement}

This research was funded in whole or in part by the Austrian Science Fund (FWF) DOI 10.55776/PAT1381823.
For open access purposes, the authors have applied a CC BY public copyright license to any author-accepted 
manuscript version arising from this submission.

The author acknowledges the use of ChatGPT, Claude, and Perplexity for assistance with literature review.


\def\cprime{$'$}
\providecommand{\bysame}{\leavevmode\hbox to3em{\hrulefill}\thinspace}
\providecommand{\MR}{\relax\ifhmode\unskip\space\fi MR }
\providecommand{\MRhref}[2]{%
  \href{http://www.ams.org/mathscinet-getitem?mr=#1}{#2}
}
\providecommand{\href}[2]{#2}

\end{document}